\documentclass[twoside]{article}
\pdfoutput=1
\usepackage[a4paper]{geometry}
\usepackage[latin1]{inputenc} % ou \usepackage[utf8]{inputenc}
\usepackage[T1]{fontenc} % ou \usepackage[OT1]{fontenc}
\usepackage{RR}
\usepackage{hyperref}

\usepackage{textcomp}
\usepackage{amsthm}
\usepackage{amsmath}
\usepackage{amssymb}
\usepackage{graphicx}
\usepackage{setspace}
\usepackage{esint}
\usepackage[section,subsection,subsubsection]{extraplaceins}
\makeatletter

\providecommand{\tabularnewline}{\\}

\numberwithin{equation}{section}
\numberwithin{figure}{section}
  \theoremstyle{plain}
  \newtheorem{prop}{\protect\propositionname}
  \theoremstyle{remark}
  \newtheorem{rem}{\protect\remarkname}
\theoremstyle{plain}
\newtheorem{thm}{\protect\theoremname}
  \theoremstyle{plain}
  \newtheorem{lem}{\protect\lemmaname}
  \theoremstyle{plain}
  \newtheorem{cor}{\protect\corollaryname}
  \theoremstyle{remark}
  \newtheorem{claim}{\protect\claimname}
  \theoremstyle{plain}
  \newtheorem*{thm*}{\protect\theoremname}

\makeatother

  \providecommand{\claimname}{Claim}
  \providecommand{\lemmaname}{Lemma}
  \providecommand{\propositionname}{Proposition}
  \providecommand{\remarkname}{Remark}
  \providecommand{\theoremname}{Theorem}
\providecommand{\corollaryname}{Corollary}
\providecommand{\theoremname}{Theorem}

\RRNo{8459}
\RRdate{January 2014}
\RRversion{4}
\RRdater{August 2015}
\RRauthor{% les auteurs
Laurent Baratchart%\thanks[sfn]{EPI APICS}
 \and
Juliette Leblond%\thanksref{sfn}
 \and
Dmitry Ponomarev}%\thanksref{sfn}}
\authorhead{Baratchart \& Leblond \& Ponomarev}
\RRtitle{Optimisation sous contraintes ponctuelles dans des classes de fonctions analytiques}
\RRetitle{Constrained optimization in classes of analytic functions with prescribed pointwise values}
\titlehead{Constrained optimization with pointwise data}
\RRresume{Nous consid\'erons un probl\`eme inverse surd\'etermin\'e pour l'\'equation de Laplace dans un disque, avec des conditions de type Dirichlet-Neumann sur une partie de la fronti\`ere et des contraintes supplémentaires d'interpolation dans le disque. 
Apr\`es reformulation, ce probl\`eme est r\'eduit \`a un probl\`eme de meilleure approximation quadratique sous contraintes, 
par les traces de fonctions holomorphes appartenant à l'espace de Hardy $H^2$, comme dans \cite{Baratchart-Leblond}, et v\'erifiant des conditions d'interpolation  dans le domaine.
De plus, nous effectuons une analyse de la stabilit\'e du probl\`eme face à des perturbations sur les donn\'ees, et proposons une nouvelle  m\'ethode  pour 
calculer certaines caract\'eristiques de la solution (erreur d'approximation, estimation de sa norme), en termes du param\`etre de Lagrange intervenant dans l'algorithme.}

\RRabstract{We consider an overdetermined problem for Laplace equation on a disk with partial boundary data where additional pointwise data inside the disk have to be taken into account. After reformulation, this ill-posed problem reduces to a bounded extremal problem of best norm-constrained approximation of partial $L^2$ boundary data by traces of holomorphic functions which satisfy given pointwise interpolation conditions. The problem of best norm-constrained approximation of a given $L^2$ function on a subset of the circle by the trace of a $H^{2}$ function has been considered in \cite{Baratchart-Leblond}. In the present work, we extend such a formulation to the case where the additional interpolation conditions are imposed. We also obtain some new results that can be applied to the original problem: we carry out stability analysis and propose a novel method of evaluation of the approximation and blow-up rates of the solution in terms of a Lagrange parameter leading to a highly-efficient computational algorithm for solving the problem.} 
\RRmotcle{Probl\`emes inverses avec donn\'ees fronti\`ere, meilleure approximation sous contrainte, fonctions holomorphes, espaces de Hardy, interpolation ponctuelle.}

\RRkeyword{Inverse boundary value problems, best norm-constrained approximation, holomorphic functions, Hardy spaces, pointwise interpolation.}
\RRprojet{APICS}  % cas d'un seul projet
\RCSophia % Sophia Antipolis M\'editerran\'ee

\begin{document}
\makeRR   % cas d'un rapport de recherche
%% \makeRT % cas d'un rapport technique.
%% a partir d'ici, chacun fait comme il le souhaite
\newpage
\tableofcontents
\newpage

\section{\label{sec:intro}Introduction}

Many stationary physical problems are formulated in terms of reconstruction
of function in a planar domain from partially available measurements
on its boundary. 
Such problems may arise from 3-dimensional settings for which symmetry properties allow reformulation of
the model in dimension 2, as it is often the case in Maxwell's equations for electro- or magnetostatics.
It is a classical problem to find the function in
the domain given values of the function itself (Dirichlet problem)
or its normal derivative (Neumann problem) on the whole boundary.
However, if the knowledge of Dirichlet or Neumann data is available
only on a part of the boundary, the recovery problem is underdetermined and
one needs to impose both conditions on the measurements on the accessible
part of the boundary. One example of such a problem is recovering an electrostatic potential $u$ satisfying conductivity
equation in $\Omega\subset\mathbb{R}^{2}$ from knowledge of its values
and those of the normal current $\sigma\partial_{n}u$ on a subset $\Gamma$
of the boundary $\partial\Omega$ for a given conductivity coefficient
$\sigma$ in $\bar{\Omega}$:
\[ %\begin{equation}
\begin{cases}
\nabla\cdot \left(\sigma\nabla u\right)=0 & \hspace{1em}\text{in}\quad\Omega,\\
u=u_{0},\quad\sigma\partial_{n}u=w_{0} & \hspace{1em}\text{on}\quad\Gamma.
\end{cases} %\label{eq:u_probl}
\] %\end{equation}
Often, only the behavior of the solution on $\partial\Omega\backslash\Gamma$ is of practical interest or, in case of a free boundary problem, one aims to find a position of this part of the boundary by imposing an additional condition there \cite{Alessandrini}.
%, in which case (\ref{eq:u_probl}) can be viewed as an inverse problem.

In the present work, we consider the prototypical case where the domain is the unit
disk $\Omega=\mathbb{D}$, the conductivity coefficient is constant
$\sigma\equiv1$, and we assume appropriate regularity of the boundary data on $\Gamma \subset \mathbb{T}$ required for the existence of a unique weak $W^{1, 2}\left(\Omega\right)$ solution:
\begin{equation}
\begin{cases}
\Delta  u=0 & \hspace{1em}\text{in}\quad\Omega,\\
u=u_{0},\quad \partial_{n}u=w_{0} & \hspace{1em}\text{on}\quad\Gamma\hspace{1em}\text{with}\quad u_{0}\in W^{1/2, 2}\left(\Gamma\right),\hspace{0.5em} w_{0}\in L^{2}\left(\Gamma\right).
\end{cases}\label{eq:u_probl}
\end{equation}
Despite these simplifications, we note that since the Laplace operator is
invariant under conformal transformations \cite{Ablowitz}, %in the case of constant conductivity coefficient, 
results obtained for the unit disk can be readily extended to more general simply connected
domains with smooth boundary. This is also true for non-constant conductivities $\sigma$  since the same conformal invariance
holds \cite{Rigat-Russ}.
%Also, in the case $\Omega=\mathbb{D}$, $\sigma\neq\text{const}$,
For non-constant conductivity equation,
such a problem without additional pointwise data has been considered within the
 framework of generalized analytic functions \cite{Fischer-Leblond-Partington} and beyond simply connected domains. In particular, the case of annular domain finds its application in a problem of recovery of plasma boundary in tokamaks \cite{Fischer}. 

The problem (\ref{eq:u_probl}) happens to be ill-posed
\cite{Lavrentiev} as might be anticipated if one recalls the celebrated Hadamard's example which demonstrates lack of continuous dependence of solution on the boundary data if the boundary conditions assumed only on a part of the boundary. In the present formulation, as we are going to see, the both boundary conditions on $\Gamma$ cannot be completely arbitrary functions as well: certain compatibility is needed in order for the solution
to exist and be finite on $\partial\Omega\backslash\Gamma$.
Therefore, one would like to find the admissible solution (that is
bounded or, even more, sufficiently close to some \textit{a priori}
known data on $\partial\Omega\backslash\Gamma$) which is in best
agreement with the given data $u_{0}$, $w_{0}$ on $\Gamma$. Put this way, the
recovery issue is approximately recast as a well-posed constrained optimization problem, as we will show further.

In our approach, we use complex analytic tools to devise solution
of the problem. Recall that if a function $g=u+iv$ is analytic (holomorphic),
then $u$ and $v$ are real-valued harmonic functions satisfying the Cauchy-Riemann
equations $\partial_{n}u=\partial_{t}v$,
$\partial_{t}u=-\partial_{n}v$, where the partial derivatives are taken with respect to polar coordinates  \cite{Nehari}. 
Applied to problem (\ref{eq:u_probl}), the first equation suggests that knowing $w_{0}$,
one can, up to an additive constant, recover $v$ on $\Gamma$, and therefore both
$u_{0}$ and $w_{0}$ define the trace on $\Gamma$ of the analytic
function $g$ in $\Omega$. However, the knowledge of an analytic
function on a subset $\Gamma\subset\mathbb{T}$ of a positive measure
defines the analytic function in the whole unit disk $\mathbb{D}$ \cite{Goluzin,Patil}
and also its values on the unit circle $\mathbb{T}$ \cite{Hoffman}. Of course, available data $u_{0}$, $w_{0}$ on $\Gamma$
may not be compatible to yield the restriction of an analytic function
onto $\Gamma$, and such instability phenomenon illustrates ill-posedness of
the problem from the viewpoint of complex analysis.

As already mentioned, (\ref{eq:u_probl}) %for $\Omega=\mathbb{D}$, $\sigma\equiv1$ 
may be recast as a well-posed bounded extremal
problem in normed Hardy spaces of holomorphic functions defined by their boundary values.
Different aspects of such problems were extensively considered
in \cite{Alpay,Baratchart-Leblond,Baratchart-Leblond-Partington}
and an algorithm for computation of the solution was proposed.
In the present work, we would like to solve such an optimization problem
incorporating additional available information from inside the domain.
Taking advantage of the complex analytic approach allowing to make sense of pointwise values (unlike
in Lebesgue spaces), we would like to extend previously obtained results
to a situation where the solution needs to meet prescribed values at some points inside
the disk. We characterize the solution in a way suitable for
further practical implementation, obtain estimates of approximation
rate and discrepancy growth, investigate the question of stability,
illustrate numerically certain technical aspects,
discuss the choice of auxiliary parameters and, based on a newly developed method of obtaining estimates on solution (which also applies to the problem without pointwise constraints), propose an improvement of the computational algorithm for solving the problem.

The paper is organized as follows. Section \ref{sec:intro_Hardy}
provides an introduction to the theory of Hardy spaces which are essential
functional spaces in the present approach. In Section \ref{sec:probl_sol},
we formulate the problem, prove existence of a unique solution and
give its useful characterization. Section \ref{sec:psi_choice} discusses
the choice of interpolation function which is a technical tool to
prescribe desired values inside the domain; we also provide an alternative
form of the solution that turns out to be useful later. In Section
\ref{sec:comp_issues}, we obtain specific balance relations governing
approximation rate on a given subset of the circle and discrepancy
on its complement, which shed light on the quality of the solution
depending on a choice of some auxiliary parameters. Also, at this point we introduce
a novel series expansion method of evaluation of quantities governing solution quality.
Section \ref{sec:compan_pb} introduces a closely related problem whose solution might be computationally
cheaper in certain cases. We further look into sensitivity of the solution to perturbations
of all input data in Section \ref{sec:stab_res} raising the stability
issue and providing technical estimates. We conclude with Section
\ref{sec:num_res} by presenting numerical illustrations of certain properties
of the solution, a short discussion of the choice of technical
parameters and suggestion of a new efficient computational algorithm based on
the results of the Section \ref{sec:comp_issues}. Some concluding remarks are given in Section \ref{sec:conclu}.

\newpage
\section{\label{sec:intro_Hardy}Background in theory of Hardy spaces}
%$H^{p}\left(\mathbb{D}\right)$}

Let $\mathbb{D}$ be the open unit disk in $\mathbb{C}$ with %the
boundary $\mathbb{T}$.

Hardy spaces $H^{p}\left(\mathbb{D}\right)$ can be defined as classes
of holomorphic functions on the disk with finite norms
\[
\left\Vert F\right\Vert _{H^{p}}=\sup_{0<r<1}\left(\dfrac{1}{2\pi}\int_{0}^{2\pi}\left|F\left(re^{i\theta}\right)\right|^{p}d\theta\right)^{1/p},\hspace{1em}1\leq p<\infty,
\]
\[
\left\Vert F\right\Vert _{H^{\infty}}=\sup_{\left|z\right|<1}\left|F\left(z\right)\right|.
\]
These are Banach spaces that enjoy plenty of interesting properties
and they have been studied in detail over the years \cite{Duren,Garnett,Hoffman,Rudin}.
In this section we give a brief introduction into the topic, yet trying
to be as much self-contained as possible, adapting general material
to our particular needs.

The key property of functions in Hardy spaces is their behavior on the boundary $\mathbb{T}$ of the disk. More precisely,
boundary values of functions belonging to the Hardy space $H^{p}$
are well-defined in the $L^{p}$ sense
\begin{equation}
\lim_{r\nearrow1}\left\Vert F\left(re^{i\theta}\right)-F\left(e^{i\theta}\right)\right\Vert _{L^{p}\left(\mathbb{T}\right)}=0,\hspace{1em}1\leq p<\infty,\label{eq:rad_limit1}
\end{equation}
as well as pointwise, for almost every $\theta \in [0, 2 \pi]$:
\begin{equation}
\lim_{r\nearrow1}F\left(re^{i\theta}\right)=F\left(e^{i\theta}\right).\label{eq:rad_limit2}
\end{equation}
It is the content of the Fatou's theorem (see, for instance, \cite{Hoffman})
that the latter limit exists almost everywhere not only radially but
also along any non-tangential path. Thanks to the Parseval's identity,
the proof of (\ref{eq:rad_limit1}) is especially simple when $p=2$
(see \cite[Th. 1.1.10]{Martinez}), the case that we will work
with presently.

Given a boundary function $f\in L^{p}\left(\mathbb{T}\right)$,
$1\leq p\leq\infty$ whose Fourier coefficients of negative index
vanish 
\begin{equation}
f_{-n}:=\frac{1}{2\pi}\int_{0}^{2\pi}f\left(e^{i\theta}\right)e^{in\theta}d\theta=0,\hspace{1em}n=1,\,2,\,\dots,\label{eq:neg_Fourier_coeff}
\end{equation}
(in this case, we say $f\in H^{p}\left(\mathbb{T}\right)$), there
exists $F\in H^{p}\left(\mathbb{D}\right)$ such that $F\left(re^{i\theta}\right)\rightarrow f\left(e^{i\theta}\right)$
in $L^{p}$ as $r\nearrow1$, and it is defined by the Poisson representation
formula, for $re^{i\theta} \in \mathbb{D}$:
\begin{equation}
F\left(re^{i\theta}\right)=\frac{1}{2\pi}\int_{0}^{2\pi}f\left(e^{it}\right)P_{r}\left(\theta-t\right)dt,\label{eq:Poisson_repr}
\end{equation}
where we employed the Poisson kernel of $\mathbb{D}$:
\[
P_{r}\left(\theta\right):=\frac{1-r^{2}}{1-2r\cos\theta+r^{2}}=\sum_{k=-\infty}^{\infty}r^{\left|k\right|}e^{ik\theta} \, , \ 0 < r < 1 \, , \ \ \theta \in [0, 2 \pi] \, .
\]
Note that the vanishing  condition for the Fourier coefficients of negative order
is  equivalent to the requirement of the Poisson integral (\ref{eq:Poisson_repr})
to be analytic in $\mathbb{D}$. Indeed, since $f\left(e^{i\theta}\right)=\underset{n=-\infty}{\overset{\infty}{\sum}}f_{n}e^{in\theta}$,
the right-hand side of (\ref{eq:Poisson_repr}) reads 
\begin{eqnarray*}
\frac{1}{2\pi}\int_{0}^{2\pi}f\left(e^{it}\right)P_{r}\left(\theta-t\right)dt & = & \frac{1}{2\pi}\sum_{k=-\infty}^{\infty}r^{\left|k\right|}e^{ik\theta}\sum_{n=-\infty}^{\infty}f_{n}\int_{0}^{2\pi}e^{i\left(n-k\right)t}dt=\sum_{n=-\infty}^{\infty}f_{n}r^{\left|n\right|}e^{in\theta}\\
 & = & f_{0}+\sum_{n=1}^{\infty}\left(f_{n}z^{n}+f_{-n}\bar{z}^{n}\right),
\end{eqnarray*}
and hence, if we want this to define a holomorphic function through (\ref{eq:Poisson_repr}), we have
to impose condition (\ref{eq:neg_Fourier_coeff}).

Because of the established isomorphism, we can identify the space
$H^{p}=H^{p}\left(\mathbb{D}\right)$ with $H^{p}\left(\mathbb{T}\right)\subset L^{p}\left(\mathbb{T}\right)$
for $p\geq1$ (the case $p=1$ requires more sophisticated reasoning
invoking F. \& M. Riesz theorem \cite{Hoffman}). It follows that
$H^{p}$ is a Banach space (as a closed subspace
of $L^{p}\left(\mathbb{T}\right)$ which is complete), and we have
inclusions due to properties of Lebesgue spaces on bounded domains
\begin{equation}
H^{\infty}\subseteq H^{s}\subseteq H^{p},\hspace{1em}s\geq p\geq1.\label{eq:Hp_incl}
\end{equation}

Summing up, we can abuse notation employing only one letter $f$,
and write 
\begin{equation}
\left\Vert f\right\Vert _{H^{p}}=\left\Vert f\right\Vert _{L^{p}\left(\mathbb{T}\right)}\label{eq:Hp_Lp}
\end{equation}
whenever $f\in L^{p}\left(\mathbb{T}\right)$, $p\geq1$, satisfies
(\ref{eq:neg_Fourier_coeff}).

Moreover, in case $p=2$, which we will focus on, the Parseval's identity
provides an isometry between the Hardy space $H^{2}=H^{2}\left(\mathbb{D}\right)$ and the space $l_{2}\left(\mathbb{N}_{0}\right)$ of
square-summable sequences \footnote{Here and onwards, we stick to the convention: $\mathbb{N}_{0}:=\left\{ 0, 1, 2, \dots\right\} $, $\mathbb{N}_{+}:=\left\{ 1, 2, 3, \dots\right\} $.}. 
%:=\left\{ \left(c_{k}\right)_{k=0}^{\infty}:\,\sum_{k=0}^{\infty}\left|c_{k}\right|^{2}<\infty\right\} $.
Hence, $H^{2}$ is a Hilbert space with the
inner product 
\begin{equation}
\left\langle f,g\right\rangle _{L^{2}\left(\mathbb{T}\right)}=\frac{1}{2\pi}\int_{0}^{2\pi}f\left(e^{i\theta}\right) \, \overline{g\left(e^{i\theta}\right)} \, d\theta=\sum_{k=0}^{\infty}f_{k}\bar{g}_{k}.\label{eq:Hp_innerprod}
\end{equation}

We will also repeatedly make use of the fact that $H^{\infty}$ functions
act as multipliers in $H^{p}$, that is, $H^{\infty}\cdot H^{p}\subset H^{p}$.\\
There is another useful property of Hardy classes to perform
factorization: if $f\in H^{p}$ and $f\left(z_{j}\right)=0$,
$z_{j}\in\mathbb{D}$, $j=1,\dots,N$, then $f=bg$ with $g\in H^{p}$
and the finite Blaschke product $b\in H^{\infty}$
defined as 
\begin{equation}
b\left(z\right)=e^{i\phi_{0}}\prod_{j=1}^{N}\left(\dfrac{z-z_{j}}{1-\bar{z}_{j}z}\right)\label{eq:Blaschke_prod}
\end{equation}
for some constant $\phi_{0}\in\left[0,2\pi\right]$. Possibility of such
factorization comes from the observation that each factor of $b\left(z\right)$
is analytic in $\mathbb{D}$ and automorphic since 
\[
\left|z\right|^{2}+\left|z_{j}\right|^{2}-\left|z\right|^{2}\left|z_{j}\right|^{2}=\left|z\right|^{2}\left(1-\left|z_{j}\right|^{2}/2\right)+\left|z_{j}\right|^{2}\left(1-\left|z\right|^{2}/2\right)\leq1,
\]
and thus 
\[
\left|\dfrac{z-z_{j}}{1-\bar{z}_{j}z}\right|^{2}=\frac{1-2\text{Re}\left(\bar{z}_{j}z\right)+\left|z\right|^{2}+\left|z_{j}\right|^{2}-1}{1-2\text{Re}\left(\bar{z}_{j}z\right)+\left|z\right|^{2}\left|z_{j}\right|^{2}}\leq1.
\]
Additionally, this shows that
\begin{equation}
\left|b\right|\equiv1,\hspace{1em}z\in\mathbb{T},\label{eq:Blaschke_circle}
\end{equation}
and hence $\left\Vert b\right\Vert _{H^{\infty}}=1$.

%Having identified $H^{2}$ with $H^{2}\left(\mathbb{T}\right)$,
We let $\bar{H}_{0}^{2}$ denote the orthogonal complement of $H^{2}$ in $L^{2}\left(\mathbb{T}\right)$,
so that $L^{2}=H^{2}\oplus\bar{H}_{0}^{2}$. Recalling characterization (\ref{eq:neg_Fourier_coeff})
of $H^{2}$ functions, we can view
$\bar{H}_{0}^{2}$ as the space of functions whose expansions have
non-vanishing Fourier coefficients of only negative index, and hence
it characterizes $L^2\left(\mathbb{T}\right)$ functions which are holomorphic in $\mathbb{C}\backslash\bar{\mathbb{D}}$
and decay to zero at infinity.\\
Similarly, we can introduce the orthogonal complement to $bH^{2}$
in $L^{2}\left(\mathbb{T}\right)$ with $b$ as in (\ref{eq:Blaschke_prod}) so that $L^{2}=bH^{2}\oplus\left(bH^{2}\right)^{\perp}$
which in its turn decomposes into a direct sum as $\left(bH^{2}\right)^{\perp}=\bar{H}_{0}^{2}\oplus\left(bH^{2}\right)^{\perp_{H^{2}}}$
with $\left(bH^{2}\right)^{\perp_{H^{2}}}\subset H^{2}$ denoting
the orthogonal complement to $bH^{2}$ in $H^{2}$; it is not empty if $b\not\equiv\text{const}$, 
whence the proper inclusion $bH^{2}\subset H^{2}$ holds. Moreover, making use of the Cauchy integral formula,
it is not difficult to show that 
\[
\left(bH^{2}\right)^{\perp_{H^{2}}}:=\left(bH^{2}\right)^{\perp}\ominus\bar{H}_{0}^{2}=\dfrac{P_{N-1}\left(z\right)}{\prod_{j=1}^{N}\left(1-\bar{z}_{j}z\right)} \, ,
\]
where $P_{N-1}\left(z\right)$ is the space of polynomials of degree
at most $N-1$ in $z$.\\

Given $J\subset\mathbb{T}$, let us introduce the Toeplitz operator $\phi$ with symbol $\chi_{J}$ (the indicator function of $J$), defined by:
\begin{eqnarray}
H^{2} & \rightarrow & H^{2}\nonumber \\
F & \mapsto & \phi\left(F\right)=P_{+}\left(\chi_{J}F\right),\label{eq:Toepl_op}
\end{eqnarray}
where we let $P_{+}$ denote
the orthogonal projection from $L^{2}\left(\mathbb{T}\right)$ onto
$H^{2}$ (that might be realized by setting
Fourier coefficients of negative index to zero or convolving the function
with the Cauchy kernel). Similarly, $P_{-}:=I-P_{+}$ defines the
orthogonal projection onto $\bar{H}_{0}^{2}$.\\
% Further, we may also
% abuse notation of the operator $\phi$ extending its domain to $L^{2}\left(\mathbb{T}\right)$
% and even $L^{2}\left(J\right)$.
We also notice that the map $L^{2}\left(\mathbb{T}\right)\rightarrow bH^{2}:\, F\mapsto bP_{+}\left(\bar{b}F\right)$
is the orthogonal projection onto $bH^{2}$. Indeed, taking into account
(\ref{eq:Blaschke_circle}), for any $u\in L^{2}\left(\mathbb{T}\right)$,
$v\in H^{2}$,
\[
\left\langle u-bP_{+}\left(\bar{b}u\right),bv\right\rangle _{L^{2}\left(\mathbb{T}\right)}=\left\langle u,bv\right\rangle _{L^{2}\left(\mathbb{T}\right)}-\left\langle P_{+}\left(\bar{b}u\right),\bar{b}bv\right\rangle _{L^{2}\left(\mathbb{T}\right)}=0.
\]

Any function in $H^{p}$, $p\geq1$,
being analytic and sufficiently regular on $\mathbb{T}$, admits integral representation in terms of its boundary
values and thus is uniquely determined by means of the Cauchy formula.
However, it is also possible to recover a function $f$ holomorphic
in $\mathbb{D}$ from its values on a subset of the boundary $I\subset\mathbb{T}$
using so-called Carleman's formulas \cite{Aizenberg,Goluzin}. Write
$\mathbb{T}=I\cup J$ with $I$ and $J$ being Lebesgue measurable
sets.
\begin{prop}
\label{prop:Carleman}Assume $\left|I\right|>0$ and let $\Phi\in H^{\infty}$
be any function such that $\left|\Phi\right|>1$ in $\mathbb{D}$
and $\left|\Phi\right|=1$ on $J$. Then, $f\in H^{p}$, $p\geq1$
can be represented from $\left.f\right|_{I}$ as 
\begin{equation}
f\left(z\right)=\frac{1}{2\pi i}\lim_{\alpha\rightarrow\infty}\int_{I}\frac{f\left(\xi\right)}{\xi-z}\left[\frac{\Phi\left(\xi\right)}{\Phi\left(z\right)}\right]^{\alpha}d\xi,\label{eq:Carleman}
\end{equation}
where the convergence is uniform on compact subsets of $\mathbb{D}$.
\end{prop}
\begin{proof}
Since $\Phi\in H^{\infty}$ and $f\in H^{p}\subseteq H^{1}$,
it is clear that $f\left(z\right)\left[\Phi\left(z\right)\right]^{\alpha}\in H^{1}$,
and so the Cauchy formula applies to $f\left(z\right)\left[\Phi\left(z\right)\right]^{\alpha}=f\left(z\right)\exp\left[\alpha\log\Phi\left(z\right)\right]$
for any $\alpha>0$ 
\[
f\left(z\right)\left[\Phi\left(z\right)\right]^{\alpha}=\frac{1}{2\pi i}\int_{\mathbb{T}}\frac{f\left(\xi\right)\left[\Phi\left(\xi\right)\right]^{\alpha}}{\xi-z}d\xi
\]
\[
\Rightarrow\hspace{1em}f\left(z\right)=\frac{1}{2\pi i}\left(\int_{I}+\int_{J}\right)\frac{f\left(\xi\right)}{\xi-z}\left[\frac{\Phi\left(\xi\right)}{\Phi\left(z\right)}\right]^{\alpha}d\xi.
\]
Since the second integral vanishes in absolute value as $\alpha\nearrow\infty$
for any $z\in\mathbb{D}$ (by the choice of $\Phi$), we have (\ref{eq:Carleman}).\end{proof}

The integral representation (\ref{eq:Carleman})
implies the following uniqueness result (see also e.g. \cite[Th. 17.18]{Rudin}, for a different argument based on the factorization which shows that $\log \left|f\right|\in L^{1}\left(\mathbb{T}\right)$ whenever $f\in H^{p}$).
\begin{cor}
\label{cor:uniqueness_by_subset} %ex \begin{rem} \label{rem:uniqueness_by_subset} : replace rem by corollary
Functions in $H^{1}$  are uniquely
determined  by their boundary values on $I \subset \mathbb{T}$ as soon as $|I|>0$.
\end{cor}
It follows that if two $H^{p}$ functions %analytic in $\mathbb{D}$ 
agree on
a subset of $\mathbb{T}$ with non-zero Lebesgue measure, then they
must coincide everywhere in $\mathbb{D}$.
This complements the identity
theorem for holomorphic functions \cite{Ablowitz} claiming that zero
set of an analytic function cannot have an accumulation point inside
the domain of analyticity which particularly implies that two functions
coinciding in a neighbourhood of a point of analyticity are necessarily
equal in the whole domain of analyticity. 
\begin{rem}
\label{rem:H02bar}
Using the isometry $H^2 \to \bar{H}_{0}^{2}$:
\[
f(z) \mapsto \frac{1}{z} \, \overline{f \left(\frac{1}{\bar{z}}\right)} \, , \ z \in \mathbb{D} 
\]
(which is clear from the Fourier expansion on the boundary), we check that
Proposition \ref{prop:Carleman} and Corollary \ref{cor:uniqueness_by_subset} 
also apply to functions in $\bar{H}_{0}^{2}$.
\end{rem}
%$\hspace{1em}$
\begin{rem}
\label{rem:quench_choice}The auxiliary function $\Phi$ termed as
``quenching'' function can be chosen as follows. Let $u$ be a Poisson
integral of a positive function vanishing on $J$ (for instance, the
characteristic function $\chi_{I}$) and $v$ its harmonic conjugate
that can be recovered (up to an additive constant) at $z=re^{i\theta}$,
$r<1$ by convolving $u$ on $\mathbb{T}$ (using normalized Lebesgue
measure $d\sigma=\dfrac{1}{2\pi}d \theta$) with the conjugate Poisson kernel
$\text{Im}\left(\dfrac{1+re^{it}}{1-re^{it}}\right)$, $t\in\left[0,2\pi\right]$,
see \cite{Hoffman} for details. Then, clearly, $\Phi=\exp\left(u+iv\right)$
is analytic in $\mathbb{D}$ and satisfies the required conditions.
More precisely, combining recovered $v$ with the Poisson representation
formula for $u$, we conclude that convolution of boundary values
of $u$ with the Schwarz kernel $\dfrac{1+re^{it}}{1-re^{it}}$, $t\in\left[0,2\pi\right]$
defines (up to an additive constant) the analytic function $u\left(z\right)+iv\left(z\right)$
for $z=re^{i\theta}\in\mathbb{D}$. An explicit quenching function  constructed in such a way will be given in Section \ref{sec:probl_sol}
by (\ref{eq:quench_func}).
\end{rem}
%$\quad$
\begin{rem}
\label{rem:Patil_ref}A similar result was also obtained and discussed
in \cite{Patil}, see also \cite{Aizenberg, Baratchart-Leblond-Partington, Krein-Nudel'man}.
\end{rem}
% Since Proposition \ref{prop:Carleman} applies particularly not only
% to functions in $H^{2}$, but also in $\bar{H}_{0}^{2}$,
% as a consequence of Remark \ref{rem:uniqueness_by_subset}, we derive
As a consequence of Remark \ref{rem:H02bar}, we derive %:uniqueness_by_subset
a useful tool in form of
\begin{prop}
\label{prop:phi_inject}The Toeplitz operator $\phi$ is an injection
on $H^{2}$.\end{prop}
\begin{proof}
% We need to show that $\text{Ker}\left(\phi\right)=0$, that is, for
% arbitrary $g\in H^{2}$, $P_{+}\left(\chi_{J}g\right)=0$
% implies $g\equiv0$. 
By the orthogonal decomposition $L^{2}=H^{2}\oplus\bar{H}_{0}^{2}$,
we have $\chi_{J}g=P_{+}\left(\chi_{J}g\right)+P_{-}\left(\chi_{J}g\right)$.
Now, if $P_{+}\left(\chi_{J}g\right)=0$, then $\chi_{J}g$ is a $\bar{H}_{0}^{2}$
function vanishing on $I$ and hence, by Remark \ref{rem:H02bar},
must be identically zero.
\end{proof}
The last result for Hardy spaces that we are going to employ is the
density of traces \cite{Baratchart-Leblond,Baratchart-Leblond-Partington}. 
\begin{prop}
\label{prop:trace_res1}Let $J\subset\mathbb{T}$ be a subset of non-full
measure, that is $\left|I\right|=\left|\mathbb{T}\backslash J\right|>0$.
Then, the restriction $\left.H^{p}\right|_{J}:=\left.\left(\text{tr}H^{p}\right)\right|_{J}$
is dense in $L^{p}\left(J\right)$, $1\leq p<\infty$. \end{prop}
\begin{proof}
In the particular case $p=2$ (other values of $p$ are treated in \cite{Baratchart-Leblond}), %the claim is easy to
we prove the claim by contradiction. 
Assume that there is non-zero 
$f\in L^{2}\left(J\right)$
orthogonal to $\left.H^{2}\right|_{J}$, then, extending it by zero
on $I$, we denote the extended function as $\tilde{f}$. 
We thus have $\left\langle \tilde{f},g\right\rangle _{L^{2}\left(\mathbb{T}\right)}=0$
for all $g\in H^{2}$ which implies $ \tilde{f}\in\bar{H}_{0}^{2}$
and hence, by Remark \ref{rem:H02bar}, $f\equiv0$.

% ATTENTION: Notation f_I \vee f_J has not been introduced yet!

%
% Define $T\left(z\right):=1/z$ that maps conformally interior of the
% unit disk into its exterior. Then the composition $F:=f\circ T$ is
% holomorphic in $\mathbb{D}$ (due to decay to zero of $f$ at infinity)
% and $F\in H^{2}$. Now, since $T$ leaves the
% boundary $\mathbb{T}$ invariant and $f=0$ on $I$, we have $F=0$
% on some $\tilde{I}\subset\mathbb{T}$ with $\bigl|\tilde{I}\bigr|=\bigl|I\bigr|>0$,
% and hence, by Remark \ref{rem:uniqueness_by_subset}, $F\equiv0$
% in $\mathbb{D}$ implying $f\equiv0$ which contradicts the assumption.
% 
\end{proof}
\begin{rem}
\label{rem:density_orth_compl}
From the proof and Remark \ref{rem:H02bar}, we see that the
same density result holds if one replaces $H^{2}$ with $\bar{H}_{0}^{2}$.
\end{rem}
There is a counterpart of Propositon \ref{prop:trace_res1} that also
characterizes boundary traces of $H^{p}$ spaces.
\begin{prop}
\label{prop:trace_res2}Assume $\left|I\right|>0$, $f\in L^{p}\left(I\right)$,
\textup{$1\leq p\leq\infty$. Let $\left\{ g_{n}\right\} _{n=1}^{\infty}$
be a sequence of $H^{p}$ functions such that $\underset{n\rightarrow\infty}{\lim}\left\Vert f-g_{n}\right\Vert _{L^{p}\left(I\right)}=0$.
Then, $\left\Vert g_{n}\right\Vert _{L^{p}\left(J\right)}\rightarrow\infty$
as $n\rightarrow\infty$ unless $f$ is the trace of a $H^{p}$ function.}\end{prop}
\begin{proof}
%For simplicity, 
Consider the case $1<p<\infty$; for the cases $p=1$
and $p=\infty$ we refer to \cite{Baratchart-Leblond} and \cite{Baratchart-Leblond-Partington},
respectively. We argue by contradiction: assume that $f$ is not the
trace on $I$ of some $H^{p}$ function, but $\underset{n\rightarrow\infty}{\lim}\left\Vert g_{n}\right\Vert _{L^{p}\left(J\right)}<\infty$.
Then, by hypothesis, the sequence $\left\{ g_{n}\right\} _{n=1}^{\infty}$
is bounded not only in $L^{p}\left(J\right)$ but also in $H^{p}$.
Since $H^{p}$ is reflexive (as any $L^{p}\left(\mathbb{T}\right)$
is for $1<p<\infty$), it follows from the Banach-Alaoglu theorem
(or see \cite[Ch. 10, Th. 7]{Lax}) that the closed unit ball in $H^{p}$
is weakly compact, therefore, we can extract a subsequence $\left\{ g_{n_{k}}\right\} $
that converges weakly in $H^{p}$: $g_{n_{k}}\rightharpoonup g$ for
some $g\in H^{p}$. However, since $g_{n}\rightarrow f$ in $L^{p}\left(I\right)$,
we must have $f=\left.g\right|_{I}$, a contradiction. \end{proof}
\begin{rem}
When $\left|J\right|=0$, the existence of a $H^{p}$ sequence $\left\{ g_{n}\right\} _{n=1}^{\infty}$
approximating $f\in L^{p}\left(I\right)$ in $L^{p}\left(I\right)$ norm, means that $f$  actually belongs to 
$H^{p}$  (which is a closed subspace of $L^{p}\left(\mathbb{T}\right) = L^{p}\left(I\right)$).
\end{rem}
\vfill

\newpage
\section{\label{sec:probl_sol}An extremal problem and its solution}

We consider the problem of finding a $H^{2}$
function which takes prescribed values $\left\{ \omega_{j}\right\} _{j=1}^{N}\in\mathbb{C}$
at interior points $\left\{ z_{j}\right\} _{j=1}^{N}\in\mathbb{D}$
which approximates best a given $L^{2}\left(I\right)$ function on a subset of
the boundary $I\subset\mathbb{T}$ while remaining close enough to
another $L^{2}\left(J\right)$ function on the complementary part $J\subset\mathbb{T}$. 

We proceed with a technical formulation of this problem. Assuming
given interpolation values at distinct interior points $\left\{ z_{j}\right\} _{j=1}^{N}\in\mathbb{D}$,
we let $\psi\in H^{2}$ be some fixed function satisfying the interpolation conditions 
\begin{equation}
\psi\left(z_{j}\right)=\omega_{j}\in\mathbb{C},\quad j=1,\dots,N.\label{eq:meas_cond}
\end{equation}
Then, any interpolating function in $H^{2}$
fulfiling these conditions can be written as $\tilde{g}=\psi+bg$
for arbitrary $g\in H^{2}$ with $b\in H^{\infty}$
the finite Blaschke product defined in (\ref{eq:Blaschke_prod}).

As before, let $\mathbb{T}=I\cup J$ with both $I$ and $J$ being
of non-zero Lebesgue measure. For the sake of simplicity, we write $f=\left.f\right|_{I}\vee\left.f\right|_{J}$ to mean
a function defined on the whole $\mathbb{T}$ through its values given
on $I$ and $J$.

For $h\in L^{2}\left(J\right)$, $M\geq0$, let us introduce the following
functional spaces

\begin{equation}
\mathcal{A}^{\psi,b}:=\left\{ \tilde{g}\in H^{2}:\,\tilde{g}=\psi+bg,\, g\in H^{2}\right\} ,\label{eq:A_space}
\end{equation}
\begin{equation}
\mathcal{B}_{M,h}^{\psi,b}:=\left\{ g\in H^{2}:\,\left\Vert \psi+bg-h\right\Vert _{L^{2}\left(J\right)}\leq M\right\} ,\label{eq:B_space}
\end{equation}
\begin{equation}
\mathcal{C}_{M,h}^{\psi,b}:=\left\{ f\in L^{2}\left(I\right):\, f=\left.\psi\right|_{I}+b\left.g\right|_{I},\, g\in\mathcal{B}_{M,h}^{\psi,b}\right\} .\label{eq:C_space}
\end{equation}
We then have inclusions $\mathcal{C}_{M,h}^{\psi,b}\subseteq\left.\mathcal{A}^{\psi,b}\right|_{I}\subseteq\left.H^{2}\right|_{I}\subset L^{2}\left(I\right)$
and $\mathcal{C}_{M,h}^{\psi,b}=\left.\left(\psi+b\mathcal{B}_{M,h}^{\psi,b}\right)\right|_{I}\neq\emptyset$
since $\mathcal{B}_{M,h}^{\psi,b}\neq\emptyset$ for any given $h\in L^{2}\left(J\right)$
and $M>0$ as follows from Proposition \ref{prop:trace_res1}.

Now the framework is set to allow us to pose the problem in precise
terms. 

Given $f\in L^{2}\left(I\right)$, our goal will be to find a solution
to the following bounded extremal problem 
\begin{equation}
\min_{g\in\mathcal{B}_{M,h}^{\psi,b}}\left\Vert \psi+bg-f\right\Vert _{L^{2}\left(I\right)}.\label{eq:BEP_main}
\end{equation}

As it was briefly mentioned at the beginning, the motivation for such a formulation is to look for 
\begin{equation}
\tilde{g}_{0}:=\psi+bg_{0}\in\mathcal{A}^{\psi,b}\hspace{1em}\text{such that}\hspace{1em}g_{0}=\text{arg}\min_{g\in\mathcal{B}_{M,h}^{\psi,b}}\bigl\Vert\underset{=\tilde{g}}{\underbrace{\psi+bg}}-f\bigr\Vert_{L^{2}\left(I\right)},\label{eq:g0_sol_def}
\end{equation}
i.e. the best $H^{2}$-approximant to $f$ on $I$ which fulfils interpolation conditions (\ref{eq:meas_cond}) and is not too far from the reference $h$
on $J$: $\left\Vert \tilde{g}_{0}-h\right\Vert _{L^{2}\left(J\right)}\leq M$.
In view of Proposition \ref{prop:trace_res2}, the $L^{2}$-constraint on $J$ is crucial whenever $f\notin\left.\mathcal{A}^{\psi,b}\right|_{I}$
(which is always the case when known data are recovered from physical
measurements necessarily subject to noise). In other words, we assume that
\begin{equation}
\left.g\right|_{I}\neq\bar{b}\left(f-\psi\right),\label{eq:g_cond}
\end{equation}
i.e. there is no $\tilde{g}=\psi+bg\in H^{2}$
whose trace on $I$ is exactly the given function $f\in L^{2}\left(I\right)$,
and at the same time remains within the $L^{2}$-distance $M$ from $h$ on $J$. This motivates
the choice (\ref{eq:B_space}) for the space of admissible solutions $\mathcal{B}_{M,h}^{\psi,b}$.
%% which incorporates closedness to some \textit{a priori} known function $h$ on $J$.%%

Existence and uniqueness of solution to (\ref{eq:BEP_main})
can be reduced to what has been proved in a general setting in \cite{Baratchart-Leblond}.
Here we present a slightly different proof.
\begin{thm}
\label{thm:existence}For any $f\in L^{2}\left(I\right)$, $h\in L^{2}\left(J\right),$
$\psi\in H^{2}$, $M\geq0$ and $b\in H^{\infty}$
defined as (\ref{eq:Blaschke_prod}), there exists a unique solution
to the bounded extremal problem (\ref{eq:BEP_main}).\end{thm}
\begin{proof}
By the existence of a best approximation projection onto a non-empty
closed convex subset of a Hilbert space (see, for instance, \cite[Th. 3.10.2]{Debnath}), it is required to show that the space of restrictions
$\left.\mathcal{B}_{M,h}^{\psi,b}\right|_{I}$ is a closed convex
subset of $L^{2}\left(I\right)$. Convexity is a direct consequence
of the triangle inequality: 
\[
\left\Vert \alpha\left(bg_{1}+\psi-h\right)+\left(1-\alpha\right)\left(bg_{2}+\psi-h\right)\right\Vert _{L^{2}\left(J\right)}\leq\alpha M+\left(1-\alpha\right)M=M
\]
for any $g_{1}$, $g_{2}\in\mathcal{B}_{M,h}^{\psi,b}$ and $\alpha\in\left[0,1\right]$.\\

We will now show the closedness property. Let $\left\{ g_{n}\right\} _{n=1}^{\infty}$
be a sequence of $\mathcal{B}_{M,h}^{\psi,b}$ functions which converges
in $L^{2}\left(I\right)$ to some $g$: $\left\Vert g-g_{n}\right\Vert _{L^{2}\left(I\right)}\rightarrow0$
as $n\rightarrow\infty$. We need to prove that $g\in\mathcal{B}_{M,h}^{\psi,b}$.

%By dual characterization of $L^{2}\left(I\right)$ norm and density
%of $\left.\bar{H}_{0}^{2}\right|_{I}$ in $L^{2}\left(I\right)$ (Proposition
%\ref{prop:trace_res1} and Remark \ref{rem:density_orth_compl}),
%we have
%\begin{eqnarray}
%\left\Vert g-g_{n}\right\Vert _{L^{2}\left(I\right)} & = & \sup_{\substack{\xi\in L^{2}\left(I\right)\\
%\left\Vert \xi\right\Vert _{L^{2}\left(I\right)}\leq1
%}
%}\left|\left\langle g-g_{n},\xi\right\rangle _{L^{2}\left(I\right)}\right|=\sup_{\substack{\xi\in\bar{H}_{0}^{2}\\
%\left\Vert \xi\right\Vert _{L^{2}\left(I\right)}\leq1
%}
%}\left|\left\langle g-g_{n},\xi\right\rangle _{L^{2}\left(I\right)}\right|.\label{norm_dual_I}
%\end{eqnarray}

%We note that $g\in\left.H^{2}\right|_{I}$, since otherwise, by Proposition
%\ref{prop:trace_res2}, $\left\Vert g_{n}\right\Vert _{L^{2}\left(J\right)}\rightarrow\infty$
%as $n\rightarrow\infty$ which would contradict the fact that $g_{n}\in\mathcal{B}_{M,h}^{\psi,b}$
%starting with some $n$. Hence $\left(g-g_{n}\right)\in H^{2}$
%and, in particular, $\left\langle g-g_{n},\xi\right\rangle _{L^{2}\left(\mathbb{T}\right)}=0$
%for any $\xi\in\bar{H}_{0}^{2}$, which implies that

We note that $g\in\left.H^{2}\right|_{I}$, since otherwise, by Proposition
\ref{prop:trace_res2}, $\left\Vert g_{n}\right\Vert _{L^{2}\left(J\right)}\rightarrow\infty$
as $n\rightarrow\infty$, which would contradict the fact that $g_{n}\in\mathcal{B}_{M,h}^{\psi,b}$
starting with some $n$. Therefore, $\psi+bg\in H^2$ and $\left\langle \psi+bg,\xi\right\rangle _{L^{2}\left(\mathbb{T}\right)}=0$
for any $\xi\in\bar{H}_{0}^{2}$, which implies that

\[
\left\langle \psi+bg,\xi\right\rangle _{L^{2}\left(I\right)}=\left\langle \left(\psi+bg\right)\vee0,\xi\right\rangle _{L^{2}\left(\mathbb{T}\right)}=-\left\langle 0\vee\left(\psi+bg\right),\xi\right\rangle _{L^{2}\left(\mathbb{T}\right)}=-\left\langle \psi+bg,\xi\right\rangle _{L^{2}\left(J\right)}.
\]

From here, using the same identity for $\psi+bg_{n}$, we obtain
\begin{eqnarray*}
\left\langle \psi+bg-h,\xi\right\rangle _{L^{2}\left(J\right)} & = & -\left\langle \psi+bg,\xi\right\rangle _{L^{2}\left(I\right)}-\left\langle h,\xi\right\rangle _{L^{2}\left(J\right)}=-\lim_{n\rightarrow\infty}\left\langle \psi+bg_{n},\xi\right\rangle _{L^{2}\left(I\right)}-\left\langle h,\xi\right\rangle _{L^{2}\left(J\right)}\\
 & = & \lim_{n\rightarrow\infty}\left\langle \psi+bg_{n},\xi\right\rangle _{L^{2}\left(J\right)}-\left\langle h,\xi\right\rangle _{L^{2}\left(J\right)}.
\end{eqnarray*}

Since $g_{n}\in\mathcal{B}_{M,h}^{\psi,b}$ for all $n$, the Cauchy-Schwarz inequality gives
\begin{eqnarray*}
\left|\left\langle \psi+bg-h,\xi\right\rangle _{L^{2}\left(J\right)}\right|=\lim_{n\rightarrow\infty}\left|\left\langle \psi+bg_{n}-h,\xi\right\rangle _{L^{2}\left(J\right)}\right|\leq M\left\Vert \xi\right\Vert _{L^{2}\left(J\right)}
\end{eqnarray*}
for any $\left.\xi\in\bar{H}_{0}^{2}\right|_{J}$.
The final result is now furnished by employing density of $\left.\bar{H}_{0}^{2}\right|_{J}$ in $L^2\left(J\right)$ (Proposition
\ref{prop:trace_res1} and Remark \ref{rem:density_orth_compl}) and the dual characterization of $L^2\left(J\right)$ norm:

\[
\left\Vert \psi+bg-h\right\Vert _{L^{2}\left(J\right)}=\sup_{\substack{\xi\in L^{2}\left(J\right)\\
\left\Vert \xi\right\Vert _{L^{2}\left(J\right)}\leq1
}
}\left|\left\langle \psi+bg-h,\xi\right\rangle _{L^{2}\left(J\right)}\right|=\sup_{\substack{\xi\in\bar{H}_{0}^{2}\\
\left\Vert \xi\right\Vert _{L^{2}\left(J\right)}\leq1
}
}\left|\left\langle \psi+bg-h,\xi\right\rangle _{L^{2}\left(J\right)}\right|\leq M.
\]

%and therefore, using (\ref{norm_dual_I}), we obtain
%\begin{equation*}
%\left|\left\langle g-g_{n},\xi\right\rangle _{L^{2}\left(J\right)}\right|\leq\sup_{\substack{\xi\in \bar{H}^{2}_{0}\\
%\left\Vert \xi\right\Vert _{L^{2}\left(I\right)}\leq1}}\left|\left\langle g-g_{n},\xi\right\rangle _{L^{2}\left(I\right)}\right|=\left\Vert g-g_{n}\right\Vert _{L^{2}\left(I\right).
%}
%\end{equation*}

%Therefore
%\[
%\lim_{n\rightarrow\infty}\left|\left\langle g-g_{n},\xi\right\rangle _{L^{2}\left(J\right)}\right|\leq\lim_{n\rightarrow\infty}\left\Vert g-g_{n}\right\Vert _{L^{2}\left(I\right)}\left\Vert \xi\right\Vert _{L^{2}\left(I\right)}=0
%\]

%for arbitrary $\xi\in\bar{H}_{0}^{2}$ and thus, by density
%of $\left.\bar{H}_{0}^{2}\right|_{J}$ in $L^{2}\left(J\right)$ (Proposition \ref{prop:trace_res1} and Remark \ref{rem:density_orth_compl}), also for arbitrary $\xi\in L^{2}\left(J\right)$. Combined with completeness of $L^2\left(J\right)$, this entails the convergence $\left\|g-g_{n}\right\|_{L^{2}\left(J\right)}\rightarrow 0$ (since $\underset{n\rightarrow\infty}{\lim}\left(g-g_{n}\right)\in L^2\left(J\right) \cap \left[L^2\left(J\right)\right]^{\perp}$), which furnishes the required bound $\left\Vert \psi+bg-h\right\Vert _{L^{2}\left(J\right)}\leq M$
%by passing to the limit as $n\rightarrow\infty$ in 
%\[
%\left\Vert \psi+bg-h\right\Vert _{L^{2}\left(J\right)}\leq\left\Vert \psi+bg_{n}-h\right\Vert _{L^{2}\left(J\right)}+\left\Vert g-g_{n}\right\Vert _{L^{2}\left(J\right)},
%\]
%employing the triangle inequality and the Blaschke product property
%(\ref{eq:Blaschke_circle}).
\end{proof}

A key property of the solution is that the constraint in (\ref{eq:B_space})
is necessarily saturated unless $f\in\left.\mathcal{A}^{\psi,b}\right|_{I}$.
\begin{lem}
\label{lem:saturation}If $f\notin\left.\mathcal{A}^{\psi,b}\right|_{I}$
and $g\in\mathcal{B}_{M,h}^{\psi,b}$ solves (\ref{eq:BEP_main}),
then $\left\Vert \psi+bg-h\right\Vert _{L^{2}\left(J\right)}=M$.\end{lem}
\begin{proof}
To show this, suppose the opposite, i.e. there is $g_{0}\in H^{2}$
solving (\ref{eq:BEP_main}) for which we have 
\[
\left\Vert \psi+bg_{0}-h\right\Vert _{L^{2}\left(J\right)}<M.
\]
The last condition means that $g_{0}$ is in interior of $\mathcal{B}_{M,h}^{\psi,b}$,
and hence we can define \mbox{$g^{\star}:=g_{0}+\epsilon\delta_{g}\in\mathcal{B}_{M,h}^{\psi,b}$}
for sufficiently small $\epsilon>0$ and $\delta_{g}\in H^{2}$,
$\left\Vert \delta_{g}\right\Vert _{H^{2}}=1$ such that \mbox{$\text{Re}\left\langle b\delta_{g},\psi+bg_{0}-f\right\rangle _{L^{2}\left(I\right)}<0$,}
where the equality case is eliminated by (\ref{eq:g_cond}). By the
smallness of $\epsilon$, the quadratic term is negligible, and thus
we have
\begin{eqnarray*}
\left\Vert \psi+bg^{\star}-f\right\Vert _{L^{2}\left(I\right)}^{2}&=&\left\Vert \psi+bg_{0}-f\right\Vert _{L^{2}\left(I\right)}^{2}+2\epsilon\text{Re}\left\langle b\delta_{g},\psi+bg_{0}-f\right\rangle _{L^{2}\left(I\right)}+\epsilon^{2}\left\Vert \delta_{g}\right\Vert _{L^{2}\left(I\right)}^{2}\\
 &<&\left\Vert \psi+bg_{0}-f\right\Vert _{L^{2}\left(I\right)}^{2},
\end{eqnarray*}
which contradicts the minimality of $g_{0}$.
\end{proof}
As an immediate consequence of saturation of the constraint, we obtain
\begin{cor}
\label{cor:M_pos}The requirement $f\in L^{2}\left(I\right)\backslash\left.\mathcal{A}^{\psi,b}\right|_{I}$
implies that the formulation of the problem should be restricted to
the case $M>0$.\end{cor}
\begin{proof}
If $f\in L^{2}\left(I\right)\backslash\left.\mathcal{A}^{\psi,b}\right|_{I}$
and $M=0$, the Lemma entails that $h\in\left.\mathcal{A}^{\psi,b}\right|_{J}$.
Then, $h=\psi+bg$ for some $g \in H^2$ and its extension to the whole $\mathbb{D}$ (given, for
instance, by Proposition \ref{prop:Carleman}) uniquely determines
$\tilde{g}=h$ without resorting to solution of the bounded
extremal problem (\ref{eq:BEP_main}), hence independently of $f$.
\end{proof}
Having established that equality holds in (\ref{eq:B_space}), we
approach (\ref{eq:BEP_main}) as a %an equality 
constrained optimization
problem following a standard idea of Lagrange multipliers (e.g. \cite{Varaia})
and claim that for a solution $g$ to (\ref{eq:BEP_main}) and for some $\lambda\in\mathbb{R}$,
we must necessarily have
\begin{equation}
\left\langle \delta_{\tilde{g}},\left(\tilde{g}-f\right)\vee\lambda\left(\tilde{g}-h\right)\right\rangle _{L^{2}\left(\mathbb{T}\right)}=0\label{eq:orthog_cond}
\end{equation}
for any $\delta_{\tilde{g}}\in bH^{2}$ (recall that $\tilde{g}=\psi+bg$
and $\delta_{\tilde{g}}=b\delta_{g}$ for $\delta_{g}\in H^{2}$)
which is a condition of tangency of level lines of the minimizing
objective functional and the constraint functional. The condition
(\ref{eq:orthog_cond}) can be shown by the same variational argument
as in the proof of Lemma \ref{lem:saturation}, it must hold true,
otherwise we would be able to improve the minimum while still remaining
in the admissible set. This motivates us to search for $g\in H^{2}$
such that, for $\lambda\in\mathbb{R}$, 
\begin{equation}
\left[\left(\psi+bg-f\right)\vee\lambda\left(\psi+bg-h\right)\right]\in\left(bH^{2}\right)^{\perp}\label{eq:Lagr_cond}
\end{equation}
which is equivalent to 
\begin{equation}
P_{+}\left[\bar{b}\left(\psi+bg-f\right)\vee\lambda\bar{b}\left(\psi+bg-h\right)\right]=0.\label{eq:Lagr_cond_fin}
\end{equation}

\begin{thm}
\label{thm:charact}
If $f\notin\left.\mathcal{A}^{\psi,b}\right|_{I}$, the solution to
the bounded extremal problem (\ref{eq:BEP_main}) is given by
\begin{equation}
g_{0}=\left(1+\mu\phi\right)^{-1}P_{+}\left(\bar{b}\left(f-\psi\right)\vee\left(1+\mu\right)\bar{b}\left(h-\psi\right)\right),\label{eq:g0_sol_fin}
\end{equation}
where the parameter $\mu>-1$ is uniquely chosen such that $\left\Vert \psi+bg_{0}-h\right\Vert _{L^{2}\left(J\right)}=M$. 
\end{thm}
The proof of Theorem \ref{thm:charact} goes in three steps.

\subsection{Solution for the case $h=0$}

For simplicity, we first assume $h=0$. Then, the equation (\ref{eq:Lagr_cond_fin})
can be elaborated as follows
\[
P_{+}\left(\bar{b}\left(\psi+bg\right)\right)+\left(\lambda-1\right)P_{+}\left(0\vee\bar{b}\left(\psi+bg\right)\right)=P_{+}\left(\bar{b}f\vee0\right),
\]
\[
g+P_{+}\left(\bar{b}\psi\right)+\left(\lambda-1\right)P_{+}\left(0\vee\bar{b}\psi\right)+\left(\lambda-1\right)\phi g=P_{+}\left(\bar{b}f\vee0\right),
\]
\begin{equation}
\left(1+\mu\phi\right)g=-P_{+}\left(\bar{b}\left(\psi-f\right)\vee\left(1+\mu\right)\bar{b}\psi\right),\label{eq:g_pre_sol_h_0}
\end{equation}
where we introduced the parameter $\mu:=\lambda-1\in\mathbb{R}$.\\
The Toeplitz operator $\phi$, defined as (\ref{eq:Toepl_op}), is
self-adjoint and, as it can be shown (see the Hartman-Wintner theorem
in Appendix), its spectrum is 
\begin{equation}
\sigma\left(\phi\right)=\left[\text{ess inf }\chi_{J},\,\text{ess sup }\chi_{J}\right]=\left[0,1\right],\label{eq:spectrum}
\end{equation}
hence $\left\| \phi \right\| \leq 1$ and the operator $\left(1+\mu\phi\right)$ is invertible on
$H^{2}$ for $\mu>-1$ allowing to claim that
\begin{equation}
g=-\left(1+\mu\phi\right)^{-1}P_{+}\left(\bar{b}\left(\psi-f\right)\vee\left(1+\mu\right)\bar{b}\psi\right).\label{eq:g_sol_h_0}
\end{equation}
This generalizes the result of \cite{Alpay} to the case when solution
needs to meet pointwise interpolation conditions.

\subsection{Solution for the case $h\neq0,$ $h\in\left.H^{2}\right|_{J}$}

Now, let $h\neq0$, but assume it to be the restriction to $J$ of
some $H^{2}$ function.\\
We write $f=\varrho+\left.\kappa\right|_{I}$ for $\kappa\in H^{2}$
such that $\left.\kappa\right|_{J}=h$. Then, the solution to (\ref{eq:BEP_main})
is
\[
g_{0}=\arg\underset{g\in\mathcal{B}_{M,h}^{\psi,b}}{\text{min}}\left\Vert \psi+bg-f\right\Vert _{L^{2}\left(I\right)}=\arg\underset{g\in\mathcal{\tilde{B}}_{M,0}}{\text{min}}\left\Vert \tilde{\psi}+bg-\varrho\right\Vert _{L^{2}\left(I\right)},
\]
where $\tilde{\psi}:=\psi-\kappa$ and 
\[
\mathcal{\tilde{B}}_{M,0}:=\left\{ g\in H^{2}:\,\left\Vert \tilde{\psi}+bg\right\Vert _{L^{2}\left(J\right)}\leq M\right\} .
\]
It is easy to see that, due to $\left.\kappa\right|_{J}=h$, we have
$\mathcal{\tilde{B}}_{M,0}=\mathcal{B}_{M,h}^{\psi,b}$. Therefore,
the already obtained results (\ref{eq:g_pre_sol_h_0}), (\ref{eq:g_sol_h_0})
apply to yield 
\begin{eqnarray}
\left(1+\mu\phi\right)g_{0} & = & -P_{+}\left(\bar{b}\left(\tilde{\psi}-\varrho\right)\vee\left(1+\mu\right)\bar{b}\tilde{\psi}\right)\nonumber \\
 & = & -P_{+}\left(\bar{b}\left(\psi-\kappa-\varrho\right)\vee\left(1+\mu\right)\bar{b}\left(\psi-\kappa\right)\right)\nonumber \\
 & = & P_{+}\left(\bar{b}\left(f-\psi\right)\vee\left(1+\mu\right)\bar{b}\left(h-\psi\right)\right),\label{eq:g0_pre_sol_fin}
\end{eqnarray}
from where (\ref{eq:g0_sol_fin}) follows.

\subsection{Solution for the case $h\neq0,$ $h\in L^{2}\left(J\right)$}

Here we assume $h\notin\left.H^{2}\right|_{J}$ but only $h\in L^{2}\left(J\right)$.
The result follows from the previous step by density of $\left.H^{2}\right|_{J}$
in $L^{2}\left(J\right)$ along the line of reasoning similar to \cite{Baratchart-Leblond}. 

More precisely, by density (Proposition \ref{prop:trace_res1}), for
a given $h\in L^{2}\left(J\right)$, we have existence of a sequence
$\left\{ h_{n}\right\} _{n=1}^{\infty}\subset\left.H^{2}\right|_{J}$
such that $h_{n}\underset{n\rightarrow\infty}{\rightarrow}h$ in $L^{2}\left(J\right)$.
This generates a sequence of solutions 
\begin{equation}
g_{n}=\arg\underset{g\in\mathcal{B}_{M,h_{n}}}{\text{min}}\left\Vert \psi+bg-f\right\Vert _{L^{2}\left(I\right)},\hspace{1em}n\in\mathbb{N}_{+},\label{eq:g_n_def}
\end{equation}
satisfying 
\begin{equation}
\left(1+\mu_{n}\phi\right)g_{n}=P_{+}\left(\bar{b}\left(f-\psi\right)\vee\left(1+\mu_{n}\right)\bar{b}\left(h_{n}-\psi\right)\right)\label{eq:g_n_presol}
\end{equation}
for $\mu_{n}>-1$ chosen such that $\left\Vert \psi+bg_{n}-h_{n}\right\Vert _{L^{2}\left(J\right)}=M$.\\
Since $\left\{ g_{n}\right\} _{n=1}^{\infty}$ is bounded in $H^{2}$
(by definition of the solution space $\mathcal{B}_{M,h_{n}}^{\psi,b}$), and
due to the Hilbertian setting, up to extraction of a subsequence,
it converges weakly in $L^{2}\left(\mathbb{T}\right)$ norm to some
element in $H^{2}$
\begin{equation}
g_{n}\underset{n\rightarrow\infty}{\rightharpoonup}\gamma\in H^{2}.\label{eq:gn_weak_conv}
\end{equation}
We will first show that $\mu_{n}\rightarrow\mu$ as $n\rightarrow\infty$.
Then, since all $\left(1+\mu\phi\right)$ and $\left(1+\mu_{n}\phi\right)$
are self-adjoint, we have, for any $\xi\in H^{2}$,
\[
\left\langle \left(1+\mu_{n}\phi\right)g_{n},\xi\right\rangle _{L^{2}\left(\mathbb{T}\right)}=\left\langle g_{n},\left(1+\mu_{n}\phi\right)\xi\right\rangle _{L^{2}\left(\mathbb{T}\right)}\underset{n\rightarrow\infty}{\rightarrow}\left\langle \gamma,\left(1+\mu\phi\right)\xi\right\rangle _{L^{2}\left(\mathbb{T}\right)}=\left\langle \left(1+\mu\phi\right)\gamma,\xi\right\rangle _{L^{2}\left(\mathbb{T}\right)},
\]
and thus $\left(1+\mu_{n}\phi\right)g_{n}\underset{n\rightarrow\infty}{\rightharpoonup}\left(1+\mu\phi\right)\gamma$.
Combining this with the convergence 
\[
P_{+}\left(\bar{b}\left(f-\psi\right)\vee\left(1+\mu_{n}\right)\bar{b}\left(h_{n}-\psi\right)\right)\underset{n\rightarrow\infty}{\rightarrow}P_{+}\left(\bar{b}\left(f-\psi\right)\vee\left(1+\mu\right)\bar{b}\left(h-\psi\right)\right)
\]
in $L^{2}\left(\mathbb{T}\right)$, equation (\ref{eq:g_n_presol})
suggests that the weak limit $\gamma$ in (\ref{eq:gn_weak_conv})
is a solution to (\ref{eq:BEP_main}). It will remain to check that
$\gamma\in\mathcal{B}_{M,h}^{\psi,b}$ and is indeed a minimizer of
the cost functional (\ref{eq:BEP_main}).
\begin{claim}
For $\mu_{n}$ in (\ref{eq:g_n_presol}), we have
\end{claim}
\begin{equation}
\underset{n\rightarrow\infty}{\lim}\mu_{n}=:\mu\in\left(-1,\infty\right).\label{eq:lim_mu}
\end{equation}

\begin{proof}
We prove this statement by contradiction. 
Because of the relation (\ref{eq:Lagr_cond_fin}), for any $\xi\in H^{2}$,
we have
\begin{equation}
\left\langle \bar{b}\left(f-\psi\right)-g_{n},\xi\right\rangle _{L^{2}\left(I\right)}=\left(1+\mu_{n}\right)\left\langle g_{n}-\bar{b}\left(h_{n}-\psi\right),\xi\right\rangle _{L^{2}\left(J\right)}.\label{eq:Lagr_cond_IJ_gn}
\end{equation}
We note that the weak convergence (\ref{eq:gn_weak_conv}) in $H^{2}$
implies the weak convergence $g_{n}\rightharpoonup\gamma$
in $L^{2}\left(J\right)$ as $n\rightarrow\infty$ since for a given $\eta\in L^{2}\left(J\right)$,
we can take $\xi=P_{+}\left(0\vee\eta\right)\in H^{2}$
in the definition $\underset{n\rightarrow\infty}{\lim}\left\langle g_{n},\xi\right\rangle _{L^{2}\left(\mathbb{T}\right)}=\left\langle \gamma,\xi\right\rangle _{L^{2}\left(\mathbb{T}\right)}$. 

Assume first that $\mu_{n}\underset{n\rightarrow\infty}{\rightarrow}\infty$.
Then, since the left-hand side of (\ref{eq:Lagr_cond_IJ_gn}) remains
bounded as $n\rightarrow\infty$, we necessarily must have 
\[
\underset{n\rightarrow\infty}{\lim}\left\langle g_{n}-\bar{b}\left(h_{n}-\psi\right),\xi\right\rangle _{L^{2}\left(J\right)}=0.
\]
Since $h_{n}\rightarrow h$ in $L^{2}\left(J\right)$ strongly, this
implies that $\gamma=\bar{b}\left(h-\psi\right)\in\left.H^{2}\right|_{J}$
contrary to the initial assumption of the section that $h\notin\left.H^{2}\right|_{J}$.\\
Next, we consider another possibility, namely that the limit $\underset{n\rightarrow\infty}{\lim}\mu_{n}$
does not exist. Then, there are at least two infinite sequences $\left\{ n_{k_{1}}\right\} $,
$\left\{ n_{k_{2}}\right\} $ such that
\[
\lim_{k_{1}\rightarrow\infty}\mu_{n_{k_{1}}}=:\mu^{\left(1\right)}\neq\mu^{\left(2\right)}:=\lim_{k_{2}\rightarrow\infty}\mu_{n_{k_{2}}}.
\]
Since the left-hand side of (\ref{eq:Lagr_cond_IJ_gn}) is independent
of $\mu_{n}$ and both limits $\mu^{\left(1\right)}$, $\mu^{\left(2\right)}$
exist and finite, we have
\[
\lim_{k_{1}\rightarrow\infty}\left(1+\mu_{n_{k_{1}}}\right)\left\langle g_{n_{k_{1}}}-\bar{b}\left(h_{n_{k_{1}}}-\psi\right),\xi\right\rangle _{L^{2}\left(J\right)}=\lim_{k_{2}\rightarrow\infty}\left(1+\mu_{n_{k_{2}}}\right)\left\langle g_{n_{k_{2}}}-\bar{b}\left(h_{n_{k_{2}}}-\psi\right),\xi\right\rangle _{L^{2}\left(J\right)}
\]
\[
\Rightarrow\hspace{1em}\left(\mu^{\left(1\right)}-\mu^{\left(2\right)}\right)\left\langle \gamma-\bar{b}\left(h-\psi\right),\xi\right\rangle _{L^{2}\left(J\right)}=0.
\]
As before, because of $h\notin\left.H^{2}\right|_{J}$, we derive
a contradiction $\mu^{\left(1\right)}=\mu^{\left(2\right)}$.\\
Now that the limit in (\ref{eq:lim_mu}) exists, we have $\mu\geq-1$.
To show $\mu>-1$, assume, by contradiction, that $\mu=-1$. 
Since $g_{n}\in\mathcal{B}_{M,h_{n}}$, for any $\xi\in H^{2}$, the Cauchy-Schwarz inequality
gives 
\[
\text{Re}\left\langle \psi+bg_{n}-h_{n},\xi\right\rangle _{L^{2}\left(J\right)}\geq-M\left\Vert \xi\right\Vert _{L^{2}\left(J\right)},
\]
and hence it follows from (\ref{eq:Lagr_cond_IJ_gn}) (taking real
part and passing to the limit as $n\rightarrow\infty$) that
\[
\underset{=0}{\underbrace{-\left(1+\mu\right)M\left\Vert \xi\right\Vert _{L^{2}\left(J\right)}}}\leq\text{Re}\left\langle f-\psi-b\gamma,\xi\right\rangle _{L^{2}\left(I\right)},
\]
which results in a contradiction since the right-hand side may be made
negative due to the assumption that $f\notin\left.\mathcal{A}^{\psi,b}\right|_{I}$ and to the arbitrary choice of $\xi$, whereas the
left-hand side vanishes by the assumption $\mu=-1$. This finishes
the proof of (\ref{eq:lim_mu}).\end{proof}
\begin{claim}
$\gamma\in\mathcal{B}_{M,h}^{\psi,b}$.\end{claim}
\begin{proof}
For $g_{n}\in\mathcal{B}_{M,h_{n}}^{\psi,b}$, we have $\left\Vert \psi+bg_{n}-h_{n}\right\Vert _{L^{2}\left(J\right)}\leq M$.
But $h_{n}\rightarrow h$ in $L^{2}\left(J\right)$, $g_{n}\rightharpoonup\gamma$
in $L^{2}\left(J\right)$ (as discussed in the proof of Claim 1) and
so also $\psi+bg_{n}-h_{n}\rightharpoonup\psi+b\gamma-h$ in $L^{2}\left(J\right)$
as $n\rightarrow\infty$. The claim now is a direct consequence of
the general property of weak limits: 
\begin{equation}
\left\Vert \tilde{g}\right\Vert \leq\underset{n\rightarrow\infty}{\lim\inf}\left\Vert \tilde{g}_{n}\right\Vert \hspace{1em}\text{whenever}\hspace{1em}\tilde{g}_{n}\rightharpoonup\tilde{g}\hspace{1em}\text{as}\hspace{1em}n\rightarrow\infty,\label{eq:weak_limits_prop}
\end{equation}
which follows from taking $\xi=\tilde{g}$ in $\underset{n\rightarrow\infty}{\lim}\left\langle \tilde{g}_{n},\xi\right\rangle =\left\langle \tilde{g},\xi\right\rangle $
and the Cauchy-Schwarz inequality.
%%
%% For $g_{n}\in\mathcal{B}_{M,h_{n}}^{\psi,b}$, we have $\left\Vert \psi+bg_{n}-h_{n}\right\Vert _{L^{2}\left(J\right)}\leq M$.
%% By the triangle inequality, let us estimate 
%% \begin{eqnarray*}
%% \left\Vert \psi+b\gamma-h\right\Vert _{L^{2}\left(J\right)}^{2} & = & \left\langle \psi+b\gamma-h-bg_{n}+bg_{n}-h_{n}+h_{n},\psi+b\gamma-h\right\rangle _{L^{2}\left(J\right)}\\
%% & \leq & \left|\left\langle \psi+bg_{n}-h_{n},\psi+b\gamma-h\right\rangle _{L^{2}\left(J\right)}\right|+\left|\left\langle h-h_{n},\psi+b\gamma-h\right\rangle _{L^{2}\left(J%%\right)}\right|\\
%% &  & +\left|\left\langle \gamma-g_{n},\bar{b}\left(\psi+b\gamma-h\right)\right\rangle _{L^{2}\left(J\right)}\right|.
%%\end{eqnarray*}
%%Since $h_{n}\rightarrow h$ in $L^{2}\left(J\right)$ and $g_{n}\rightharpoonup\gamma$
%%in $L^{2}\left(J\right)$ as $n\rightarrow\infty$ (as discussed in
%%the proof of Claim 1), we employ the Cauchy-Schwarz inequality to
%%obtain
%%\[
%%\left\Vert \psi+b\gamma-h\right\Vert _{L^{2}\left(J\right)}^{2}\leq\underset{n\rightarrow\infty}{\lim}\left|\left\langle \psi+bg_{n}-h_{n},\psi+b\gamma-h\right\rangle _{L^{2}\left(J\right)}\right|\leq M\left\Vert \psi+b\gamma-h\right\Vert _{L^{2}\left(J\right)},
%%\]
%%and thus conclude that $\left\Vert \psi+b\gamma-h\right\Vert _{L^{2}\left(J\right)}\leq M$,
%%that is $\gamma\in\mathcal{B}_{M,h}^{\psi,b}$.

\end{proof}
\begin{claim}
$\gamma$ is a minimizer of (\ref{eq:BEP_main}).\end{claim}
\begin{proof}
Since $\gamma\in\mathcal{B}_{M,h}^{\psi,b}$ and $g_{0}$ is a minimizer
of (\ref{eq:BEP_main}), we have
\[
\left\Vert \psi+bg_{0}-f\right\Vert _{L^{2}\left(I\right)}\leq\left\Vert \psi+b\gamma-f\right\Vert _{L^{2}\left(I\right)}.
\]
To deduce the equality, by contradiction, we assume the strict inequality,
or equivalently
\begin{equation}
\left\Vert \psi+bg_{0}-f\right\Vert _{L^{2}\left(I\right)}\leq\left\Vert \psi+b\gamma-f\right\Vert _{L^{2}\left(I\right)}-\xi\label{eq:gamma_assump}
\end{equation}
for some $\xi>0$. We want to show that this inequality would lead
to a contradiction between optimality of solutions $g_{0}\in\mathcal{B}_{M,h}^{\psi,b}$
and $g_{n}\in\mathcal{B}_{M,h_{n}}^{\psi,b}$ for sufficiently large
$n$.

First of all, there exists $g_{0}^{\star}\in\mathcal{B}_{M,h}^{\psi,b}$
and $\tau>0$ such that 
\begin{equation}
\left\Vert \psi+bg_{0}-f\right\Vert _{L^{2}\left(I\right)}=\left\Vert \psi+bg_{0}^{\star}-f\right\Vert _{L^{2}\left(I\right)}-\tau\label{eq:g0_star_ineq_I}
\end{equation}
and $\left\Vert \psi+bg_{0}^{\star}-h\right\Vert _{L^{2}\left(J\right)}<M$.
Indeed, take $g_{0}^{\star}=g_{0}+\epsilon\delta_{g}$ with $\delta_{g}\in H^{2}$,
$\left\Vert \delta_{g}\right\Vert _{H^{2}}=1$ such that

\begin{equation}
\text{Re}\left\langle \psi+bg_{0}-h,b\delta_{g}\right\rangle _{L^{2}\left(J\right)}<0.\label{eq:dg_J_assump}
\end{equation}
Then, since $\left\Vert \psi+bg_{0}-h\right\Vert _{L^{2}\left(J\right)}=M$
(according to Lemma \ref{lem:saturation}), we have 
\[
\left\Vert \psi+bg_{0}^{\star}-h\right\Vert _{L^{2}\left(J\right)}^{2}=\left\Vert \psi+bg_{0}-h\right\Vert _{L^{2}\left(J\right)}^{2}+2\epsilon\text{Re}\left\langle \psi+bg_{0}-h,b\delta_{g}\right\rangle _{L^{2}\left(J\right)}+\epsilon^{2}\left\Vert \delta_{g}\right\Vert _{L^{2}\left(J\right)}^{2}=M^{2}-\eta_{0}
\]
with $\eta_{0}:=-2\epsilon\text{Re}\left\langle \psi+bg_{0}-h,b\delta_{g}\right\rangle _{L^{2}\left(J\right)}-\epsilon^{2}\left\Vert \delta_{g}\right\Vert _{L^{2}\left(J\right)}^{2}>0$
for sufficiently small $\epsilon>0$, that is
\begin{equation}
\left\Vert \psi+bg_{0}^{\star}-h\right\Vert _{L^{2}\left(J\right)}=M-\eta,\hspace{1em}\eta:=\dfrac{\eta_{0}}{\left\Vert \psi+bg_{0}^{\star}-h\right\Vert _{L^{2}\left(J\right)}+M}>0.\label{eq:g0_star_ineq_J}
\end{equation}
Now we consider 
\begin{eqnarray*}
\left\Vert \psi+bg_{0}^{\star}-f\right\Vert _{L^{2}\left(I\right)}^{2} & = & \left\Vert \psi+bg_{0}-f\right\Vert _{L^{2}\left(I\right)}^{2}+2\epsilon\text{Re}\left\langle \psi+bg_{0}-f,b\delta_{g}\right\rangle _{L^{2}\left(I\right)}+\epsilon^{2}\left\Vert \delta_{g}\right\Vert _{L^{2}\left(I\right)}^{2}
\end{eqnarray*}
and note that the optimality condition (\ref{eq:Lagr_cond_fin}) implies
\[
\left\langle \bar{b}\left(f-\psi\right)-g_{0},\delta_{g}\right\rangle _{L^{2}\left(I\right)}=\left(1+\mu\right)\left\langle g_{0}-\bar{b}\left(h-\psi\right),\delta_{g}\right\rangle _{L^{2}\left(J\right)}
\]
\[
\Rightarrow\hspace{1em}\text{Re}\left\langle \psi+bg_{0}-f,b\delta_{g}\right\rangle _{L^{2}\left(I\right)}=-\left(1+\mu\right)\text{Re}\left\langle \psi+bg_{0}-h,b\delta_{g}\right\rangle _{L^{2}\left(J\right)}>0,
\]
where $\mu>-1$ is the Lagrange parameter for the solution $g_{0}$.
Therefore,
\[
\left\Vert \psi+bg_{0}^{\star}-f\right\Vert _{L^{2}\left(I\right)}^{2}=\left\Vert \psi+bg_{0}-f\right\Vert _{L^{2}\left(I\right)}^{2}+\tau_{0}
\]
with $\tau_{0}:=-2\left(1+\mu\right)\text{Re}\left\langle \psi+bg_{0}-h,b\delta_{g}\right\rangle _{L^{2}\left(J\right)}+\epsilon^{2}\left\Vert \delta_{g}\right\Vert _{L^{2}\left(I\right)}^{2}>0$
for small enough $\epsilon$, and so (\ref{eq:g0_star_ineq_I}) follows
with
\begin{equation}
\tau:=\dfrac{\tau_{0}}{\left\Vert \psi+bg_{0}^{\star}-f\right\Vert _{L^{2}\left(I\right)}+\left\Vert \psi+bg_{0}-f\right\Vert _{L^{2}\left(I\right)}}>0.\label{eq:tau_def}
\end{equation}

Now it is easy to see that for large enough $n$, we also have $g_{0}^{\star}\in\mathcal{B}_{M,h_{n}}^{\psi,b}$.
Since $h_{n}\rightarrow h$ in $L^{2}\left(J\right)$ as $n\rightarrow\infty$,
there exists $N_{1}\in\mathbb{N}_{+}$ such that $\left\Vert h-h_{n}\right\Vert _{L^{2}\left(J\right)}<\eta$
whenever $n>N_{1}$, so from (\ref{eq:g0_star_ineq_J}), we deduce
the bound
\begin{equation}
\left\Vert \psi+bg_{0}^{\star}-h_{n}\right\Vert _{L^{2}\left(J\right)}\leq\left\Vert \psi+bg_{0}^{\star}-h\right\Vert _{L^{2}\left(J\right)}+\left\Vert h-h_{n}\right\Vert _{L^{2}\left(J\right)}\leq M.\label{eq:g0_star_in_B_h_n}
\end{equation}

On the other hand, by the property of weak limits (\ref{eq:weak_limits_prop}), we have

%On the other hand, by the triangle inequality, we have 
%\[
%\left\Vert \psi+bg_{n}-f\right\Vert _{L^{2}\left(I\right)}=\left\Vert \psi+b\gamma-f-b\left(\gamma-g_{n}\right)\right\Vert _{L^{2}\left(I\right)}\geq\left\Vert \psi+b\gamma-f\right\Vert _{L^{2}\left(I\right)}-\left\Vert g_{n}-\gamma\right\Vert _{L^{2}\left(I\right)}.
%\]
%Due to the boundedness of $\left\{ g_{n}\right\} _{n=1}^{\infty}$
%in $L^{2}\left(I\right)$, the limit $\underset{n\rightarrow\infty}{\lim\inf}\left\Vert g_{n}-\gamma\right\Vert _{L^{2}\left(I\right)}$
%exists and it must coincide with the weak $L^{2}\left(I\right)$ limit
%$g_{n}-\gamma\rightharpoonup0$ as $n\rightarrow\infty$. So taking
%$\underset{n\rightarrow\infty}{\lim\inf}$ of the both sides, we arrive
%at
\[
\underset{n\rightarrow\infty}{\lim\inf}\left\Vert \psi+bg_{n}-f\right\Vert _{L^{2}\left(I\right)}\geq\left\Vert \psi+b\gamma-f\right\Vert _{L^{2}\left(I\right)},
\]
that is, for any given $\rho>0$, 
\begin{equation}
\left\Vert \psi+bg_{n}-f\right\Vert _{L^{2}\left(I\right)}>\left\Vert \psi+b\gamma-f\right\Vert _{L^{2}\left(I\right)}-\rho\label{eq:g_n_on_I}
\end{equation}
holds when $n$ is taken large enough. In particular, there is $N_{2}\in\mathbb{N}_{+}$
such that (\ref{eq:g_n_on_I}) holds for $n\geq N_{2}$ with $\rho=\tau$.
Then, for any $n\geq\max\left\{ N_{1},N_{2}\right\} $, (\ref{eq:g_n_on_I})
can be combined with (\ref{eq:gamma_assump}) and (\ref{eq:g0_star_ineq_I})
to give
\[
\left\Vert \psi+bg_{n}-f\right\Vert _{L^{2}\left(I\right)}>\left\Vert \psi+bg_{0}^{\star}-f\right\Vert _{L^{2}\left(I\right)}+\xi-2\tau.
\]
According to (\ref{eq:tau_def}), $\tau$ can be made arbitrarily
small by the choice of $\delta_{g}$ and $\epsilon$ whereas $\xi$
is a fixed number. Therefore, we have $\left\Vert \psi+bg_{0}^{\star}-f\right\Vert _{L^{2}\left(I\right)}<\left\Vert \psi+bg_{n}-f\right\Vert _{L^{2}\left(I\right)}$
and $g_{0}^{\star}\in\mathcal{B}_{M,h_{n}}^{\psi,b}$ (according to (\ref{eq:g0_star_in_B_h_n})).
In other words, $g_{0}^{\star}$ gives a better solution than $g_{n}$,
and hence, by  uniqueness (Theorem \ref{thm:existence}), we get
a contradiction to the minimality of $g_{n}$ in (\ref{eq:g_n_def}).\end{proof}
\begin{rem}
\label{rem:constr_dropout}As it is mentioned in the formulation of
Theorem \ref{thm:charact}, for $g_{0}$ to be a solution to (\ref{eq:BEP_main}),
the Lagrange parameter $\mu$ has yet to be chosen such that $g_{0}$ given
by (\ref{eq:g0_sol_fin}) satisfies the constraint $\left\Vert \psi+bg_{0}-h\right\Vert _{L^{2}\left(J\right)}=M$, which makes the well-posedness (regularization) effective, see Proposition \ref{prop:trace_res2} and discussion in the beginning of Section \ref{sec:comp_issues}.

We note that the formal substitution $\mu=-1$ in (\ref{eq:g0_pre_sol_fin})
leaves out the constraint on $J$ and leads to the situation $\left.g\right|_{I}=\bar{b}\left(f-\psi\right)$
that was ruled out initially by the requirement (\ref{eq:g_cond}).
\end{rem}
When $f\in\left.\mathcal{A}^{\psi,b}\right|_{I}$, we face an extrapolation
problem of holomorphic extension from $I$ inside the disk preserving
interior pointwise data. In such a case, $\bar{b}\left(f-\psi\right)\in\left.H^{2}\right|_{I}$
and Proposition \ref{prop:Carleman} (or alternative scheme from \cite{Patil} mentioned in Remark \ref{rem:Patil_ref})
applies to construct the extension $g_{0}$ such that $\left.g_{0}\right|_{I}=\bar{b}\left(f-\psi\right)$
which can be regarded as the solution if we give up the control on
$J$ which means that for a given $h$ the parameter $M$ should
be relaxed (yet remaining finite) to avoid an overdetermined problem.
Otherwise, keeping the original bound $M$, despite $f\in\left.\mathcal{A}^{\psi,b}\right|_{I}$,
we must accept non-zero minimum of the cost functional of the problem
in which case the solution $g_{0}$ is still given by (\ref{eq:g0_sol_fin})
which proof is valid since now $\left.g_{0}\right|_{I}\neq\bar{b}\left(f-\psi\right)$.
The latter situation, from geometrical point of view, is nothing but
finding a projection of $f\in\left.\mathcal{A}^{\psi,b}\right|_{I}$
onto the convex subset ${\cal C}_{M,h}^{\psi,b}\subseteq\left.\mathcal{A}^{\psi,b}\right|_{I}$.

However, returning back to the realistic case, when $f\in L^{2}\left(I\right)\backslash\left.\mathcal{A}^{\psi,b}\right|_{I}$,
the solution to (\ref{eq:BEP_main}) can still be written in an integral
form in spirit of the Carleman's formula (\ref{eq:Carleman}) as given
by the following result (see also \cite{Baratchart-Leblond} where
it was stated for the case $\psi\equiv0$, $b\equiv1$).
\begin{prop}
For $\mu\in\left(-1,0\right)$, the solution (\ref{eq:g0_sol_fin})
can be represented as

\begin{equation}
g_{0}\left(z\right)=\frac{1}{2\pi i}\int_{\mathbb{T}}\left(\dfrac{\Phi\left(\xi\right)}{\Phi\left(z\right)}\right)^{\alpha}\left(\bar{b}\left(f-\psi\right)\vee\bar{b}\left(h-\psi\right)\right)\left(\xi\right)\dfrac{d\xi}{\xi-z},\hspace{1em}z\in\mathbb{D},\label{eq:g_sol_int}
\end{equation}
where 
\begin{equation}
\Phi\left(z\right)=\exp\left\{ \frac{\log\rho}{2\pi i}\int_{I}\dfrac{\xi+z}{\xi-z}d\xi\right\} ,\hspace{1em}\alpha=-\frac{\log\left(1+\mu\right)}{2\log\rho},\hspace{1em}\rho>1.\label{eq:quench_func}
\end{equation}
\end{prop}
\begin{proof}
First of all, we note that (\ref{eq:quench_func}) is a quenching
function satisfying $\left|\Phi\right|=\rho\vee1$ on $\mathbb{T}$
and $\left|\Phi\right|>1$ on $\mathbb{D}$ which can be constructed
following the recipe of Remark \ref{rem:quench_choice}. The condition
$\left|\Phi\right|>1$ on $\mathbb{D}$ and the minimum modulus principle
for analytic functions imply the requirement $\rho>1$.\\
To show the equivalence, one can start from (\ref{eq:g_sol_int})
and arrive at (\ref{eq:g0_sol_fin}) for a suitable choice of the
parameters. Indeed, since $\Phi\in H^{\infty}$, (\ref{eq:g_sol_int})
implies
\[
\Phi^{\alpha}g_{0}=P_{+}\left[\Phi^{\alpha}\left(\bar{b}\left(f-\psi\right)\vee\bar{b}\left(h-\psi\right)\right)\right]
\]
\[
\Rightarrow\hspace{1em}P_{+}\left(\left|\Phi\right|^{2\alpha}g_{0}\right)=P_{+}\bigl(\bar{\Phi}^{\alpha}P_{+}\left[\Phi^{\alpha}\left(\bar{b}\left(f-\psi\right)\vee\bar{b}\left(h-\psi\right)\right)\right]\bigr).
\]
We can represent 
\[
P_{+}\left[\Phi^{\alpha}\left(\bar{b}\left(f-\psi\right)\vee\bar{b}\left(h-\psi\right)\right)\right]=\Phi^{\alpha}\left(\bar{b}\left(f-\psi\right)\vee\bar{b}\left(h-\psi\right)\right)-P_{-}\left[\Phi^{\alpha}\left(\bar{b}\left(f-\psi\right)\vee\bar{b}\left(h-\psi\right)\right)\right]
\]
with $P_{-}$ being anti-analytic projection defined in Section \ref{sec:intro_Hardy}.
Since 
\begin{eqnarray*}
\left\langle \bar{\Phi}^{\alpha}P_{-}\left[\Phi^{\alpha}\left(\bar{b}\left(f-\psi\right)\vee\bar{b}\left(h-\psi\right)\right)\right],u\right\rangle _{L^{2}\left(\mathbb{T}\right)} & = & \left\langle P_{-}\left[\Phi^{\alpha}\left(\bar{b}\left(f-\psi\right)\vee\bar{b}\left(h-\psi\right)\right)\right],\Phi^{\alpha}u\right\rangle _{L^{2}\left(\mathbb{T}\right)}=0
\end{eqnarray*}
 for any $u\in H^{2}$, it follows that $P_{+}\bigl(\bar{\Phi}^{\alpha}P_{-}\left[\Phi^{\alpha}\left(\bar{b}\left(f-\psi\right)\vee\bar{b}\left(h-\psi\right)\right)\right]\bigr)=0$
and so we deduce

\[
P_{+}\left(\left|\Phi\right|^{2\alpha}g_{0}\right)=P_{+}\left[\left|\Phi\right|^{2\alpha}\left(\bar{b}\left(f-\psi\right)\vee\bar{b}\left(h-\psi\right)\right)\right].
\]
Given $\rho>1$, choose $\alpha>0$ such that $\rho^{2\alpha}=\dfrac{1}{1+\mu}$
(this restricts the range $\mu>-1$ to $\mu\in\left(-1,0\right)$).
Then, $\left.\left|\Phi\right|^{2\alpha}\right|_{I}=\dfrac{1}{1+\mu}$,
$\left.\left|\Phi\right|^{2\alpha}\right|_{J}=1$, and hence
\[
P_{+}\left(\dfrac{1}{1+\mu}g_{0}\vee g_{0}\right)=P_{+}\left(\dfrac{\bar{b}}{1+\mu}\left(f-\psi\right)\vee\bar{b}\left(h-\psi\right)\right)
\]
\[
\Rightarrow\hspace{1em}P_{+}\left(g_{0}\vee g_{0}\right)+\mu P_{+}\left(0\vee g_{0}\right)=P_{+}\left(\bar{b}\left(f-\psi\right)\vee\left(1+\mu\right)\bar{b}\left(h-\psi\right)\right),
\]
which directly furnishes (\ref{eq:g0_sol_fin}).
\end{proof}
\newpage
\section{\label{sec:psi_choice}Choice of interpolation function and solution
reduction}

Before we proceed with computational aspects, it is worth discussing
the choice of interpolant $\psi$ which up to this point was any $H^{2}$
function satisfying the interpolation conditions (\ref{eq:meas_cond}). 

We will first consider a particular choice of the interpolant following
\cite{Schvedenko} and then discuss the general case.
\begin{prop}
\label{prop:interp_construction}The $H^2$ function defined for $z \in \mathbb{D}$ by
\begin{equation}
\psi\left(z\right)=\sum_{k=1}^{N}\psi_{k}\mathcal{K}\left(z_{k},z\right)\hspace{1em}\text{with}\hspace{1em}\mathcal{K}\left(z_{k},z\right):=\frac{1}{1-\bar{z}_{k}z}\label{eq:interp_expans}
\end{equation}
interpolates the data \eqref{eq:meas_cond} for an appropriate choice of the constants
$\left\{ \psi_{k}\right\} _{k=1}^{N}\in\mathbb{C}$ which exists regardless
of a priori prescribed values $\left\{ \omega_{k}\right\} _{k=1}^{N}$
and choice of the points $\left\{ z_{k}\right\} _{k=1}^{N}$ (providing they are all different). Moreover, it is the unique interpolant
of minimal $H^{2}$ norm.\end{prop}
\begin{proof}
We note that the function $\mathcal{K}\left(\cdot,\cdot\right)$ is
the reproducing kernel for $H^{2}$ meaning
that, for any $u\in H^{2}$, $z_{0}\in\mathbb{D}$,
point evaluation is given by the inner product $u\left(z_{0}\right)=\left\langle u,\mathcal{K}\left(z_{0},\cdot\right)\right\rangle _{L^{2}\left(\mathbb{T}\right)},$
which is a direct consequence of the Cauchy integral formula because
$d\theta=\dfrac{1}{iz}dz$ in (\ref{eq:Hp_innerprod}). The coefficients
$\left\{ \psi_{k}\right\} _{k=1}^{N}\in\mathbb{C}$ in (\ref{eq:interp_expans})
are to be found from the requirement (\ref{eq:meas_cond}). We therefore
have
\begin{equation}
\psi_{k}=\sum_{j=1}^{N}S_{kj}\omega_{j},\quad\text{where}\quad S:=\left[S_{kj}\right] =  \left[\mathcal{K}\left(z_{k},z_{j}\right)\right]^{-1} \, , \ k, j=1, \dots, N.\label{eq:S_def}
\end{equation}
In order to see that the existence of the inverse matrix $S$ is unconditional,
we note that $\mathcal{K}\left(z_{k},z_{j}\right)=\left\langle \mathcal{K}\left(z_{k},\cdot\right),\mathcal{K}\left(z_{j},\cdot\right)\right\rangle _{L^{2}\left(\mathbb{T}\right)}$,
and hence it is the inverse of a Gram matrix which exists since $z_{k}\neq z_{j}$
whenever $k\neq j$ providing that all functions $\left\{ \mathcal{K}\left(z_{k},z\right)\right\} _{k=1}^{N}$
are linearly independent. To check the latter, we verify the implication
\[
\sum_{k=1}^{N}c_{k}\mathcal{K}\left(z_{k},z\right)=0\qquad\Rightarrow\qquad c_{k}=0,\quad k=1, \dots, N.
\]
Employing the identity $\dfrac{1}{1-\bar{z}_{k}z}=\sum_{n=0}^{\infty}\bar{z}_{k}^{n}z^{n}$
that holds due to $\left|z_{k}z\right|<1$, we see that
\[
\sum_{n=0}^{\infty}\left(\sum_{k=1}^{N}c_{k}\bar{z}_{k}^{n}\right)z^{n}=0,\quad \forall z \in \mathbb{D} \qquad\Rightarrow\qquad\sum_{k=1}^{N}c_{k}\bar{z}_{k}^{n}=0,\quad n\in\mathbb{N}_{0}.
\]
But, by induction on $k$, this necessarily implies that $c_{k}=0$,
$k=1,\dots,N$ and thus proves the linear independence. 

To show that $\psi\in H^{2}$ is the unique
interpolant of minimal norm, we let $\psi_{0}\in H^{2}$
be another interpolant satisfying (\ref{eq:meas_cond}). Then, $\phi_{0}:=\psi-\psi_{0}\in H^{2}$
is such that $\left.\phi_{0}\right|_{z=z_{k}}=0$, $k=1, \dots, N$,
or equivalently,
\[
\left\langle \phi_{0},\mathcal{K}\left(z_{k},\cdot\right)\right\rangle _{L^{2}\left(\mathbb{T}\right)}=0,\quad k=1,\dots,N
\]
meaning orthogonality of $\phi_{0}\left(z\right)$ to a linear span
of $\left\{ \mathcal{K}\left(z_{k},z\right)\right\} _{k=1}^{N}$.
But $\psi$ exactly belongs to this span, and hence
\begin{equation}
\left\Vert \psi_{0}\right\Vert _{H^{2}}^{2}=\left\Vert \psi\right\Vert _{H^{2}}^{2}+\left\Vert \phi_{0}\right\Vert _{H^{2}}^{2}>\left\Vert \psi\right\Vert _{H^{2}}^{2},\label{eq:interp_H2_minimality}
\end{equation}
which shows that $\psi$ is the unique  interpolating $H^2$ function of minimal norm.\end{proof}
\begin{rem}
\label{rem:sol_fin_alt}With this choice of $\psi$, the solution
(\ref{eq:g0_sol_fin}) takes the form
\begin{equation}
g_{0}=\left(1+\mu\phi\right)^{-1}\left[P_{+}\left(\bar{b}\left(f\vee h\right)\right)+
\mu \, 
P_+ \left(0 \vee \bar{b}\left(h-\psi\right)\right) %\phi\left(\bar{b}\left(h-\psi\right)\right)
\right] \, .\label{eq:g0_sol_fin_alt}
\end{equation}
%if one abuses notation for the operator $\phi$ allowing it to be defined on $L^{2}\left(J\right)$. 
Indeed, since $\left\langle \mathcal{K}\left(z_{k},z\right),bu\right\rangle _{L^{2}\left(\mathbb{T}\right)}=0$,
$k=1,\dots,N$ for any $u\in H^{2}$, we have
$P_{+}\left(\bar{b}\psi\right)=0$ whenever $\psi$ is given by (\ref{eq:interp_expans}).
\end{rem}
Now it may look tempting to consider other choices of the interpolant
to improve the $L^{2}$-bounds in (\ref{eq:B_space}) or (\ref{eq:BEP_main})
rather than being itself of minimal $L^{2}\left(\mathbb{T}\right)$ norm.
However, the choice of the interpolant does not affect the combination $\tilde{g}_{0}=\psi+bg_{0}$, a result that is not surprising at all from physical point of view since $\psi$ is an auxiliary tool which should not affect solution whose dependence must
eventually boil down to given data (measurement related quantities) only:
$\left\{ z_{k}\right\} _{k=1}^{N}$, $\left\{ \omega_{k}\right\} _{k=1}^{N}$,
$f$ and $h$. More precisely, we have
\begin{lem}
\label{lem:psi_indep}Given arbitrary $\psi_{1}$, $\psi_{2}\in H^{2}$
satisfying (\ref{eq:meas_cond}), we have $\psi_{1}+bg_{0}\left(\psi_{1}\right)=\psi_{2}+bg_{0}\left(\psi_{2}\right)$. \end{lem}

\begin{proof}
First of all, we note that the dependence $g_{0}\left(\psi\right)$
is not only due to explicit appearance of $\psi$ in (\ref{eq:g0_sol_fin}),
but also because the Lagrange parameter $\mu$, in general, has to
be readjusted according to $\psi$, that is $\mu=\mu\left(\psi\right)$
so that
\begin{equation}
\left\Vert \psi_{k}+bg_{0}\left(\psi_{k}\right)-h\right\Vert _{L^{2}\left(J\right)}^{2}=M^{2},\hspace{1em}k=1,2,\label{eq:constr_var_psi}
\end{equation}
where we mean $g_{0}\left(\psi\right)=g_{0}\left(\psi,\mu\left(\psi\right)\right)$.
Let us denote $\delta_{\psi}:=\psi_{2}-\psi_{1}$, $\delta_{\mu}:=\mu\left(\psi_{2}\right)-\mu\left(\psi_{1}\right)$,
$\delta_{g}:=g_{0}\left(\psi_{2}\right)-g_{0}\left(\psi_{1}\right)$.
Taking difference of both equations (\ref{eq:constr_var_psi}), we
have
\[
\left\langle \delta_{\psi}+b\delta_{g},\psi_{1}+bg_{0}\left(\psi_{1}\right)-h\right\rangle _{L^{2}\left(J\right)}+\left\langle \psi_{2}+bg_{0}\left(\psi_{2}\right)-h,\delta_{\psi}+b\delta_{g}\right\rangle _{L^{2}\left(J\right)}=0
\]

\begin{equation}
\Rightarrow\hspace{1em}2\text{Re}\left\langle \bar{b}\delta_{\psi}+\delta_{g},\bar{b}\psi_{2}+g_{0}\left(\psi_{2}\right)-\bar{b}h\right\rangle _{L^{2}\left(J\right)}=\left\Vert \delta_{\psi}+b\delta_{g}\right\Vert _{L^{2}\left(J\right)}^{2}.\label{eq:delta_psi_J}
\end{equation}
On the other hand, the optimality condition (\ref{eq:orthog_cond})
implies that, for any $\xi\in H^{2}$,
\[
\left\langle \bar{b}\psi_{k}+g_{0}\left(\psi_{k}\right)-\bar{b}f,\xi\right\rangle _{L^{2}\left(I\right)}=-\left(1+\mu\left(\psi_{k}\right)\right)\left\langle \bar{b}\psi_{k}+g_{0}\left(\psi_{k}\right)-\bar{b}h,\xi\right\rangle _{L^{2}\left(J\right)},\hspace{1em}k=1,2,
\]
and therefore
\begin{equation}
\left\langle \bar{b}\delta_{\psi}+\delta_{g},\xi\right\rangle _{L^{2}\left(I\right)}=-\left(1+\mu\left(\psi_{1}\right)\right)\left\langle \bar{b}\delta_{\psi}+\delta_{g},\xi\right\rangle _{L^{2}\left(J\right)}-\delta_{\mu}\left\langle \bar{b}\psi_{2}+g_{0}\left(\psi_{2}\right)-\bar{b}h,\xi\right\rangle _{L^{2}\left(J\right)}.\label{eq:delta_psi_IJ}
\end{equation}
Since $\delta_{\psi}\in H^{2}$, due to (\ref{eq:meas_cond}), it
is zero at each $z_{j}$, $j=1,\dots,N$, and hence factorizes as $\delta_{\psi}=b\eta$
for some $\eta\in H^{2}$. This allows us to take $\xi=\bar{b}\delta_{\psi}+\delta_{g}\in H^{2}$
in (\ref{eq:delta_psi_IJ}) to yield
\[
\left\Vert \eta+\delta_{g}\right\Vert _{L^{2}\left(I\right)}^{2}=-\left(1+\mu\left(\psi_{1}\right)\right)\left\Vert \eta+\delta_{g}\right\Vert _{L^{2}\left(J\right)}^{2}-\delta_{\mu}\left\langle \bar{b}\psi_{2}+g_{0}\left(\psi_{2}\right)-\bar{b}h,\eta+\delta_{g}\right\rangle _{L^{2}\left(J\right)}.
\]
Note that the inner product term here is real-valued since the others
are, and so employing (\ref{eq:delta_psi_J}), we arrive at
\[
\left\Vert \eta+\delta_{g}\right\Vert _{L^{2}\left(I\right)}^{2}+\left(1+\mu\left(\psi_{1}\right)\right)\left\Vert \eta+\delta_{g}\right\Vert _{L^{2}\left(J\right)}^{2}=-\dfrac{1}{2}\delta_{\mu}\left\Vert \eta+\delta_{g}\right\Vert _{L^{2}\left(J\right)}^{2}
\]
which, due to $\mu>-1$, entails that $\delta_{\mu}\leq0$. But, clearly,
interchanging $\psi_{1}$ and $\psi_{2}$, we would get $\delta_{\mu}\geq0$,
and so $\delta_{\mu}=0$ leading to $\left\Vert \delta_{\psi}+b\delta_{g}\right\Vert _{L^{2}\left(\mathbb{T}\right)}^{2}=\left\Vert \eta+\delta_{g}\right\Vert _{L^{2}\left(I\right)}^{2}+\left\Vert \eta+\delta_{g}\right\Vert _{L^{2}\left(J\right)}^{2}=0$
which finishes the proof.\end{proof}

%\begin{proof}
%Let us denote $\delta_{\psi}:=\psi_{2}-\psi_{1}$. Using (\ref{eq:g0_sol_fin}),
%we have
%\[
%g_{0}\left(\psi_{2}\right)-g_{0}\left(\psi_{1}\right)=-\left(1+\mu\phi\right)^{-1}P_{+}\left(\bar{b}\delta_{\psi}\vee\left(1+\mu\right)\bar{b}\delta_{\psi}\right).
%\]
%Since $\delta_{\psi}\in H^{2}$, due to (\ref{eq:meas_cond}),
%it is zero at each $z_{j}$, $j=1,\dots,N$ and hence factorizes as
%$\delta_{\psi}=b\xi$ for some $\xi\in H^{2}$.
%Then
%% \begin{eqnarray*}
%% g_{0}\left(\psi_{2}\right)-g_{0}\left(\psi_{1}\right) & = & -\left(1+\mu\phi\right)^{-1}P_{+}\left(\xi\vee\left(1+\mu\right)\xi\right)\\
%%  & = & -\left(1+\mu\phi\right)^{-1}\left(\xi+\mu P_{+}\left(0\vee\xi\right)\right)\\
%%  & = & -\xi,
%% \end{eqnarray*}
%\[
%g_{0}\left(\psi_{2}\right)-g_{0}\left(\psi_{1}\right) =  -\left(1+\mu\phi\right)^{-1}P_{+}\left(\xi\vee\left(1+\mu\right)\xi\right)\\
%  =  -\left(1+\mu\phi\right)^{-1}\left(\xi+\mu P_{+}\left(0\vee\xi\right)\right)\\
% = -\xi,
%\]
%and so we conclude 
%\[
%\psi_{2}-\psi_{1}+b\left(g_{0}\left(\psi_{2}\right)-g_{0}\left(\psi_{1}\right)\right)=\delta_{\psi}-b\xi=0.
%\]
% \end{proof}
Combining this lemma with Remark \ref{rem:sol_fin_alt}, we can formulate
\begin{cor}
Independently of choice of $\psi\in H^{2}$
fulfilling (\ref{eq:meas_cond}), the final solution $\tilde{g}_{0}=\psi+bg_{0}$
is given by
\begin{equation}
\tilde{g}_{0}=\psi+b\left(1+\mu\phi\right)^{-1}\left[P_{+}\left(\bar{b}\left(f\vee h\right)\right)+\mu
P_+ \left(0 \vee \bar{b}\left(h-\psi\right)\right) %\phi\left(\bar{b}\left(h-\psi\right)\right)
\right].
\label{eq:g0_sol_fin_alt_compl}
\end{equation}
\label{cor:corgtilde0}
\end{cor}

These results will be employed for analytical purposes in Section
\ref{sec:stab_res}.

Even though it is not going to be used here, we also note that it
is possible to construct an interpolant whose norm does not exceed
\textit{a priori} given bound providing a certain quadratic form involving
interpolation data and value of the bound is positive semidefinite
\cite{Duren-Williams}.

\newpage
\section{\label{sec:comp_issues}Computational issues and error estimate}

We would like to stress again that the obtained formulas (\ref{eq:g0_sol_fin}),
(\ref{eq:g_sol_int}) and (\ref{eq:g0_sol_fin_alt}) furnish solution
only in an implicit form with the Lagrange parameter $\mu$ still
to be chosen such that the solution satisfies the equality constraint in (\ref{eq:B_space}).
As it was mentioned in Remark \ref{rem:constr_dropout}, the constraint
in $\mathcal{B}_{M,h}^{\psi,b}$ does not enter the solution characterisation
(\ref{eq:g0_pre_sol_fin}) when $\mu=-1$, so as $\mu\searrow-1$
we expect perfect approximation of the given $f\in L^{2}\left(I\right)\backslash\left.\mathcal{A}^{\psi,b}\right|_{I}$
at the expense of uncontrolled growth of the quantity 
\begin{equation}
\label{eq:defM0}
M_{0}\left(\mu\right):=\left\Vert \psi+bg_{0}\left(\mu\right)-h\right\Vert _{L^{2}\left(J\right)}
\end{equation}
according to Propositions \ref{prop:trace_res1} and \ref{prop:trace_res2}. This is not surprising since the inclusion $\mathcal{B}_{M_{1},h}^{\psi,b}\subset\mathcal{B}_{M_{2},h}^{\psi,b}$ whenever $M_{1}<M_{2}$ implies that the minimum of the cost functional of (\ref{eq:BEP_main}) sought over $\mathcal{B}_{M_{1},h}^{\psi,b}$ is bigger than that for $\mathcal{B}_{M_{2},h}^{\psi,b}$.
For devising a feasible for applications solution, a suitable
trade-off between value of $\mu$ (governing quality of approximation
on $I$) and choice of the admissible bound $M$ has to be found. 
To gain insight into this situation, we define the error of approximation
as 
\begin{equation}
\label{eq:def_e}
e\left(\mu\right):=\left\Vert \psi+bg_{0}\left(\mu\right)-f\right\Vert _{L^{2}\left(I\right)}^{2},
\end{equation}
and proceed with establishing connection between $e$ and $M_{0}$.

\subsection{Monotonicity and boundedness}
Here we mainly follow the steps of \cite{Alpay, Baratchart-Leblond} where similar studies has been done without interpolation conditions.
\begin{prop} \label{prop:monot}
The following monotonicity results hold 
\begin{equation}
\dfrac{de}{d\mu}>0,\qquad\dfrac{dM_{0}^{2}}{d\mu}<0.\label{eq:e-M_monot}
\end{equation}
Moreover, we have
\begin{equation}
\dfrac{de}{d\mu}=-\left(\mu+1\right)\dfrac{dM_{0}^{2}}{d\mu}.\label{eq:e-M_connect}
\end{equation}
\end{prop}
\begin{proof}
From (\ref{eq:g0_sol_fin}), using commutation of $\phi$ and $\left(1+\mu\phi\right)^{-1}$,
we compute 
\[
\dfrac{dg_{0}}{d\mu}=-\left(1+\mu\phi\right)^{-2}\phi P_{+}\left(\bar{b}\left(f-\psi\right)\vee\left(1+\mu\right)\bar{b}\left(h-\psi\right)\right)+\left(1+\mu\phi\right)^{-1}P_{+}\left(0\vee\bar{b}\left(h-\psi\right)\right)
\]
\begin{equation}
\Rightarrow\hspace{1em}\dfrac{dg_{0}}{d\mu}=-\left(1+\mu\phi\right)^{-1}\left[\phi g_{0}+P_{+}\left(0\vee\bar{b}\left(\psi-h\right)\right)\right],\label{eq:dg0_dmu}
\end{equation}
and thus
\[
\dfrac{dM_{0}^{2}}{d\mu} =  2\text{Re}\left\langle b\frac{dg_{0}}{d\mu},\psi+bg_{0}-h\right\rangle _{L^{2}\left(J\right)}
\]
\begin{equation}
 =  -2\text{Re}\left\langle \left(1+\mu\phi\right)^{-1}\left[\phi g_{0}+P_{+}\left(0\vee\bar{b}\left(\psi-h\right)\right)\right],\phi g_{0}+P_{+}\left(0\vee\bar{b}\left(\psi-h\right)\right)\right\rangle _{L^{2}\left(\mathbb{T}\right)} <0 ,
\label{eq:M_monot}
\end{equation}
The inequality here is due to the spectral result (\ref{eq:spectrum})
implying 
\[
\text{Re}\left\langle \left(1+\mu\phi\right)^{-1}\xi,\xi\right\rangle _{L^{2}\left(\mathbb{T}\right)}=\left\langle \left(1+\mu\phi\right)^{-1}\xi,\xi\right\rangle _{L^{2}\left(\mathbb{T}\right)}\geq0
\]
for any $\xi\in H^{2}$ and $\mu>-1$ whereas
the equality in (\ref{eq:M_monot}) would be possible, according to
Proposition \ref{prop:phi_inject}, only when $\left.g_{0}\right|_{J}=\bar{b}$$\left(h-\psi\right)$,
that is $M_{0}=0$, the case that was eliminated by Corollary \ref{cor:M_pos}.

Now, for any $\beta\in\mathbb{R}$, making use of (\ref{eq:dg0_dmu})
again, we compute
\begin{eqnarray*}
\dfrac{de}{d\mu}&=&2\text{Re}\left\langle \frac{dg_{0}}{d\mu},\bar{b}\left(\psi-f\right)+g_{0}\right\rangle _{L^{2}\left(I\right)}\\
 & = & -2\text{Re}\left\langle \left(1+\mu\phi\right)^{-1}\left[\phi g_{0}+P_{+}\left(0\vee\bar{b}\left(\psi-h\right)\right)\right],\left(\bar{b}\left(\psi-f\right)+g_{0}\right)\vee0\right\rangle _{L^{2}\left(\mathbb{T}\right)}\\ 
 & = & -\beta\dfrac{dM_{0}^{2}}{d\mu}-2\text{Re} B,
\end{eqnarray*}

%\[
%\dfrac{de}{d\mu} =2\text{Re}\left\langle \frac{dg_{0}}{d\mu},\bar{b}\left(\psi-f\right)+g_{0}\right\rangle _{L^{2}\left(I\right)}
%\]
%\[  =  -2\text{Re}\left\langle \left(1+\mu\phi\right)^{-1}\left[\phi g_{0}+P_{+}\left(0\vee\bar{b}\left(\psi-h\right)\right)\right],\left(\bar{b}\left(\psi-f\right)+g_{0}\right)\vee0\right\rangle _{L^{2}\left(\mathbb{T}\right)}
%\]
%\[
% = -\beta\dfrac{dM_{0}^{2}}{d\mu}-2\text{Re} B \, ,
%\]

with $B$ given by

\begin{eqnarray*}
 & & \left\langle \left(1+\mu\phi\right)^{-1}\left[\phi g_{0}+P_{+}\left(0\vee\bar{b}\left(\psi-h\right)\right)\right],\beta\phi g_{0}+\beta P_{+}\left[0\vee\bar{b}\left(\psi-h\right)\right]+\left(\bar{b}\left(\psi-f\right)+g_{0}\right)\vee0\right\rangle _{L^{2}\left(\mathbb{T}\right)}\\
 & & = \left\langle \left(1+\mu\phi\right)^{-1}\left[\phi g_{0}+P_{+}\left(0\vee\bar{b}\left(\psi-h\right)\right)\right],\left(\bar{b}\left(\psi-f\right)+g_{0}\right)\vee\beta\left[\bar{b}\left(\psi-h\right)+g_{0}\right]\right\rangle _{L^{2}\left(\mathbb{T}\right)}\\
 & & = \left\langle b\left(1+\mu\phi\right)^{-1}\left[\phi g_{0}+P_{+}\left(0\vee\bar{b}\left(\psi-h\right)\right)\right],\left(\psi+bg_{0}-f\right)\vee\beta\left(\psi+bg_{0}-h\right)\right\rangle _{L^{2}\left(\mathbb{T}\right)},
\end{eqnarray*}
where we suppressed the $P_{+}$ operator on the right part of the inner product in the second line due to the fact that the left part of it belongs to $H^{2}$.

The choice $\beta=\mu+1=\lambda$ entails $\text{Re} B = 0$
due to (\ref{eq:orthog_cond}), and we thus obtain (\ref{eq:e-M_connect}).
Since $\mu+1>0$, (\ref{eq:e-M_connect}) combines with (\ref{eq:M_monot})
to furnish the remaining inequality in (\ref{eq:e-M_monot}).
\end{proof}
In particular, equation (\ref{eq:e-M_connect}) encodes how the decay
of the approximation error on $I$ is accompanied by $\tilde{g}_{0}=\psi+bg_{0}$
departing further away from given $h$ on $J$ as $\mu\searrow-1$.
Even though more concrete asymptotic estimates on the increase of $M_{0}\left(\mu\right)$
near $\mu=-1$ will be discussed later on, we start providing merely
a rough square-integrability result which is contained in the following

\begin{prop}
\label{prop:M0_integr}The deviation $M_{0}$ of the
solution $\tilde{g}_{0}$ from $h$ on $J$ has moderate growth
as $\mu\searrow-1$ so that, for any $-1<\mu_{0}<\infty$, 
\begin{equation}
\int_{-1}^{\mu_{0}}M_{0}^{2}\left(\mu\right)d\mu<\infty.\label{eq:propM0}
\end{equation}
\end{prop}
\begin{proof}
Integration of (\ref{eq:e-M_connect}) by
parts from $\mu$ to $\mu_{0}$ yields 
\begin{equation}
e\left(\mu_{0}\right)-e\left(\mu\right)=\left(\mu+1\right)M_{0}^{2}\left(\mu\right)-\left(\mu_{0}+1\right)M_{0}^{2}\left(\mu_{0}\right)+\int_{\mu}^{\mu_{0}}M_{0}^{2}\left(\tau\right)d\tau.\label{eq:e_mu_integral}
\end{equation}
 As it was already mentioned in the beginning of the section, Proposition
\ref{prop:trace_res1} implies that the cost functional goes to $0$
when $\mu$ decays to $-1$: 
\begin{equation}
e\left(\mu\right)\searrow0\hspace{1em}\text{as}\hspace{1em}\mu\searrow-1.\label{eq:e_mu_lowlimit}
\end{equation}
We are now going to estimate the behavior of the product $\left(\mu+1\right)M_{0}^{2}\left(\mu\right)$.
First of all, since the constraint is saturated (Lemma \ref{lem:saturation}),
condition (\ref{eq:Lagr_cond_fin}) implies that
\begin{eqnarray}
\left\langle f-\psi-bg_{0},bg_{0}\right\rangle _{L^{2}\left(I\right)} & = & \left(1+\mu\right)\left\langle h-\psi-bg_{0},-bg_{0}\right\rangle _{L^{2}\left(J\right)}\nonumber \\
 & = & \left(1+\mu\right)M_{0}^{2}-\left(1+\mu\right)\left\langle h-\psi-bg_{0},h-\psi\right\rangle _{L^{2}\left(J\right)},\label{eq:Lagr_cond_eM_impl}
\end{eqnarray}
 and therefore 
\[
e^{1/2}\left(\mu\right)\left\Vert g_{0}\right\Vert _{L^{2}\left(I\right)}\geq\left|\left\langle f-\psi-bg_{0},bg_{0}\right\rangle _{L^{2}\left(I\right)}\right|\geq\left(1+\mu\right)M_{0}\left(M_{0}-\left\Vert h-\psi\right\Vert _{L^{2}\left(J\right)}\right).
\]
Now, since $M_{0}\nearrow\infty$ as $\mu\searrow-1$ (because of (\ref{eq:e_mu_lowlimit}) and Proposition \ref{prop:trace_res1}),
the first term is dominant, and thus the right-hand side remains positive.
Then, because of (\ref{eq:e_mu_lowlimit}) and finiteness of $\left\Vert g_{0}\right\Vert _{L^{2}\left(I\right)}$
(by the triangle inequality, $\left\Vert g_{0}\right\Vert _{L^{2}\left(I\right)}\leq e^{1/2}\left(\mu\right)+\left\Vert \psi-f\right\Vert _{L^{2}\left(I\right)}$),
we conclude that 
\begin{equation}
\left(\mu+1\right)M_{0}^{2}\searrow0\hspace{1em}\text{as}\hspace{1em}\mu\searrow-1,\label{eq:mu_lowlimit}
\end{equation}
which allows us to deduce (\ref{eq:propM0}) from (\ref{eq:e_mu_integral}). 
\end{proof}
\begin{rem}
In the simplified case with no pointwise interpolation conditions (or those of zero-values) and no information on $J$,
the conclusion of the Proposition can be strengthened to 
\begin{equation}
\left\Vert M_{0}\right\Vert _{L^{2}\left(-1,\infty\right)}:=\left(\int_{-1}^{\infty}M_{0}^{2}\left(\mu\right)d\mu\right)^{1/2}=\left\Vert f\right\Vert _{L^{2}\left(I\right)},\label{eq:M0_norm}
\end{equation}
a result that was given in \cite{Alpay}. This mainly relies on
the fact that, for $\psi\equiv0$ and $h\equiv0$, 
\begin{equation}
g_{0}\rightarrow0\hspace{1em}\text{in}\hspace{1em}L^{2}\left(\mathbb{T}\right)\hspace{1em}\text{as}\hspace{1em}\mu\nearrow\infty,\label{eq:g0_uplimit}
\end{equation}
which holds by the following argument. Denoting $\tilde{f}:=P_{+}\left(\bar{b}f\vee0\right)$,
the solution formulas (\ref{eq:g0_sol_fin}) and (\ref{eq:g0_pre_sol_fin})
become $g_{0}=\left(1+\mu\phi\right)^{-1}\tilde{f}$ and $\mu\phi g_{0}=\tilde{f}-g_{0}$,
respectively. From these, as $\mu\nearrow\infty$, using the spectral
theorem (see Appendix), we obtain 
\[
\left\Vert \phi g_{0}\right\Vert _{H^{2}}=\dfrac{1}{\mu}\left\Vert \tilde{f}-g_{0}\right\Vert _{H^{2}}\leq\dfrac{1}{\mu}\left\Vert f\right\Vert _{L^{2}\left(I\right)}\left[1+\left\Vert \left(1+\mu\phi\right)^{-1}\right\Vert \right]\leq\dfrac{2}{\mu}\left\Vert f\right\Vert _{L^{2}\left(I\right)}\searrow0,
\]
 and hence, by Proposition \ref{prop:phi_inject}, conclude that $\left\Vert g_{0}\right\Vert _{H^{2}}\searrow0$.
We also need to show that 
\begin{equation}
\left(\mu+1\right)M_{0}^{2}\searrow0\hspace{1em}\text{as}\hspace{1em}\mu\nearrow\infty , \label{eq:mu_uplimit}
\end{equation}
but this follows from the positivity $\left(\mu+1\right)M_{0}^{2}>0$
and the observation that, for large enough $\mu$, we have 
\[
\frac{d\left[\left(\mu+1\right)M_{0}^{2}\right]}{d\mu}=M_{0}^{2}+\left(\mu+1\right)\frac{dM_{0}^{2}}{d\mu}<0
\]
(the inequality holds since, due to (\ref{eq:M_monot}), the second
term in the right-hand side is strictly negative whereas the first one goes to zero as $\mu$
increases). Finally, further elaboration of (\ref{eq:Lagr_cond_eM_impl})
into 
\[
e\left(\mu\right)+\left(1+\mu\right)M_{0}^{2}\left(\mu\right)=\left\langle \psi+bg_{0}-f,\psi-f\right\rangle _{L^{2}\left(I\right)}+\left(1+\mu\right)\left\langle \psi+bg_{0}-h,\psi-h\right\rangle _{L^{2}\left(J\right)}
\]
yields, in the case $\psi\equiv0$, $h\equiv0$, 
\[
e\left(\mu\right)+\left(1+\mu\right)M_{0}^{2}\left(\mu\right)=\left\langle f-bg_{0},f\right\rangle _{L^{2}\left(I\right)},
\]
which, by (\ref{eq:g0_uplimit})-(\ref{eq:mu_uplimit}), furnishes
$\underset{\mu\rightarrow\infty}{\lim}e\left(\mu\right)=\left\Vert f\right\Vert _{L^{2}\left(I\right)}^{2}$,
and hence (\ref{eq:M0_norm}) follows from (\ref{eq:e_mu_integral}) recalling
again (\ref{eq:e_mu_lowlimit}) and (\ref{eq:mu_lowlimit}).
\end{rem}

\subsection{Sharper estimates}
Precise asymptotic estimates near $\mu=-1$ were obtained in \cite{Baratchart-Grimm}
using concrete spectral theory of Toeplitz operators \cite{Rosenblum,Rosenblum-Rovnyak}.
Namely, under some specific regularity assumptions on the boundary
data $f$ (related to integrability of the first derivative on $I$),
we have 
\begin{equation}
M_{0}^{2}\left(\mu\right)={\cal O}\left(\left(1+\mu\right)^{-1}\log^{-2}\left(1+\mu\right)\right),\hspace{1em}e\left(\mu\right)={\cal O}\left(\left|\log^{-1}\left(1+\mu\right)\right|\right)\hspace{1em}\hspace{1em}\text{as\hspace{1em}}\mu\searrow-1.\label{eq:e_M0_asympt}
\end{equation}

Here we suggest a way of \textit{a
priori} estimation of approximation rate and error bounds without
resorting to an iterative solution procedure. This is based on a Neumann-like expansion of the inverse Toeplitz operator which provides series representations for the quantities $e\left(\mu\right)$ and $M_{0}^{2}\left(\mu\right)$ for values of $\mu$ moderately greater than $-1$ and, therefore, complements
previously obtained estimates of the asymptotic behavior of these quantities in the vicinity of $\mu=-1$.
Moreover, using these series expansions, we further attempt to recover the estimates (\ref{eq:e_M0_asympt})
without having concrete spectral theory involved, yet still appealing
to some general spectral theory results.

It is convenient to introduce the quantity
\begin{equation}
\xi\left(\mu\right):=\phi g_{0}\left(\mu\right)+P_{+}\left(0\vee\bar{b}\left(\psi-h\right)\right)\label{eq:xi_mu_def}
\end{equation}
that enters equation (\ref{eq:dg0_dmu}). The main results will be obtained in terms of
\begin{equation}
\xi_{0}:=\xi\left(0\right)=\phi\left(P_{+}\left(\bar{b}\left(f-\psi\right)\vee\bar{b}\left(h-\psi\right)\right)\right)-P_{+}\left(0\vee\bar{b}\left(h-\psi\right)\right).\label{eq:xi_mu0_def}
\end{equation}

\begin{prop}
For $\left|\mu\right|<1$, we have 
\begin{equation}
M_{0}^{2}\left(\mu\right)=M_{0}^{2}\left(0\right)-\sum_{k=0}^{\infty}\left(-1\right)^{k}\left(k+2\right)F\left(k\right)\mu^{k+1},\label{eq:M_series}
\end{equation}
\begin{equation}
e\left(\mu\right)=e\left(0\right)+2\sum_{k=0}^{\infty}\left(-1\right)^{k}F\left(k\right)\mu^{k+1}+\sum_{k=1}^{\infty}\left(-1\right)^{k}k\left[F\left(k\right)-F\left(k-1\right)\right]\mu^{k+1},\label{eq:e_series}
\end{equation}
where \textup{$F\left(k\right):=\left\langle \phi^{k}\xi_{0},\xi_{0}\right\rangle _{L^{2}\left(\mathbb{T}\right)}$,
$k\in\mathbb{N}_{+}$.}\end{prop}
\begin{proof}
Consider, for $k\in\mathbb{N}_{+}$, $\mu>-1$,

\[
A_{k}\left(\mu\right):=\left\langle \left(1+\mu\phi\right)^{-k}\phi^{k-1}\xi\left(\mu\right),\xi\left(\mu\right)\right\rangle _{L^{2}\left(\mathbb{T}\right)}.
\]
Since $\xi^{\prime}\left(\mu\right)=\phi\dfrac{dg_{0}}{d\mu}=-\left(1+\mu\phi\right)^{-1}\phi\xi\left(\mu\right)$
(according to (\ref{eq:dg0_dmu})), it follows that
\begin{eqnarray*}
A_{k}^{\prime}\left(\mu\right) & = & -k\left\langle \left(1+\mu\phi\right)^{-k-1}\phi^{k}\xi\left(\mu\right),\xi\left(\mu\right)\right\rangle _{L^{2}\left(\mathbb{T}\right)}-\left\langle \left(1+\mu\phi\right)^{-k-1}\phi^{k}\xi\left(\mu\right),\xi\left(\mu\right)\right\rangle _{L^{2}\left(\mathbb{T}\right)}\\
 &  & -\left\langle \left(1+\mu\phi\right)^{-k}\phi^{k-1}\xi\left(\mu\right),\left(1+\mu\phi\right)^{-1}\phi\xi\left(\mu\right)\right\rangle _{L^{2}\left(\mathbb{T}\right)},
\end{eqnarray*}
and we thus arrive at the infinite-dimensional linear dynamical system
\begin{equation}
\begin{cases}
A_{k}^{\prime}\left(\mu\right) & =-\left(k+2\right)A_{k+1}\left(\mu\right),\\
A_{k}\left(0\right) & =\left\langle \phi^{k-1}\xi_{0},\xi_{0}\right\rangle _{L^{2}}=:F\left(k-1\right),
\end{cases}\hspace{2em}k\in\mathbb{N}_{+}.\label{eq:dyn_sys}
\end{equation}
Introduce the matrix $\cal{M}$ whose powers are upper-diagonal with evident structure
\[
\mathcal{M}=\left(\begin{array}{ccccc}
0 & -3 & 0 & 0 & \dots\\
0 & 0 & -4 & 0 & \dots\\
0 & 0 & 0 & -5 & \dots\\
0 & 0 & 0 & 0 & \dots\\
\dots & \dots & \dots & \dots & \dots
\end{array}\right),\hspace{0.5em}\mathcal{M}^{2}=\left(\begin{array}{ccccc}
0 & 0 & \left(-3\right)\left(-4\right) & 0 & \dots\\
0 & 0 & 0 & \left(-4\right)\left(-5\right) & \dots\\
0 & 0 & 0 & 0 & \dots\\
0 & 0 & 0 & 0 & \dots\\
\dots & \dots & \dots & \dots & \dots
\end{array}\right),\hspace{0.5em}\dots,
\] which makes the matrix exponential $e^{\mathcal{M}}$ easily computable.
Then, due to such a structure, the system (\ref{eq:dyn_sys})
is readily solvable, but of particular interest is the first component
of the solution vector 
\begin{eqnarray*}
A_{1}\left(\mu\right) & = & \sum_{k=1}^{\infty}\left[e^{\mathcal{M}\mu}\right]_{1,k}F\left(k-1\right)=\sum_{k=0}^{\infty}\left(-1\right)^{k}\dfrac{\left(k+2\right)!}{2}\frac{\mu^{k}}{k!}F\left(k\right),
\end{eqnarray*}
where the series converges for $\left|\mu\right|<1$ since $F\left(k\right)$
is  bounded by $\left\Vert \xi_{0}\right\Vert _{H^{2}}^{2} = A_{1}\left(0\right) = F\left(0\right)$,
as the Toeplitz operator $\phi$ is a contraction: $F\left(k\right)$ slowly decays
to zero with $k$ (see also plots and discussion at the end of Section \ref{sec:num_res}).\\
On the other hand, observe that, due to (\ref{eq:M_monot}), $A_{1}\left(\mu\right)=-\dfrac{1}{2}\dfrac{dM_{0}^{2}}{d\mu}$
and thus 
\begin{equation}
\dfrac{dM_{0}^{2}}{d\mu}=-\sum_{k=0}^{\infty}\left(-1\right)^{k}\left(k+1\right)\left(k+2\right)\mu^{k}F\left(k\right).\label{eq:dM_series}
\end{equation}
Finally, termwise integration of (\ref{eq:dM_series}) and use of
(\ref{eq:e-M_connect}) followed by rearrangement of terms furnish
the results (\ref{eq:M_series})-(\ref{eq:e_series}).\end{proof}
\begin{rem}
Note that when set $\psi\equiv0$, $h\equiv0$, it is seen that (\ref{eq:dM_series})
can be obtained directly from (\ref{eq:g0_sol_fin}), (\ref{eq:M_monot}) which now reads
\[
\dfrac{dM_{0}^{2}}{d\mu}=-2\text{Re}\left\langle \left(1+\mu\phi\right)^{-3}\phi^{2}P_{+}\left(\bar{b}f\vee0\right),P_{+}\left(\bar{b}f\vee0\right)\right\rangle _{L^{2}\left(\mathbb{T}\right)}.
\]
The result follows since a Neumann series (defining an analytic function for
$\left|\mu\right|<1$) is differentiable:
\[
\left(1+\mu\phi\right)^{-1}=\sum_{k=0}^{\infty}\left(-1\right)^{k}\mu^{k}\phi^{k}\hspace{1em}\Rightarrow\hspace{1em}\left(1+\mu\phi\right)^{-3}=\dfrac{1}{2}\sum_{k=0}^{\infty}\left(-1\right)^{k}\left(k+1\right)\left(k+2\right)\mu^{k}\phi^{k}.
\]
 
\end{rem}
We can also get some insight in behavior of $F\left(k\right)$ which
lies in the heart of the series expansions (\ref{eq:M_series})-(\ref{eq:e_series})
that will allow us to infer the bounds (\ref{eq:e_M0_asympt}). First,
we need the following
\begin{lem}
\label{lem:Fk_sum_fin} The sequence $\left\{F\left(k\right)\right\}_{k=0}^{\infty}$ is Abel summable%
\footnote{By such summability we mean that $\sum_{k=0}^{\infty}\mu^{k}F\left(k\right)$
converges for all $\left|\mu\right|<1$ and the limit $\underset{\mu\nearrow1}{\lim}\sum_{k=0}^{\infty}\mu^{k}F\left(k\right)$
exists and is finite.%
} and it holds true that
\begin{equation}
\lim_{\mu\rightarrow-1}\sum_{k=0}^{\infty}\left(-\mu\right)^{k}F\left(k\right)=e\left(0\right)<\infty.\label{eq:Fk_sum}
\end{equation}
\end{lem}
\begin{proof}
Set $R_{\mu}\left(N\right):=\sum_{k=1}^{N}\left[F\left(k\right)-F\left(k-1\right)\right]k\left(-\mu\right)^{k}$
and apply summation by parts formula
\[
R_{\mu}\left(N\right)=F\left(N\right)\left(N+1\right)\left(-\mu\right)^{N+1}+\mu F\left(0\right)-\sum_{k=1}^{N}F\left(k\right)\left(\left(-\mu\right)^{k+1}\left(k+1\right)-\left(-\mu\right)^{k}k\right).
\]
Passing to the limit and rearranging the terms, we obtain 
\[
\lim_{N\rightarrow\infty}R_{\mu}\left(N\right)=-\sum_{k=0}^{\infty}\left(-\mu\right)^{k+1}F\left(k\right)+\left(\mu+1\right)\sum_{k=1}^{\infty}\left(-\mu\right)^{k}kF\left(k\right),
\]
and hence it follows from (\ref{eq:e_series}) that
\begin{equation}
e\left(\mu\right)=e\left(0\right)+\left(\mu+2\right)\sum_{k=0}^{\infty}\left(-\mu\right)^{k+1}F\left(k\right)+\left(\mu+1\right)\sum_{k=1}^{\infty}\left(-\mu\right)^{k+1}kF\left(k\right).\label{eq:e_series_alt}
\end{equation}

Combining the local integrability of $M_{0}^{2}\left(\mu\right)$, equivalent to (\ref{eq:mu_lowlimit}), with the series expansion (\ref{eq:M_series}),
we conclude that:
\[
\left(\mu+1\right)\sum_{k=1}^{\infty}\left(-\mu\right)^{k}kF\left(k\right)\rightarrow0 \mbox{ as } \mu\searrow-1 \, . 
\]
Therefore, taking the limit $\mu\searrow-1$
in (\ref{eq:e_series_alt}), the result (\ref{eq:Fk_sum}) follows
due to (\ref{eq:e_mu_lowlimit}). 
\end{proof}
Now, without getting into detail of concrete spectral theory of Toeplitz
operators, we only employ existence of a unitary transformation $U:\, H^{2}\rightarrow L_{\lambda}^{2}\left(\sigma\right)$
onto the spectral space where the Toeplitz operator is diagonal, meaning
that its action simply becomes a multiplication by the spectral variable
$\lambda$. Existence of such an isometry along with information on the
spectrum of $\phi$ (Hartman-Wintner theorem, see Appendix), $\sigma=\left[0,1\right]$,
and an assumption on the constant spectral density%
\footnote{Such an assumption is reasonable since the operator symbol $\chi_{J}$
is the simplest in a sense that it does not differ from one point
to another in the region where it is non-zero and therefore the spectral
mapping is anticipated to be uniform. Precise expression for the constant
$\rho_{0}$ can be found in \cite{Baratchart-Grimm,Rosenblum}.%
} $\rho_{0}>0$ make the following representation possible 
\begin{equation}
F\left(k\right)=\int_{0}^{1}\lambda^{k}\left|\left(U\xi_{0}\right)\left(\lambda\right)\right|^{2}\rho_{0}d\lambda\label{eq:Fk_spec_repr}
\end{equation}
with $\int_{0}^{1}\left|\left(U\xi_{0}\right)\left(\lambda\right)\right|^{2}\rho_{0}d\lambda=\left\Vert \xi_{0}\right\Vert _{H^{2}}^{2}.$

All the essential information on asymptotics (\ref{eq:e_M0_asympt})
is contained in behavior of $\left(U\xi_{0}\right)\left(\lambda\right)$
near $\lambda=1$. Even though $\left(U\xi_{0}\right)\left(\lambda\right)$
can be computed since $\xi_{0}$ is a fixed function defined by
(\ref{eq:xi_mu0_def}) and the concrete spectral theory describes
explicit action of the transformation $U$ \cite{Baratchart-Grimm,Rosenblum-Rovnyak},
we avoid these details and proceed by deriving essential estimates
invoking only rather intuitive arguments on the behavior of the resulting
function $\left(U\xi_{0}\right)\left(\lambda\right)$.

Considering $-1<\mu<0$ in what follows, we, first of all, claim that
the function $\left(U\xi_{0}\right)\left(\lambda\right)$ must necessarily
decrease to zero as $\lambda\nearrow1$. Indeed, even though $L^{2}$-behavior
allows to have an integrable singularity at $\lambda=1$, we note
that even if regularity was assumed, that is $\underset{\lambda\rightarrow1}{\lim}\left|\left(U\xi_{0}\right)\left(\lambda\right)\right|^{2}=C$
for some $C>0$, after summation of a geometric series, we would have
\[
\dfrac{1}{\rho_{0}}\sum_{k=0}^{\infty}\left(-\mu\right)^{k}F\left(k\right)\geq C_{0}\sum_{k=0}^{\infty}\int_{1-\delta}^{1}\left(-\mu\lambda\right)^{k}d\lambda=C_{0}\int_{1-\delta}^{1}\frac{1}{1+\mu\lambda}d\lambda=\frac{C_{0}}{\mu}\log\left(\frac{1+\mu}{1+\mu-\mu\delta}\right)
\]
for some $0<C_{0}\leq C$ and sufficiently small fixed $\delta>0$.
The right-hand side here grows arbitrary large as $\mu$ comes closer
to $-1$ contradicting the boundedness prescribed by Lemma \ref{lem:Fk_sum_fin}.
Therefore, the decay to zero of $\left(U\xi_{0}\right)\left(\lambda\right)$
as $\lambda\nearrow1$ is necessary. 

Next, it is natural to proceed by checking if a very mild (meaning
slower than any power) decay to zero can be reconciled with the previously
obtained results. Namely, we consider $\left(U\xi_{0}\right)\left(\lambda\right)$
such that 
% \begin{equation}
% \left|\left(U\xi_{0}\right)\left(\lambda\right)\right|^{2}={\cal O}\left(\tau\left(\lambda\right)\left|\log\left(1-\lambda\right)\right|^{-l}\right)\hspace{1em}\text{as}\hspace{1em}\lambda\nearrow1\label{eq:U_log_gen}
% \end{equation}
\begin{equation}
\left|\left(U\xi_{0}\right)\left(\lambda\right)\right|^{2}={\cal O}\left(\left|\log\left(1-\lambda\right)\right|^{-l}\right)\hspace{1em}\text{as}\hspace{1em}\lambda\nearrow1,\label{eq:U_log}
\end{equation}
for $l>1$.
%  and some continuous function $\tau\left(\lambda\right)$
% with $\tau\left(1\right)\neq0$
This entails the following result generalizing (\ref{eq:e_M0_asympt}), see also Remarks \ref{rem:l2}, \ref{rem:l22}.
\begin{prop}
Under assumption (\ref{eq:U_log}) with $l >1$, the solution blow-up
and approximation rates near $\mu=-1$, respectively, are as follows 
\begin{equation}
M_{0}^{2}\left(\mu\right)={\cal O}\left(\frac{1}{1+\mu}\left|\log\left(1+\mu\right)\right|^{-l}\right),\hspace{1em}e\left(\mu\right)={\cal O}\left(\left|\log\left(1+\mu\right)\right|^{-l+1}\right).\label{eq:M0_e_asympt_gen}
\end{equation}
\end{prop}
\begin{proof}
Choose a constant $0<\lambda_{0}<1$ sufficiently close to $1$ so
that the asymptotic (\ref{eq:U_log}) is applicable. % Without loss
% of generality, instead of (\ref{eq:U_log_gen}), it is sufficient
% to consider 
%This simplification is 
% Due to the integral mean-value theorem:
% \[
% \int_{\lambda_{0}}^{1}h\left(\lambda\right)\left|\left(U\xi_{0}\right)\left(\lambda\right)\right|^{2}d\lambda=\int_{\lambda_{0}}^{1}\tau\left(\lambda\right)h\left(\lambda\right)\left|\log\left(1-\lambda\right)\right|^{-l}d\lambda=\tau\left(\tilde{\lambda}\right)\int_{\lambda_{0}}^{1}h\left(\lambda\right)\left|\log\left(1-\lambda\right)\right|^{-l}d\lambda
% \]
% for some $\tilde{\lambda}\in\left[\lambda_{0},1\right]$ and any given
% $h\in L^{\infty}\left(\lambda_{0},1\right)$.
Therefore, we can write
\[
\dfrac{1}{\rho_{0}}\sum_{k=0}^{\infty}\left(-\mu\right)^{k}F\left(k\right) 
 =  S_{1}+S_{2}+S_{3}
\]
\[
:= \int_{0}^{\lambda_{0}}\frac{1}{1+\mu\lambda}\left|\left(U\xi_{0}\right)\left(\lambda\right)\right|^{2}d\lambda+\left(\int_{\lambda_{0}}^{1-\delta_{0}}+\int_{1-\delta_{0}}^{1}\right)\frac{1}{1+\mu\lambda}\left(-\log\left(1-\lambda\right)\right)^{-l}d\lambda \, .
\]
The first integral here is bounded regardless of the value of $\mu$:
\[
S_{1}\leq\frac{1}{1+\mu\lambda_{0}}\int_{0}^{1}\left|\left(U\xi_{0}\right)\left(\lambda\right)\right|^{2}d\lambda=\frac{1}{\left(1+\mu\lambda_{0}\right)\rho_{0}}\left\Vert \xi_{0}\right\Vert _{H^{2}}^{2}.
\]
To deal with $S_{3}$, we perform the change of variable $\beta=-\log\left(1-\lambda\right)$
and bound the factor $\dfrac{1}{\beta^{l}}\leq\left(-\log\delta_{0}\right)^{-l}$
to obtain
\[
\int_{-\log\delta_{0}}^{\infty}\frac{1}{\beta^{l}}\frac{e^{-\beta}}{1+\mu-\mu e^{-\beta}}d\beta\leq\frac{1}{\left(-\mu\right)\left(-\log\delta_{0}\right)^{l}}\log\left(1-\frac{\mu\delta_{0}}{1+\mu}\right)\leq\frac{\log2}{\left(-\mu\right)\left|\log\left(1+\mu\right)-\log\left(-\mu\right)\right|^{l}}
\]
providing we choose $\delta_{0}\leq\dfrac{1+\mu}{\left(-\mu\right)}$.
The quantity on the right is ${\cal O}\left(\left|\log\left(1+\mu\right)\right|^{-l}\right)$
in the vicinity of $\mu=-1$.\\
It remains to estimate $S_{2}$. The change of variable $\eta=1-\lambda$
leads to 
\begin{eqnarray*}
S_{2} & = & \int_{\delta_{0}}^{1-\lambda_{0}}\frac{\eta}{1+\mu-\mu\eta}\frac{1}{\eta\left(-\log\eta\right)^{l}}d\eta\leq\left(\int_{\delta_{0}}^{1-\lambda_{0}}\frac{d\eta}{\eta\left(-\log\eta\right)^{l}}\right)\underset{\eta\in\left[\delta_{0},1-\lambda_{0}\right]}{\sup}\left(\frac{\eta}{1+\mu-\mu\eta}\right)\\
%  & = & \frac{1}{l-1}\left(\frac{1}{\left|\log\delta_{0}\right|^{l-1}}-\frac{1}{\left|\log\left(1-\lambda_{0}\right)\right|^{l-1}}\right)\frac{1-\lambda_{0}}{1+\mu\lambda_{0}}\\
& \leq & \frac{1}{l-1}\left(\frac{1}{\left|\log\delta_{0}\right|^{l-1}}-\frac{1}{\left|\log\left(1-\lambda_{0}\right)\right|^{l-1}}\right)\frac{1-\lambda_{0}}{1+\mu\lambda_{0}}.
\end{eqnarray*}
Therefore, we conclude that the choice (\ref{eq:U_log}) with
$l>1$ does not contradict the finiteness imposed by Lemma \ref{lem:Fk_sum_fin}
anymore and we move on to obtain the growth rate for $M_{0}^{2}\left(\mu\right)$
near $\mu=-1$. Recalling (\ref{eq:M_series}) and that $\sum_{k=0}^{\infty}\left(-\mu\lambda\right)^{k}\left(k+1\right)=\dfrac{1}{\left(1+\mu\lambda\right)^{2}}$,
we now have
\[
\dfrac{1}{\rho_{0}}\sum_{k=0}^{\infty}\left(-\mu\right)^{k}\left(k+1\right)F\left(k\right) 
=  R_{1}+R_{2}+R_{3}+R_{4}
\]
\[
:=  \int_{0}^{\lambda_{0}}\frac{1}{\left(1+\mu\lambda\right)^{2}}\left|\left(U\xi_{0}\right)\left(\lambda\right)\right|^{2}d\lambda+\left(\int_{\lambda_{0}}^{1-\delta_{1}}+\int_{1-\delta_{1}}^{1-\delta_{2}}+\int_{1-\delta_{2}}^{1}\right)\frac{1}{\left(1+\mu\lambda\right)^{2}}\left(-\log\left(1-\lambda\right)\right)^{-l}d\lambda \, .
\]
As before, we estimate
\[
R_{1}\leq\frac{1}{\left(1+\mu\lambda_{0}\right)^{2}}\int_{0}^{1}\left|\left(U\xi_{0}\right)\left(\lambda\right)\right|^{2}d\lambda=\frac{1}{\left(1+\mu\lambda_{0}\right)^{2}\rho_{0}}\left\Vert \xi_{0}\right\Vert _{H^{2}}^{2},
\]
whereas the rest is now split into 3 parts and we start with the last term
and decide on proper size of $\delta_{2}$ in
\[
R_{4}=\int_{1-\delta_{2}}^{1}\frac{1}{\left(1+\mu\lambda\right)^{2}}\frac{1}{\left|-\log\left(1-\lambda\right)\right|^{l}}d\lambda.
\]
Again, under the change of variable $\beta=-\log\left(1-\lambda\right)$,
this becomes
%\[
%R_{4}=\frac{1}{\left(1+\mu\right)^{2}}\int_{-\log\delta_{2}}^{\infty}\frac{e^{-\beta}}{\beta^{l}}\frac{1}{\left(1-\frac{\mu}{1+\mu}e^{-\beta}\right)^{2}}d\beta=\dfrac{1}{\left(1+\mu\right)^{2}}\sum_{k=0}^{\infty}\left(\frac{\mu}{1+\mu}\right)^{k}\left(k+1\right)\int_{-\log\delta_{2}}^{\infty}\frac{e^{-\left(k+1\right)\beta}}{\beta^{l}}d\beta,
%\]
\begin{eqnarray*}
R_{4}&=&\frac{1}{\left(1+\mu\right)^{2}}\int_{-\log\delta_{2}}^{\infty}\frac{e^{-\beta}}{\beta^{l}}\frac{1}{\left(1-\frac{\mu}{1+\mu}e^{-\beta}\right)^{2}}d\beta\\
&=&\dfrac{1}{\left(1+\mu\right)^{2}}\sum_{k=0}^{\infty}\left(\frac{\mu}{1+\mu}\right)^{k}\left(k+1\right)\int_{-\log\delta_{2}}^{\infty}\frac{e^{-\left(k+1\right)\beta}}{\beta^{l}}d\beta,
\end{eqnarray*}
where the series expansion is valid for $\delta_{2}<\dfrac{1+\mu}{\left(-\mu\right)}$.
The integral on the right is the incomplete gamma function (see, for
instance, \cite{Abramowitz-Stegun}) whose asymptotic expansion for
large values of $\left(-\log\delta_{2}\right)$ can be easily obtained
with integration by parts. In particular, at the leading order we
have 
\begin{eqnarray*}
\int_{-\log\delta_{2}}^{\infty}\frac{e^{-\left(k+1\right)\beta}}{\beta^{l}}d\beta&=&\left(k+1\right)^{l-1}\int_{-\left(k+1\right)\log\delta_{2}}^{\infty}\frac{e^{-\beta}}{\beta^{l}}d\beta\\
&=&\left(k+1\right)^{l-1}\delta_{2}^{k+1}\left(-\left(k+1\right)\log\delta_{2}\right)^{-l}\left[1+{\cal O}\left(\frac{1}{\left(k+1\right)\left|\log\delta_{2}\right|}\right)\right],
\end{eqnarray*}
and hence
\[
R_{4}=\frac{\delta_{2}}{\left(1+\mu\right)^{2}}\frac{1}{\left(-\log\delta_{2}\right)^{l}}\sum_{k=0}^{\infty}\left(\frac{\mu\delta_{2}}{1+\mu}\right)^{k}=\frac{\delta_{2}}{\left(1+\mu\right)^{2}}\frac{1}{\left(-\log\delta_{2}\right)^{l}}\frac{1}{1-\frac{\mu\delta_{2}}{1+\mu}}.
\]
Fixing $\delta_{2}=\dfrac{1}{2}\dfrac{1+\mu}{\left(-\mu\right)}$,
we arrive at
\[
R_{4}=\frac{1}{\left(-\mu\right)\left(1+\mu\right)}\frac{1}{\left[-\log\left(1+\mu\right)+\log\left(-\mu\right)+\log 2\right]^{l}}.
\]
To estimate $R_{2}$ and $R_{3}$, we use change of variable $\eta=1-\lambda$.
Similarly to $S_{2}$, we have
\[
R_{2}=\int_{\delta_{1}}^{1-\lambda_{0}}\frac{\eta}{\left(1+\mu-\mu\eta\right)^{2}}\frac{1}{\eta\left(-\log\eta\right)^{l}}d\eta\leq\left(\int_{\delta_{1}}^{1-\lambda_{0}}\frac{d\eta}{\eta\left(-\log\eta\right)^{l}}\right)\underset{\eta\in\left[\delta_{1},1-\lambda_{0}\right]}{\sup}\left(\frac{\eta}{\left[1+\mu-\mu\eta\right]^{2}}\right),
\]
however, now under the supremum sign, instead of a monotonic function,
we have an expression that attains a maximum value $\dfrac{1}{4\left(-\mu\right)\left(1+\mu\right)}$
if $\delta_{1}<\dfrac{1+\mu}{\left(-\mu\right)}$ which lacks the
smallness we obtained in $R_{4}$. Therefore, to remedy the situation,
we require $\delta_{1}>\dfrac{1+\mu}{\left(-\mu\right)}$ and obtain
\[
R_{2}\leq\frac{1}{l-1}\left(\frac{1}{\left|\log\delta_{1}\right|^{l-1}}-\frac{1}{\left|\log\left(1-\lambda_{0}\right)\right|^{l-1}}\right)\frac{\delta_{1}}{\left(1+\mu-\mu\delta_{1}\right)^{2}}={\cal O}\left(\frac{1}{1+\mu}\left|\log\left(1+\mu\right)\right|^{-\gamma}\right)
\]
near $\mu=-1$, if we fix $\delta_{1}=\dfrac{1+\mu}{\left(-\mu\right)}\left(1+\left[-\log\left(1+\mu\right)\right]^{\gamma}\right)$
for arbitrary $\gamma>0$.\\
The last integral $R_{3}$ is to bridge the gap between the two neighborhoods
of $\lambda=1$:
\[
R_{3}=\int_{\delta_{2}}^{\delta_{1}}\frac{1}{\left(1+\mu-\mu\eta\right)^{2}}\frac{1}{\left(-\log\eta\right)^{l}}d\eta\leq\frac{1}{\left(-\log\delta_{1}\right)^{l}}\left(\frac{1}{1+\mu-\mu\delta_{2}}-\frac{1}{1+\mu-\mu\delta_{1}}\right)
\]
and hence, using the fact that $\log\left(-\log\left(1+\mu\right)\right)=o\left(-\log\left(1+\mu\right)\right)$,
we deduce that near $\mu=-1$
\[
R_{3}={\cal O}\left(\frac{1}{1+\mu}\left|\log\left(1+\mu\right)\right|^{-l}\right).
\]
Now that all the integral terms are estimated, choice of the parameter
$\gamma=l$ in $\delta_{1}$ leads to the first estimate in (\ref{eq:M0_e_asympt_gen})
whereas integration of (\ref{eq:e-M_connect}) recovers the second
one. \end{proof}
\begin{rem}
\label{rem:l2}
The case $l=2$ gives exactly the expressions in (\ref{eq:e_M0_asympt}).
The assumed behavior \eqref{eq:U_log} of $\left(U\xi_{0}\right)\left(\lambda\right)$ is analogous (with direct correspondence in the case $\psi\equiv0$, $h\equiv0$) to the conclusion of \cite[Prop. 4.1]{Baratchart-Grimm}
which was used to generate further estimates therein, and the case
$l=3$ is related to improved estimates given in \cite[Cor. 4.6]{Baratchart-Grimm} under
assumption of even higher regularity of boundary data (roughly speaking,
integrability of second derivatives).
It is noteworthy that the choice
$l=1$ yields non-integrable behavior of $M_{0}^{2}\left(\mu\right)$
contradicting Proposition \ref{prop:M0_integr}, and therefore was
eliminated in the formulation. This is not due to the fact that the
method of estimation of the $S_{2}$ integral fails, but because of
non-integrability near $\mu=-1$ of the overall bound. 
The $R_{4}$
term has been computed asymptotically sharply though it could be made even 
smaller by shrinking the neighborhood $\delta_{2}$. Indeed, instead
of the $\dfrac{1}{2}$ factor in $\delta_{2}$, we could have put
$\dfrac{1}{1+\left[-\log\left(1+\mu\right)\right]^{\beta}}$ for any
$\beta\geq0$ similarly to what was done in the $R_{2}$ term which
allowed a multiplier with arbitrary logarithmical smallness regulated
by the parameter $\gamma$. This, however, would not reduce
the overall blow-up because of the stiff bridging term $R_{3}$. Even though the estimate for $R_{3}$ is rough, we do not expect
improvement by an order of magnitude because the logarithmic factor of
the integrand picks up $\left(1+\mu\right)$ as a major multiplier
near $\eta=\delta_{1}$ which makes any choice of $\gamma\geq l$
and $\beta\geq0$ useless in attempt to improve the smallness factor
in the blow-up of $M_{0}^{2}\left(\mu\right)$.
\end{rem}
$\hspace{1em}$
\begin{rem}
\label{rem:l22}
Generally, we note that the appearance of the $\log\left(1+\mu\right)$
factors in the bounds is not accident, but intrinsically encoded in
the connection between $e\left(\mu\right)$ and $M_{0}^{2}\left(\mu\right)$
since (\ref{eq:e-M_connect}) can be rewritten as $e^{\prime}\left(\mu\right)=-\dfrac{dM_{0}^{2}}{d\left[\log\left(1+\mu\right)\right]}$
which also explains the choice of (\ref{eq:U_log}).
\end{rem}
We would like to point out again that even though our reasoning was
meant to provide an intuitive explanation of the estimates (\ref{eq:e_M0_asympt}),
more rigourous proofs can be found in \cite{Baratchart-Grimm} where
an elegant connection of the bounds with regularity of given boundary
data is established by elaborating concrete spectral theory results
\cite{Rosenblum-Rovnyak} into formulation of a certain integral transformation
followed by application of $L^{1}$-theory of Fourier transforms (Riemann-Lebesgue
lemma). Also, one can take an alternative viewpoint based on the results
of \cite{Rosenblum}. In that case, the unitary transformation $U$
diagonalizing the Toeplitz operator $\phi$ acts on Fourier coefficients $\left\{ \eta_{n}\right\} _{n=0}^{\infty}\in l^{2}\left(\mathbb{N}_{0}\right)$
of a given $\xi_{0}\in H^{2}$ as 
\begin{equation}
\left(U\xi_{0}\right)\left(\lambda\right)=\sum_{n=0}^{\infty}\eta_{n}\psi_{n}\left(\lambda\right),\label{eq:U_orth_expans}
\end{equation}
where the orthonormal sequence of $L^{2}\left(0,1\right)$ functions
$\psi_{n}\left(\lambda\right)$ are explicitly defined in terms of
the Meixner-Pollaczek polynomials of order $1/2$ \cite{Olver}:
\[
\psi_{n}\left(\lambda\right):=e^{a\beta}\left(1+e^{-2\pi\beta}\right)^{1/2}P_{n}^{\left(1/2\right)}\left(\beta,a\right),\hspace{1em}\hspace{1em}\beta:=-\frac{1}{2\pi}\log\left(\dfrac{1}{\lambda}-1\right)
\]
providing $I=\left(e^{-ia},e^{ia}\right)$, $a\in\left(0,\pi\right)$,
an assumption that does not reduce the generality if the original
sets $I$ and $J$ are two disjoint arcs.

A recurrence formula for the Meixner-Pollaczek polynomials follows
from that for the Pollaczek polynomials \cite{Szego}:
\begin{equation}
P_{n}^{\left(1/2\right)}\left(\beta\right)=\frac{1}{n}\left(2\beta\sin a-\left(2n-1\right)\cos a\right)P_{n-1}^{\left(1/2\right)}\left(\beta\right)-\frac{n-1}{n}P_{n-2}^{\left(1/2\right)}\left(\beta\right),\label{eq:Pollaczek_pol}
\end{equation}
\[
P_{-1}^{\left(1/2\right)}\left(\beta\right)=0,\hspace{1em}P_{0}^{\left(1/2\right)}\left(\beta\right)=1,
\]
which allows to generate all the coefficients $k_{m}^{\left(n\right)}$
in $P_{n}^{\left(1/2\right)}\left(\beta\right)=\sum_{m=0}^{n}k_{m}^{\left(n\right)}\beta^{m}$,
for instance,
\[
k_{n}^{\left(n\right)}=\frac{\left(2\sin a\right)^{n}}{n!},\hspace{1em}k_{n-1}^{\left(n\right)}=-n\cos a\frac{\left(2\sin a\right)^{n-1}}{\left(n-1\right)!},
\]
\[k_{n-2}^{\left(n\right)}=\frac{1}{6}\left[3n\left(n-1\right)\cos^{2}a-\left(2n-1\right)\sin^{2}a\right]\frac{\left(2\sin a\right)^{n-2}}{\left(n-2\right)!}.
\]
Rearranging the terms in (\ref{eq:U_orth_expans}), we can write (suppressing
the first two factors for the sake of compactness)
\begin{eqnarray}
\left(U\xi_{0}\right)\left(\lambda\right)&\propto&\sum_{n=0}^{\infty}\eta_{n}\sum_{m=0}^{n}k_{m}^{\left(n\right)}\beta^{m}=\sum_{m=0}^{\infty}\left(\sum_{n=m}^{\infty}\eta_{n}k_{m}^{\left(n\right)}\right)\beta^{m}\nonumber \\
&=&\sum_{m=0}^{\infty}\left(\eta_{m}k_{m}^{\left(m\right)}+\eta_{m+1}k_{m}^{\left(m+1\right)}+\dots\right)\beta^{m}.\label{eq:U_orth_expans2}
\end{eqnarray}

It would be interesting to see, in such a representation, what decay
assumptions on the Fourier coefficients $\eta_{n}$ are consistent
with (\ref{eq:U_log}), and thus (\ref{eq:M0_e_asympt_gen}),
with $1<l<2$ in which case there is no violation of integrability
of $M_{0}^{2}\left(\mu\right)$ and less regularity assumptions (namely,
milder than decay of $n\eta_{n}$ to zero as $n\rightarrow\infty$)
are expected than those related with integrability of the first derivative
of boundary data.

Note that, because of the Taylor series of the exponential function,
we have
\begin{eqnarray*}
\left|\sum_{m=0}^{\infty}\left(\eta_{m}k_{m}^{\left(m\right)}\right)\beta^{m}\right|&\leq&\left(\sup_{m\in\mathbb{N}_{0}}\left|\eta_{m}\right|\right)\sum_{m=0}^{\infty}\frac{1}{m!}\left(\frac{\sin a}{\pi}\right)^{m}\left|\log\left(\dfrac{1}{\lambda}-1\right)\right|^{m}\\
&=&\left(\sup_{m\in\mathbb{N}_{0}}\left|\eta_{m}\right|\right)\begin{cases}
\left(\dfrac{1}{\lambda}-1\right)^{\dfrac{\sin a}{\pi}}, & 0<\lambda<\dfrac{1}{2},\\
\left(\dfrac{1}{\lambda}-1\right)^{-\dfrac{\sin a}{\pi}}, & \dfrac{1}{2}\le\lambda<1,\end{cases}
\end{eqnarray*}
and thus the very first term already adds to the singular behavior
of (\ref{eq:U_orth_expans}) near $\lambda=1$ (unless additional
assumptions on alternation of sign of $\eta_{m}$ are made) instead
of revealing any decay to zero. This suggests that terms in the brackets
of (\ref{eq:U_orth_expans2}) should not be estimated separately:
the other terms contribute equally to $\left(U\xi_{0}\right)\left(\lambda\right)$
though their expressions are much more cumbersome for straightforward
analysis.

An alternative way might be to work in direction of obtaining estimates
of (\ref{eq:M_series})-(\ref{eq:e_series}) near $\mu=-1$ in terms
of $\eta_{m}$ from 
\[
\int_{0}^{1}\dfrac{1}{1+\mu\lambda}\left|\left(U\xi_{0}\right)\left(\lambda\right)\right|^{2}d\lambda\hspace{1em} \mbox{ and }\hspace{1em}
\int_{0}^{1}\dfrac{1}{\left(1+\mu\lambda\right)^{2}}\left|\left(U\xi_{0}\right)\left(\lambda\right)\right|^{2}d\lambda 
\]
directly without deducing behavior of $\left(U\xi_{0}\right)\left(\lambda\right)$
in vicinity of $\lambda=1$, but using explicit form of the unitary
transformation (\ref{eq:U_orth_expans}). To take advantage of it,
one can potentially expand integrand factors $\dfrac{1}{1+\mu\lambda}$
in terms of $\beta$ and iteratively employ the recurrence formula
(\ref{eq:Pollaczek_pol}) rewritten as 
\[
\beta P_{n}^{\left(1/2\right)}\left(\beta,a\right)=\frac{n+1}{2\sin a}P_{n+1}^{\left(1/2\right)}\left(\beta,a\right)+\frac{\left(2n+1\right)\cot a}{2}P_{n}^{\left(1/2\right)}\left(\beta,a\right)+\frac{n}{2\sin a}P_{n-1}^{\left(1/2\right)}\left(\beta,a\right)
\]
followed by application of orthonormality. Note that such a strategy
(but based on expansion of $\lambda$ in terms of $\beta$) along
with the fact that $U^{-1}\psi_{n}\left(\lambda\right)=z^{n}$ might
also be used to see how the Toeplitz operator $\phi$ acts on Fourier
coefficients of a function.

\newpage
\section{\label{sec:compan_pb}Companion problem}

At this moment, it is time to point out a link with another bounded
extremal problem which relies on the observation that formal substitution
of $\mu=0$ in (\ref{eq:g0_sol_fin_alt_compl}) implies that 
\begin{equation}
\tilde{g}_{0}=\psi+bP_{+}\left(\bar{b}\left(f\vee h\right)\right)\label{eq:g0_sol_mu0}
\end{equation}
 is an explicit solution for the problem with the particular constraint 
\[
M=M_0\left(0\right)=\left\Vert \psi+bP_{+}\left(\bar{b}\left(f\vee h\right)\right)-h\right\Vert _{L^{2}\left(J\right)}.
\]

Recalling that $bP_{+}\bar{b}$ is a projector onto $bH^{2}$
(see Section \ref{sec:intro_Hardy}), we note that, geometrically,
the solution (\ref{eq:g0_sol_mu0}) is simply a realization of projection of $f\vee h\in L^{2}\left(\mathbb{T}\right)$
onto $\mathcal{C}_{M,h}^{\psi,b}$. Now, exploiting the arbitrariness
of choice of interpolant $\psi$ (Remark \ref{rem:sol_fin_alt}),
we can change our viewpoint and look for $\psi\in H^{2}$
meeting pointwise constraints (\ref{eq:meas_cond}) such that $\psi+bP_{+}\left(\bar{b}\left(f\vee h\right)\right)-h$
is sufficiently close to the constant\footnote{Alternatively, one can take any $L^2\left(J\right)$ function that has norm $M$.} $M/\sqrt{\left|J\right|}$ in $L^{2}\left(J\right)$ yet remaining
$L^{2}$-bounded on $I$. In other words, given arbitrary $\psi_{0}\in H^{2}$
satisfying the pointwise interpolation conditions (\ref{eq:meas_cond}) (take,
for instance, (\ref{eq:interp_expans})), we represent $\psi=\psi_{0}+b\Psi$
and thus search for an approximant $\Psi\in H^{2}$
to $\bar{b}\left(h-\psi_{0}+M\right)-P_{+}\left(\bar{b}\left(f\vee h\right)\right)\in L^{2}\left(J\right)$
such that $\left\Vert \Psi\right\Vert _{L^{2}\left(I\right)}=K$ for
arbitrary $K\in\left(0,\infty\right)$. We thus reduce the original problem to an associated approximation
problem on $J$ for which all known data are now prescribed on $J$
alone. Since the constraint on $I$ is especially simple (role of $\psi$ and $h$ play identically zero functions), % $"\psi"=0$, $"h"=0$
such a companion problem has
a computational advantage over the original one as, due to the form
of solution (\ref{eq:g0_sol_fin}), it requires integration only over
a subset of $\mathbb{T}$ (see (\ref{eq:num_S_term})).

To be more precise, let $\Psi_{0}$ be a solution to the companion
problem such that 
\[
\left\Vert \psi_{0}+b\Psi_{0}+bP_{+}\left(\bar{b}\left(f\vee h\right)\right)-h\right\Vert _{L^{2}\left(J\right)}^{2}=M^{2}+\delta_{M^{2}},
\]
where $\delta_{M^{2}}$ measures accuracy of the solution of the companion
problem. Then, solution to the original problem should be sought as
a series expansion near (\ref{eq:g0_sol_mu0}) with respect to $\delta_{M^{2}}$
as a small parameter
\begin{equation}
\tilde{g}_{0}=\psi_{0}-bP_{+}\left(\bar{b}\psi_{0}\right)+bP_{+}\left(\bar{b}\left(f\vee h\right)\right)+b\left.\frac{dg_{0}}{d\mu}\right|_{\mu=0}\left.\frac{d\mu}{dM_{0}^{2}}\right|_{M_{0}^{2}=M^{2}}\delta_{M^{2}}+\dots,\label{eq:orig_compan_pb_conn}
\end{equation}
and further the relations (\ref{eq:dg0_dmu})-(\ref{eq:M_monot})
followed by $\left.\dfrac{d\mu}{dM_{0}^{2}}\right|_{M_{0}^{2}=M^{2}}=\left.\left(\dfrac{dM_{0}^{2}}{d\mu}\right)^{-1}\right|_{\mu=0}$
should be employed (here $g_{0}$ is as in (\ref{eq:g0_sol_fin})).
Recalling Section \ref{sec:intro_Hardy}, we note that the first two
terms realize a projection of $\psi_{0}$ onto $\left(bH^{2}\right)^{\perp_{H^{2}}}$
which will be simply $\psi_{0}$ if (\ref{eq:interp_expans}) was
used as the arbitrary interpolant (see Remark \ref{rem:sol_fin_alt}). 

If the companion problem was solved with good accuracy so that $\delta_{M^{2}}$
is small, linear order approximation in $\delta_{M^{2}}$ may be sufficient
to recover the solution of the original problem. However, this connection
between solution of two problems is valid for arbitrary values of
$\delta_{M^{2}}$ if one considers infinite series in $\delta_{M^{2}}$.
This can be formalized with use of the Fa\`a di Bruno formula which
provides explicit form of the Taylor expansion for the function composition
$g_{0}\left(\mu\left(M_{0}^{2}\right)\right)$ in terms of the derivatives
$\left.\dfrac{d^{k}g_{0}}{\left(d\mu\right)^{k}}\right|_{\mu=0}$
and $\left.\dfrac{d^{k}\mu}{\left(dM_{0}^{2}\right)^{k}}\right|_{M_{0}^{2}=M^{2}}$
for any $k\in\mathbb{N}_{+}$. Applying the product rule and expression
(\ref{eq:dg0_dmu}) successively it can be shown that, after collection
of terms at each differentiation, we have
\[
\dfrac{d^{k}g_{0}}{\left(d\mu\right)^{k}}=\left(-1\right)^{k}k!\left(1+\mu\phi\right)^{-k}\phi^{k}\tilde{\xi}\hspace{1em}\Rightarrow\hspace{1em}\left.\dfrac{d^{k}g_{0}}{\left(d\mu\right)^{k}}\right|_{\mu=0}=\left(-1\right)^{k}k!\phi^{k}\tilde{\xi}_{0},
\]
where % (slightly abusing notation of $\phi$ for the sake of brevity)
\[
\tilde{\xi}:=P_+ \left(0 \vee \left(g_{0}+\bar{b}\left(\psi_{0}-h\right)+\Psi_{0}\right)\right),\hspace{1em}\tilde{\xi}_{0}:= \phi \left(P_{+}\left(\bar{b}\left(f\vee h\right)\right)-P_{+}\left(\bar{b}\psi_{0}\right)-\Psi_{0}\right).
\]

As far as computation of derivates of $\dfrac{d\mu}{dM_{0}^{2}}$
is concerned, complexity of the expressions grows and precise pattern
seem to be hard to find especially since implicit differentiation
has to be repeated every time resulting in successive appearance of
extra factor $\dfrac{d\mu}{dM_{0}^{2}}$. Even though in practice
one may look at the truncated Taylor expansion $M_{0}^{2}\left(\mu\right)$
and, since derivatives $\dfrac{dM_{0}^{2}}{d\mu}$ are readily computable,
use reversion of the series to obtain power series expansion of $\mu$
in terms of $M_{0}^{2}$ (for reversion of series coefficient formula,
see \cite{Morse-Feshbach}) or, alternatively, employ the Lagrange
inversion theorem that yields the inverse function $\mu\left(M_{0}^{2}\right)$
as %in a form of 
an infinite series, in the latter case we would have
to decide at which term the both series should be truncated so that
to preserve desired accuracy at given order of $\delta_{M^{2}}$.
For small $\delta_{M^{2}}$, only few terms are needed to give quite
accurate connection between solution of the original and companion
problems. Those can be precomputed manually or using computer algebra
systems once and such calculations need not be repeated iteratively.

\newpage
\section{\label{sec:stab_res}Stability results}

The issue to be discussed here is linear stability of the solution (\ref{eq:g0_sol_def})
with respect to all physical components that the expression (\ref{eq:g0_sol_fin})
involves explicitly and implicitly. In practice, functions $f$, $h$
are typically obtained by interpolating discrete boundary data and
hence may vary depending on interpolation method, measurement positions
$\left\{ z_{j}\right\} _{j=1}^{N}$ are usually known with a small
error and pointwise data $\left\{ \omega_{j}\right\} _{j=1}^{N}$
are necessarily subject to a certain noise. Therefore, we assume that
boundary data $f$, $h$ are slightly perturbed by $\delta_{f}\in L^{2}\left(I\right)$,
$\delta_{h}\in L^{2}\left(J\right)$ and internal data $\left\{ \omega_{j}\right\} _{j=1}^{N}$
with measurement positions $\left\{ z_{j}\right\} _{j=1}^{N}$ by
complex vectors $\boldsymbol{\delta_{\omega}}$, $\boldsymbol{\delta_{z}}\in\mathbb{C}^{N}$,
respectively. Varying one of the quantities while the rest are kept fixed,
we are going to estimate separately the linear effects of such perturbations
on the solution $\tilde{g}_{0}=\psi+bg_{0}$ to \eqref{eq:g0_sol_def}, denoting the induced deviations
as $\delta_{\tilde{g}}$.
%\xi_{0}:=P_{+}\left(0\vee\left[g_{0}\left(\mu\right)+\bar{b}\left(\psi-h\right)\right]\right)

\begin{prop}
For $\mu>-1$, $f\in L^{2}\left(I\right)\backslash\left.\mathcal{A}^{\psi,b}\right|_{I}$, $h\in L^{2}\left(J\right)$, 
and small enough data perturbations $\delta_{f}\in L^{2}\left(I\right)$,
$\delta_{h}\in L^{2}\left(J\right)$, $\boldsymbol{\delta_{\omega}}$, $\boldsymbol{\delta_{z}}\in\mathbb{C}^{N}$, 
the following estimates hold:\\
(1) $\quad\left\Vert \delta_{\tilde{g}}\right\Vert _{H^{2}}\leq m_{1}\left(1+\dfrac{m_{1}M^{2}}{m_{0}\left\Vert \xi\right\Vert _{H^{2}}^{2}}\right)\left\Vert \delta_{f}\right\Vert _{L^{2}\left(I\right)}$,\\
(2) $\quad\left\Vert \delta_{\tilde{g}}\right\Vert _{H^{2}}\leq\left[\left(1+m_{1}\left(1+\mu\right)\right)\left(1+\dfrac{m_{1}M^{2}}{m_{0}\left\Vert \xi\right\Vert _{H^{2}}^{2}}\right)-1\right]\left\Vert \delta_{h}\right\Vert _{L^{2}\left(J\right)}$,\\
(3) $\quad\left\Vert \delta_{\tilde{g}}\right\Vert _{H^{2}}\leq\left(1+\left|\mu\right|m_{1}\right)\left(1+\dfrac{m_{1}M^{2}}{m_{0}\left\Vert \xi\right\Vert _{H^{2}}^{2}}\right)\underset{j=1,\dots,N}{\max}\Bigl\Vert\underset{\substack{k=1\\
k\neq j
}
}{\overset{N}{\prod}}\dfrac{z-z_{k}}{z_{j}-z_{k}}\Bigr\Vert_{H^{2}}\left\Vert \boldsymbol{\delta_{\omega}}\right\Vert _{l^{1}}$,\\
(4) $\quad\left\Vert \delta_{\tilde{g}}\right\Vert _{H^{2}}\leq$$\left(1+\dfrac{m_{1}M^{2}}{m_{0}\left\Vert \xi\right\Vert _{H^{2}}^{2}}\right)\left(C_{\mu}^{\left(1\right)}\left\Vert \delta_{b}\right\Vert _{H^{\infty}}+C_{\mu}^{\left(2\right)}\left\Vert \delta_{\psi}\right\Vert _{H^{2}}\right)$,\\
where
\begin{equation}
\xi:=P_{+}\left(0\vee\left(g_{0}+\bar{b}\left(\psi-h\right)\right)\right),\hspace{0.5em}m_{0}:=\min\left\{ \left(1+\mu\right)^{-1},1\right\} ,\hspace{0.5em}m_{1}:=\max\left\{ \left(1+\mu\right)^{-1},1\right\} ,\hspace{1em}\hspace{1em}\hspace{1em}\hspace{1em}\hspace{1em}\;\label{eq:xi0_m0_m1}
\end{equation}\\
$C_{\mu}^{\left(1\right)}:=m_{1}\left(\left\Vert f\vee h\right\Vert _{L^{2}\left(\mathbb{T}\right)}+\left|\mu\right|\left\Vert h-\psi\right\Vert _{L^{2}\left(J\right)}\right), \hspace{1em} C_{\mu}^{\left(2\right)}:=1+\left|\mu\right|m_{1},$ \hspace{1em} and\\
$\left\Vert \delta_{b}\right\Vert _{H^{\infty}}\leq2\underset{j=1,\dots,N}{\max}\left\Vert \left(z-z_{j}\right)^{-1}\right\Vert _{H^{\infty}}\left\Vert \boldsymbol{\delta_{z}}\right\Vert _{l^{1}},$\\
%\textup{\vspace{0.2cm}}\\
$\left\Vert \delta_{\psi}\right\Vert _{H^{2}}\leq2\underset{j=1,\dots,N}{\max}\left|\omega_{j}\right| \underset{j=1,\dots,N}{\max}\Bigl\Vert\underset{\substack{m=1\\
m\neq j
}
}{\overset{N}{\prod}}\left(z-z_{m}\right)\Bigr\Vert_{H^{2}}\times\\
~~~~~~~~~~~~~~~\underset{j=1,\dots,N}{\max}\underset{\substack{k=1\\
k\neq j
}
}{\overset{N}{\sum}}\left|z_{j}-z_{k}\right|^{-1}\left(\underset{j=1,\dots,N}{\min}\underset{\substack{k=1\\
k\neq j
}
}{\overset{N}{\prod}}\left|z_{j}-z_{k}\right|\right)^{-1}\left\Vert \boldsymbol{\delta}_{\boldsymbol{z}}\right\Vert _{l^{1}},$

\end{prop}
\begin{proof}
When the quantities entering the solution (\ref{eq:g0_sol_fin}) vary,
the overall variation of the solution $\delta_{g}$ will consist of
parts entering the solution formula explicitly $\delta_{g_{0}}$ as
well as those coming from the change of the norm of $g_{0}$ on $J$
which leads to readjustment of the Lagrange parameter $\delta_{\mu}$
so that the quantity $M_{0}^{2}\left(\mu\right)=\left\Vert \psi+bg_{0}\left(\mu\right)-h\right\Vert _{L^{2}\left(J\right)}^{2}$
be equal to the same given constraint $M^{2}$. For the sake of brevity,
we are going to use the notations $\xi$, $m_{0}$ and $m_{1}$ introduced
in (\ref{eq:xi0_m0_m1}) to denote certain quantities entering common
estimates. The spectral bounds (\ref{eq:spectrum}) for $\mu>-1$
imply
\[
\sigma\left(1+\mu\phi\right)\geq\min\left\{ 1+\mu,1\right\} ,\hspace{1em}\sigma\left(1+\mu\phi\right)\leq\max\left\{ 1+\mu,1\right\} 
\]

\[
\hspace{1em}\Rightarrow\hspace{1em}\left\Vert \left(1+\mu\phi\right)^{-1}\right\Vert \leq\max\left\{ \left(1+\mu\right)^{-1},1\right\} ,\hspace{1em}\left\Vert \left(1+\mu\phi\right)^{-1}\right\Vert \geq\min\left\{ \left(1+\mu\right)^{-1},1\right\} ,
\]
and so, in particular, 
\[
\text{Re}\left\langle \left(1+\mu\phi\right)^{-1}\xi,\xi\right\rangle _{L^{2}\left(\mathbb{T}\right)}\geq m_{0}\left\Vert \xi\right\Vert _{H^{2}}^{2}.
\]
Then, the connection between $\delta_{M^{2}}$ denoting the change
of $M_{0}^{2}\left(\mu\right)$ and $\delta_{\mu}$ can be established
based on the strict monotonicity (\ref{eq:M_monot}) of $M_{0}\left(\mu\right)$
which allows the following estimate by inversion

\begin{equation}
\delta_{\mu}=\frac{\delta_{M^{2}}}{\left(M_{0}^{2}\left(\mu\right)\right)^{\prime}}=-\frac{\delta_{M^{2}}}{2\text{Re}\left\langle \left(1+\mu\phi\right)^{-1}\xi,\xi\right\rangle _{L^{2}\left(\mathbb{T}\right)}}\hspace{1em}\Rightarrow\hspace{1em}\left|\delta_{\mu}\right|\leq\frac{\left|\delta_{M^{2}}\right|}{2m_{0}\left\Vert \xi\right\Vert _{H^{2}}^{2}}.\label{eq:dmu_dM0}
\end{equation}

Note that the bound in the right-hand side is finite due to the fact
that $\left\Vert \xi\right\Vert _{H^{2}}>0$ which holds unless $M_{0}\left(\mu\right)=0$,
the situation that was initially ruled out by Corollary \ref{cor:M_pos}.
Discussion on \textit{a priori} estimate of $\left\Vert \xi\right\Vert _{H^{2}}$
will be given in Remark \ref{rem:xi0_estim}.

Following this strategy, we embark on consecutive proof of the results
\textit{(1)}-\textit{(4).}

\textit{Result (1):} 

This is the simplest case, the variation of $M_{0}^{2}\left(\mu\right)$
is induced only by change of $g_{0}$. Namely, 
\begin{equation}
\delta_{M^{2}}=2\text{Re}\left\langle \psi+bg_{0}\left(\mu\right)-h,b\delta_{g_{0}}\left(\mu\right)\right\rangle _{L^{2}\left(J\right)},\label{eq:dM0_dg0}
\end{equation}
where 
\begin{equation}
\delta_{g_{0}}=\left(1+\mu\phi\right)^{-1}P_{+}\left(\bar{b}\delta_{f}\vee0\right).\label{eq:dg0_df}
\end{equation}

Application of the Cauchy-Schwarz inequality to (\ref{eq:dM0_dg0})
yields 
\begin{eqnarray*}
\left|\delta_{M^{2}}\right| & \leq & 2M_{0}\left(\mu\right)\left\Vert \left(1+\mu\phi\right)^{-1}\right\Vert \left\Vert P_{+}\left(\bar{b}\delta_{f}\vee0\right)\right\Vert _{L^{2}\left(\mathbb{T}\right)}\leq2M_{0}\left(\mu\right)m_{1}\left\Vert \delta_{f}\right\Vert _{L^{2}\left(I\right)}.
\end{eqnarray*}
and hence, by (\ref{eq:dmu_dM0}),
\[
\left|\delta_{\mu}\right|\leq\frac{m_{1}M_{0}\left(\mu\right)}{m_{0}\left\Vert \xi\right\Vert _{H^{2}}^{2}}\left\Vert \delta_{f}\right\Vert _{L^{2}\left(I\right)}.
\]
Now since $\delta_{\tilde{g}}=b\delta_{g}$, due to (\ref{eq:dg0_dmu}),
we have
\begin{equation}
\delta_{\tilde{g}}=b\delta_{g_{0}}-b\left(1+\mu\phi\right)^{-1}P_{+}\left(0\vee\left(g_{0}+\bar{b}\left(\psi-h\right)\right)\right)\delta_{\mu},\label{eq:dg_tilde}
\end{equation}
from where we deduce the inequality \textit{(1)}. 

\textit{Result (2):}

This is totally analogous to the previous result except for now we
have 

\begin{equation}
\delta_{M^{2}}=2\text{Re}\left\langle \psi+bg_{0}\left(\mu\right)-h,b\delta_{g_{0}}\left(\mu\right)-\delta_{h}\right\rangle _{L^{2}\left(J\right)}\label{eq:dM2_dh}
\end{equation}
with 
\begin{equation}
\delta_{g_{0}}=\left(1+\mu\phi\right)^{-1}P_{+}\left(0\vee\left(1+\mu\right)\bar{b}\delta_{h}\right).\label{eq:dg0_dh}
\end{equation}
Therefore, 
\begin{equation*}
\left|\delta_{M^{2}}\right|\leq2M_{0}\left(\mu\right)\left[1+\left(1+\mu\right)m_{1}\right]\left\Vert \delta_{h}\right\Vert _{L^{2}\left(J\right)}\hspace{1em}\Rightarrow\hspace{1em}\left|\delta_{\mu}\right|\leq\frac{M_{0}\left(\mu\right)\left[1+\left(1+\mu\right)m_{1}\right]}{m_{0}\left\Vert \xi\right\Vert _{H^{2}}^{2}}\left\Vert \delta_{h}\right\Vert _{L^{2}\left(J\right)}.%\label{eq:dmu_dh}
\end{equation*}
Feeding this in the relation (\ref{eq:dg_tilde}), which still holds
in this case, gives

\[
\left\Vert \delta_{\tilde{g}}\right\Vert _{H^{2}}\leq m_{1}\left(1+\mu+\dfrac{\left[1+\left(1+\mu\right)m_{1}\right]M^{2}}{m_{0}\left\Vert \xi\right\Vert _{H^{2}}^{2}}\right)\left\Vert \delta_{h}\right\Vert _{L^{2}\left(J\right)},
\]
that is exactly a rewording of estimate \textit{(2)}.

\textit{Result (3):}

The estimates \textit{(3)} and\textit{ (4)} explore sensitivity of
solution to measurement noise which any experimental data are prone
to. In both cases proofs are similar to those of \textit{(1)-(2)}
with only few new ingredients.

In case of \textit{(3),} a perturbed data vector $\boldsymbol{\delta_{\omega}}\in\mathbb{C}^{N}$
affects the solution $\tilde{g}_{0}$ by means of the induced variation
of $\psi$ that we will denote by $\delta_{\psi}\in H^{2}\left(\mathbb{D}\right)$.\\
If $\psi$ is given by (\ref{eq:interp_expans}), its perturbation
can be estimated as
\begin{equation}
\left\Vert \delta_{\psi}\right\Vert _{H^{2}}\leq\max_{k=1,\dots,N}\left\Vert \mathcal{K}\left(z_{k},\cdot\right)\right\Vert _{H^{2}}\left\Vert S\right\Vert _{1}\left\Vert \boldsymbol{\delta_{\omega}}\right\Vert _{l^{1}},\label{eq:delta_psi1}
\end{equation}
where $\left\Vert \boldsymbol{\delta_{\omega}}\right\Vert _{l^{1}}:=\sum_{k=1}^{N}\left|\left(\delta_{\omega}\right)_{k}\right|$,
$\left\Vert S\right\Vert _{1}:=\underset{j=1,\dots,N}{\max}\sum_{k=1}^{N}\left|S_{kj}\right|$
with $S$ as defined in (\ref{eq:S_def}). However, to get more explicit
result with respect to data positions $\left\{ z_{j}\right\} _{j=1}^{N}$
(which will be more relevant in case \textit{(4)}) avoiding reference
to (\ref{eq:S_def}), we employ polynomial interpolation in Lagrange
form 
\begin{equation}
\psi=\sum_{j=1}^{N}\omega_{j}\prod_{\substack{k=1\\
k\neq j
}
}^{N}\frac{z-z_{k}}{z_{j}-z_{k}},\label{eq:interp_Lagr}
\end{equation}
in which case we have
\begin{equation}
\left\Vert \delta_{\psi}\right\Vert _{H^{2}}\leq\max_{j=1,\dots,N}\Bigl\Vert\prod_{\substack{k=1\\
k\neq j
}
}^{N}\frac{z-z_{k}}{z_{j}-z_{k}}\Bigr\Vert_{H^{2}}\left\Vert \boldsymbol{\delta_{\omega}}\right\Vert _{l^{1}}.\label{eq:delta_psi2}
\end{equation}
Nevertheless, we note that the choice of interpolant (\ref{eq:interp_Lagr})
is not good for practical usage (making way for the barycentric interpolation
formula, see \cite{Berrut}), but done only for the sake of analysis
(again recall that, by Lemma \ref{lem:psi_indep}, the final solution
$\tilde{g}_{0}$ does not depend on a particular choice of the interpolant).
In particular, we see that closedness of interpolation points amplifies
the bound in the right-hand side which corresponds to ill-conditioning
of the matrix $\mathcal{K}\left(z_{k},z_{j}\right)$ for the choice
of interpolant (\ref{eq:interp_expans}).

From this point on, we follow the same steps as in case \textit{(2)}
with (\ref{eq:dM2_dh})-(\ref{eq:dg0_dh}) replaced by

\begin{equation}
\delta_{M^{2}}=2\text{Re}\left\langle \psi+bg_{0}\left(\mu\right)-h,\delta_{\tilde{g}_{0}}\left(\mu\right)\right\rangle _{L^{2}\left(J\right)},\label{eq:dM2_dpsi}
\end{equation}
\begin{equation}
\delta_{\tilde{g}_{0}}=\delta_{\psi}-\mu\left(1+\mu\phi\right)^{-1}P_{+}\left(0\vee\bar{b}\delta_{\psi}\right),\label{eq:dg0_dpsi}
\end{equation}
where the latter variation is estimated from (\ref{eq:g0_sol_fin_alt_compl}).
Then, we have 
\begin{equation}
\left|\delta_{M^{2}}\right|\leq2M_{0}\left(\mu\right)\left(1+\left|\mu\right|m_{1}\right)\left\Vert \delta_{\psi}\right\Vert _{L^{2}\left(J\right)}\hspace{1em}\Rightarrow\hspace{1em}\left|\delta_{\mu}\right|\leq\frac{M_{0}\left(\mu\right)\left(1+\left|\mu\right|m_{1}\right)}{m_{0}\left\Vert \xi\right\Vert _{H^{2}}^{2}}\left\Vert \delta_{\psi}\right\Vert _{L^{2}\left(J\right)}.\label{eq:dmu_dpsi}
\end{equation}
Now 
\begin{equation}
\delta_{\tilde{g}}=\delta_{\tilde{g}_{0}}-b\left(1+\mu\phi\right)^{-1}P_{+}\left(0\vee\left(g_{0}+\bar{b}\left(\psi-h\right)\right)\right)\delta_{\mu},\label{eq:g_tilde2}
\end{equation}
and the resulting estimate \textit{(3)} follows using (\ref{eq:dg0_dpsi})-(\ref{eq:dmu_dpsi})
and recalling (\ref{eq:delta_psi2}).

\textit{Result (4):}

For a perturbation vector of positions $\boldsymbol{\delta_{z}}\in\mathbb{C}^{N}$,
the respective deviation of the interpolant (\ref{eq:interp_Lagr})
is given by
\begin{equation}
\delta_{\psi}=\sum_{j=1}^{N}\omega_{j}\sum_{\substack{k=1\\
k\neq j
}
}^{N}\left(\prod_{\substack{m=1\\
m\neq k,j
}
}^{N}\frac{z-z_{m}}{z_{j}-z_{m}}\right)\frac{\left(z-z_{j}\right)\left(\delta_{z}\right)_{k}-\left(z-z_{k}\right)\left(\delta_{z}\right)_{j}}{\left(z_{j}-z_{k}\right)^{2}},\label{eq:delta_psi3_gen}
\end{equation}
and can be bounded, for instance, as
\begin{equation}
\left\Vert \delta_{\psi}\right\Vert _{H^{2}}\leq2\omega_{0}\max_{j=1,\dots,N}\Bigl\Vert\prod_{\substack{m=1\\
m\neq j
}
}^{N}\left(z-z_{m}\right)\Bigr\Vert_{H^{2}}\frac{\underset{j=1,\dots,N}{\max}\sum_{\substack{k=1\\
k\neq j
}
}^{N}\left|z_{j}-z_{k}\right|^{-1}}{\underset{j=1,\dots,N}{\min}\prod_{\substack{k=1\\
k\neq j
}
}^{N}\left|z_{j}-z_{k}\right|}\left\Vert \boldsymbol{\delta}_{\boldsymbol{z}}\right\Vert _{l^{1}},\label{eq:delta_psi3}
\end{equation}
where $\omega_{0}:=\underset{j=1,\dots,N}{\max}\left|\omega_{j}\right|$.
However, more compact but even rougher bounds can be obtained in terms
of $d_{0}^{-N}$, where $d_{0}:=\underset{\begin{subarray}{c}
j,k=1,\dots,N\\
j\neq k
\end{subarray}}{\min}\left|z_{j}-z_{k}\right|$, which are undesirable for large number of points that are not uniformly
spaced. \\
This case is the most tedious one since now, in addition to $\psi$,
the Blaschke products undergo the variation
\begin{equation}
\delta_{b}=\sum_{j=1}^{N}\left(\prod_{\substack{m=1\\
m\neq j
}
}^{N}\frac{z-z_{m}}{1-\bar{z}_{m}z}\right)\frac{z\left(z-z_{j}\right)\left(\delta_{\bar{z}}\right)_{j}-\left(1-\bar{z}_{j}z\right)\left(\delta_{z}\right)_{j}}{\left(1-\bar{z}_{j}z\right)^{2}},\label{eq:delta_b_gen}
\end{equation}
which can be estimated as
\begin{eqnarray}
\left\Vert \delta_{b}\right\Vert _{H^{\infty}}&\leq&\max_{j=1,\dots,N}\left(\left\Vert \frac{z\left(z-z_{j}\right)}{\left(1-\bar{z}_{j}z\right)^{2}}\right\Vert _{H^{\infty}}+\left\Vert \left(1-\bar{z}_{j}z\right)^{-1}\right\Vert _{H^{\infty}}\right)\left\Vert \boldsymbol{\delta_{z}}\right\Vert _{l^{1}}\nonumber\\
&=&2\underset{j=1,\dots,N}{\max}\left\Vert \left(z-z_{j}\right)^{-1}\right\Vert _{H^{\infty}}\left\Vert \boldsymbol{\delta_{z}}\right\Vert _{l^{1}}.\label{eq:delta_b}
\end{eqnarray}

The rest of the computations is most similar to those in case \textit{(3)}
but slightly more general. Namely, (\ref{eq:dM2_dpsi}) and (\ref{eq:g_tilde2})
hold with
\begin{eqnarray*}
\delta_{\tilde{g}_{0}} & = & \delta_{\psi}+\delta_{b}\left(1+\mu\phi\right)^{-1}\left[P_{+}\left(\bar{b}\left(f\vee h\right)\right)+\mu P_{+}\left(0\vee\bar{b}\left(h-\psi\right)\right)\right]\\
 &  & +b\left(1+\mu\phi\right)^{-1}\left[P_{+}\left(\delta_{\bar{b}}\left(f\vee h\right)\right)+\mu P_{+}\left(0\vee\delta_{\bar{b}}\left(h-\psi\right)\right)\right]-\mu b\left(1+\mu\phi\right)^{-1}P_{+}\left(0\vee\bar{b}\delta_{\psi}\right)
\end{eqnarray*}
estimated from (\ref{eq:g0_sol_fin_alt_compl}). Therefore, 
\[
\left\Vert \delta_{\tilde{g}}\right\Vert _{H^{2}}\leq m_{1}\left(1+\dfrac{m_{1}M^{2}}{m_{0}\left\Vert \xi\right\Vert _{H^{2}}^{2}}\right)\left\Vert \delta_{\tilde{g}_{0}}\right\Vert _{H^{2}},
\]
\[
\left\Vert \delta_{\tilde{g}_{0}}\right\Vert _{H^{2}}\leq\left(1+\left|\mu\right|m_{1}\right)\left\Vert \delta_{\psi}\right\Vert _{H^{2}}+m_{1}\left(\left\Vert f\vee h\right\Vert _{L^{2}\left(\mathbb{T}\right)}+\left|\mu\right|\left\Vert h-\psi\right\Vert _{L^{2}\left(J\right)}\right)\left\Vert \delta_{b}\right\Vert _{H^{\infty}},
\]
and the final estimate \textit{(4)} follows.\end{proof}
\begin{rem}
\label{rem:xi0_estim}The quantity $\xi$ introduced in (\ref{eq:xi0_m0_m1})
enters the results \textit{(1)}-\textit{(4)} and should be bounded
away from zero. This fact, however, follows from Proposition \ref{prop:phi_inject}
and Corollary \ref{cor:M_pos}. Moreover, the norm of $\xi$ can be
\textit{a priori} estimated as 
\begin{equation}
\left\Vert \xi\right\Vert _{H^{2}}\geq\frac{1}{\left|\mu\right|}\left(M-\left\Vert \psi-h+bP_{+}\left(\bar{b}\left(f\vee h\right)\right)\right\Vert _{L^{2}\left(J\right)}\right)\label{eq:xi0_estim}
\end{equation}
by applying the triangle inequality for $L^{2}\left(J\right)$ norm
of the quantity 
\[
\psi+bg_{0}-h=\psi-h+bP_{+}\left(\bar{b}\left(f\vee h\right)\right)+\mu bP_{+}\left(0\vee\left(\bar{b}\left(h-\psi\right)-g_{0}\right)\right),
\]
which is a consequence of (\ref{eq:g0_pre_sol_fin}). Of course, the
estimate (\ref{eq:xi0_estim}) is useful only under assumption 
\begin{equation}
\left\Vert \psi-h+bP_{+}\left(\bar{b}\left(f\vee h\right)\right)\right\Vert _{L^{2}\left(J\right)}<M,\label{eq:stab_est_rem_cond}
\end{equation}
but we do not include it in formulation of the Proposition, since
this inequality can be achieved without imposing any restriction on
given boundary data $f$ and $h$ or increasing the bound $M$: since,
according to Lemma \ref{lem:psi_indep}, choice of $\psi$ does not
affect solution $\tilde{g}_{0}$ whose stability we are investigating,
one can consider another instance of bounded extremal problem, now
formulated for $\psi\in H^{2}\left(\mathbb{D}\right)$ meeting pointwise
constraints (\ref{eq:meas_cond}) and approximating $h-bP_{+}\left(\bar{b}\left(f\vee h\right)\right)\in L^{2}\left(J\right)$
on $J$ sufficiently closely (with precision $M$) with a finite bound
on $I$ without any additional information (meaning that for such a
problem $"I"=J$, $"h"=0$). To be more precise, given arbitrary $\psi_{0}\in H^{2}\left(\mathbb{D}\right)$
satisfying pointwise interpolation conditions (\ref{eq:meas_cond}) (for
instance, one can use (\ref{eq:interp_expans})), we represent $\psi=\psi_{0}+b\Psi$
and thus search for approximant $\Psi\in H^{2}\left(\mathbb{D}\right)$
to $"f"=\bar{b}\left(h-\psi_{0}\right)-P_{+}\left(\bar{b}\left(f\vee h\right)\right)\in L^{2}\left(J\right)$
such that $\left\Vert \Psi\right\Vert _{L^{2}\left(I\right)}=\tilde{M}$
for arbitrary $\tilde{M}\in\left(0,\infty\right)$. We also note that
in the case of reduction to the previously considered problem with
no pointwise data imposed (\cite{Alpay}, \cite{Baratchart-Leblond}),
i.e. when $\psi\equiv0$, $b\equiv1$, one does not have flexibility
of varying the interpolant. However, the stability estimates still
persist in the region of interest (that is, for $-1<\mu<0$) since
the condition (\ref{eq:stab_est_rem_cond}) is fulfilled as long as
$\mu<0$ due to (\ref{eq:g0_sol_fin}) evaluated at $\mu=0$ and (\ref{eq:e-M_monot}).
\end{rem}
$\hspace{1em}$
\begin{rem}
Results \textit{(3)}-\textit{(4)} technically show stability in terms of finite pointwise data sets $\left\{ \omega_{j}\right\} _{j=1}^{N}$, $\left\{ z_{j}\right\} _{j=1}^{N}$ in $l^{1}$ norm,
however, by the equivalence of norms in finite dimensions, the same results, but with
different bounds, also hold for $l^{p}$ norms, for any $p\in\mathbb{N}_{+}$
and $p=\infty$.
\end{rem}

\newpage
\section{\label{sec:num_res}Numerical illustrations and algorithmic aspects}

To illustrate the results of Sections \ref{sec:psi_choice}-\ref{sec:comp_issues}
and estimate practical computational parameters, we perform the following
numerical simulations. First of all, without loss of generality, choose
$J=\left\{e^{i\theta}: \theta\in\left[-\theta_{0},\theta_{0}\right]\right\}$ for some fixed $\theta_{0}\in\left(0,\pi\right)$.
In order to invert the Toeplitz operator in (\ref{eq:g0_sol_fin})
in a computationally efficient way, we realize projection of equation
(\ref{eq:g0_pre_sol_fin}) onto finite dimensional (truncated) Fourier basis $\left\{ z^{k-1}\right\} _{k=1}^{Q}$
for large enough $Q\in\mathbb{N}_{+}$ and look for approximate solution
in the form 
\begin{equation}
g\left(z\right)=\sum_{k=1}^{Q}g_{k}z^{k-1}.\label{eq:num_sol}
\end{equation}

Introducing, for $m$, $k\in\left\{ 1,\dots,Q\right\} $,
\begin{equation}
A_{k,m}:=\begin{cases}
\dfrac{\sin\left(m-k\right)\theta_{0}}{\pi\left(m-k\right)}, & m\neq k,\\
\theta_{0}/\pi, & m=k,
\end{cases}
\ \hspace{2em} A := \left[A_{k,m}\right]_{k,m=1}^Q \, , 
\label{eq:num_T_matr}
\end{equation}
\begin{equation}
s_{k}:=\left\langle \left(\bar{b}\left(f-\psi\right)\vee\left(1+\mu\right)\bar{b}\left(h-\psi\right)\right),e^{i\left(k-1\right)\theta}\right\rangle _{L^{2}\left(0,2\pi\right)}, \hspace{2em} \boldsymbol{s} :=\left[s_{k}\right]_{k=1}^Q, \label{eq:num_S_term}
\end{equation}
the projection equation 
\[
\left\langle \left(1+\mu\phi\right)g-P_{+}\left(\bar{b}\left(f-\psi\right)\vee\left(1+\mu\right)\bar{b}\left(h-\psi\right)\right),z^{k-1}\right\rangle _{L^{2}\left(\mathbb{T}\right)}=0
\]
becomes the vector equation (if we employ $1$ to denote the identity
$Q\times Q$ matrix)
\[
\left(1+\mu A\right)\boldsymbol{g}=\boldsymbol{s},  \hspace{2em} \boldsymbol{g} :=\left[g_{k}\right]_{k=1}^Q \label{eq:num_sys}
\]
with a real symmetric Toeplitz matrix which is computationally cheap
to invert: depending on the algorithm, asymptotic complexity of inversion
may be as low as ${\cal O}\left(Q\log^{2}Q\right)$ (see \cite{Codevico}
and references therein).

Now, in order to numerically demonstrate the monotonicity results (\ref{eq:e-M_monot}) for $e$ and
$M_{0}$ with respect to the parameter $\mu$
and to compare the behavior with that of series expansions (\ref{eq:M_series})-(\ref{eq:e_series}),
we run simulation for the following set of data. We choose $N=5$,
$\theta_{0}=\pi/3$, and 
\[
f\left(\theta\right)=f_{0}\left(\theta\right)+\dfrac{\epsilon}{\exp\left(i\theta\right)-0.4-0.3i}, \hspace{1em}
f_{0}\left(\theta\right):=\exp\left(5i\theta\right)+\exp\left(2i\theta\right)+1\in\mathcal{A}^{\psi,b} 
\] 
(when the parameter $\epsilon\neq0$,
obviously, $f\in L^{2}\left(I\right)$ does not extend inside the
disk as a $H^{2}$ function).
Further, $f_0$ is the restriction of the function $z^{5}+z^{2}+1$ satisfying pointwise interpolation
conditions (\ref{eq:meas_cond}) for points $\left\{ z_{j}\right\} _{j=1}^{5}$
and values $\left\{ \omega_{j}\right\} _{j=1}^{5}$  chosen as
given in Table \ref{tab:meas_data}. We also take $h\in L^{2}\left(J\right)$ as
\[
h\left(\theta\right)=\dfrac{1}{\exp\left(i\theta\right)-0.5i}.
\]
Based on the points $\left\{ z_{j}\right\} _{j=1}^{5}$, we construct the Blaschke
product according to (\ref{eq:Blaschke_prod}) with the choice of constant 
$\phi_{0}=0$ (obviously, final physical results should not depend on a choice of this auxiliary
parameter which is also clear from the solution formula (\ref{eq:g0_sol_fin_alt_compl})).
The interpolant $\psi$ was chosen as (\ref{eq:interp_expans}). Series expansions (\ref{eq:M_series})-(\ref{eq:e_series})
are straightforward to evaluate numerically since $F\left(k\right)$
involves the quantity $\xi_0$ given by (\ref{eq:xi_mu0_def}). The projections $P_{+}$ there
are computed by performing non-negative-power expansions as
(\ref{eq:num_sol}) whereas $\phi^{k}$ is simply iterative multiplication
of the first $Q$ Fourier coefficients of $\xi_0$ by the Toeplitz operator matrix
(\ref{eq:num_T_matr}). Such iterations are extremely cheap to compute once
the matrix $A$ is diagonalized.

Figures \ref{fig:multi_eps_on_I}-\ref{fig:multi_eps_on_J} illustrate
approximation errors on $I$ and discrepancies on $J$ versus the
parameter $\mu$ for different values of $\epsilon$ when the dimension
of the solution space is fixed to $Q=20$. Number of terms in the
series expansions (\ref{eq:M_series})-(\ref{eq:e_series}) was kept
fixed at $S=10$ (such that it is the maximal power of $\mu$ in the
series). It is remarkable that even such a low number of terms gives
bounds which are in very reasonable agreement with those computed
from solution up to relatively close neighborhood of $\mu=-1$. On
Figure \ref{fig:multi_S}, we further investigate change of deviation
of the series expansion from the solution computed numerically (which
is taken as a reference in this case, see the discussion in the next
paragraph) as more terms are taken into account in the expansions.

Figure \ref{fig:multi_Q} shows variation of the results with respect
to truncation of the solution basis while the parameter $\epsilon=0.5$
is kept fixed. Errors are compared to results obtained for $Q=50$
which is taken as reference. We conclude that a choice of $Q$ between
$10$ and $20$ is already sufficiently good for practical purposes.
In particular, we can regard the numerical computation results obtained
for $Q=20$ as those corresponding to faithful solution so to compare
them with what follows from the series expansions (\ref{eq:M_series})-(\ref{eq:e_series}).
Clearly, a choice of $Q<N=5$ does not make sense since, according
to Lemma \ref{lem:psi_indep}, the interpolant $\psi$ can be chosen
as a polynomial which, under such a restriction, will not even be able
to meet all pointwise constraints.

Finally, on Figure \ref{fig:F_k}, we plot auxiliary quantities $F\left(k\right)$
and $kF(k)$ versus $k$ which fundamentally enter the series expansions (\ref{eq:M_series})-(\ref{eq:e_series}).
In such a computation of multiple iterative action of the Toeplitz operator
$\phi$ on a fixed $H^{2}$ function mentioned
above, we used high value of $Q=50$ to prevent possible accumulation
of error steming from the truncation to a finite dimensional basis.
The first quantity $F\left(k\right)$ demonstates the expected decay
to zero, while the second one shows that the decay is not fast enough
to produce a summable series (that is, $F\left(k\right)\neq o\left(1/k\right)$
as $k\rightarrow\infty$) which illustrates the sharpness of Lemma \ref{lem:Fk_sum_fin}
and, on the other hand, is consistent with blow-up of $M_{0}^{2}\left(\mu\right)$
near $\mu=-1$.
\paragraph{Suggested computational algorithm}

Even though Figure \ref{fig:multi_S} shows good accuracy of approximation $e\left(\mu\right)$ and $M^{2}_{0}\left(\mu\right)$
from the series expansions (\ref{eq:M_series})-(\ref{eq:e_series}), it is clear, by nature of such expansions, that the convergence slows down as $\mu$ gets closer to $-1$, and hence, for the genuine values, the number of terms in the series should be increased dramatically. However, as it was mentioned, the quantities $F\left(k\right)$ are very cheap to compute. It remains only to estimate $S$, that is the number of terms in series for the accurate approximation of $e\left(\mu\right)$ and $M^{2}_{0}\left(\mu\right)$, but it suffices to perform such a calibration only once, namely, for the lowest value of $\mu$ in the computational range.
This suggests the following computational strategy:\\
\textbf{1}. Decide on the lowest value of the Lagrange parameter $\mu_{0}$ by checking the approximation rate computed from solving the system (\ref{eq:num_sys}). The quantity $e\left(\mu_{0}\right)$ will then be the best approximation rate on $I$.\\  
\textbf{2}. Determine the number of terms $S$ by comparing the approximation rate with that evaluated from the expansion (\ref{eq:e_series}) for $\mu_{0}$.\\
\textbf{3}. Fix $S$, precompute the values $F\left(k\right)$, $k=1,\dots,S$. Vary the parameter $\mu$ and evaluate the approximation and blow-up rates from the expansions (\ref{eq:M_series})-(\ref{eq:e_series}) in order to find a suitable trade-off.

\vspace{-5cm}

\begin{center}
\begin{table}[p]
\centering{}%
\begin{tabular}{c|c}
$z$ & $\omega$\tabularnewline
\hline 
\hline 
$\phantom{-}0.5+0.4i$ & $0.9852+0.3752i$\tabularnewline
\hline 
$-0.3+0.3i$ & $1.0097-0.1897i$\tabularnewline
\hline 
$\phantom{-}0.2+0.6i$ & $0.7811+0.2362i$\tabularnewline
\hline 
$\phantom{-}0.2-0.5i$ & $0.8328-0.1852i$\tabularnewline
\hline 
$\phantom{-}0.8-0.1i$ & $1.9069-0.3584i$\tabularnewline
\hline 
\end{tabular}\caption{\label{tab:meas_data}Interior pointwise data}
\end{table}

\par\end{center}

\begin{center}
\begin{figure}[p]
%\begin{centering}
\hspace{-5em}
\begin{tabular}{cc}
\includegraphics[trim=0cm 7cm 0cm 7cm, clip=true, width=0.7\textwidth]{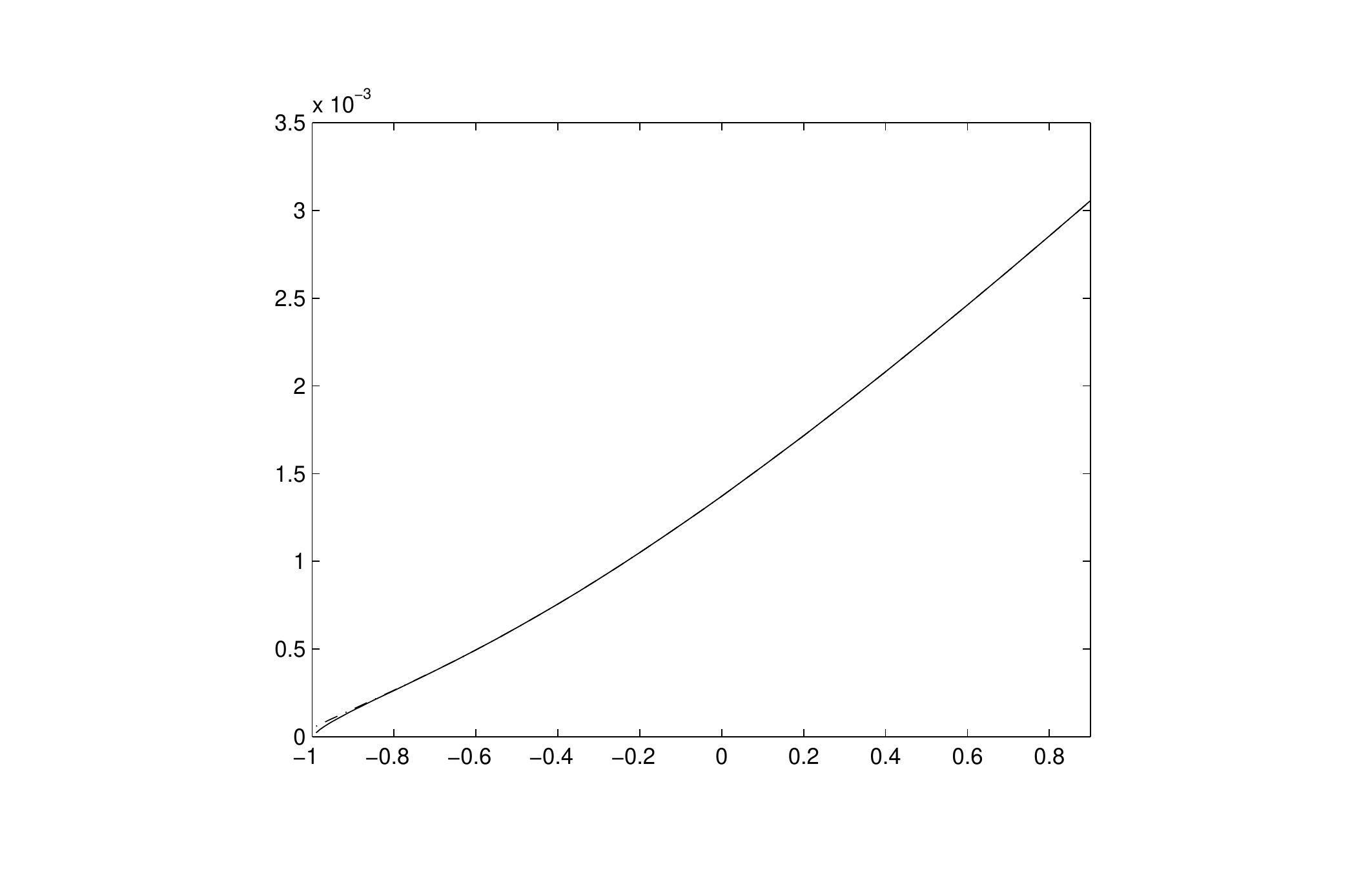} &
\hspace{-10em} \includegraphics[trim=0cm 7cm 0cm 7cm, clip=true, width=0.7\textwidth]{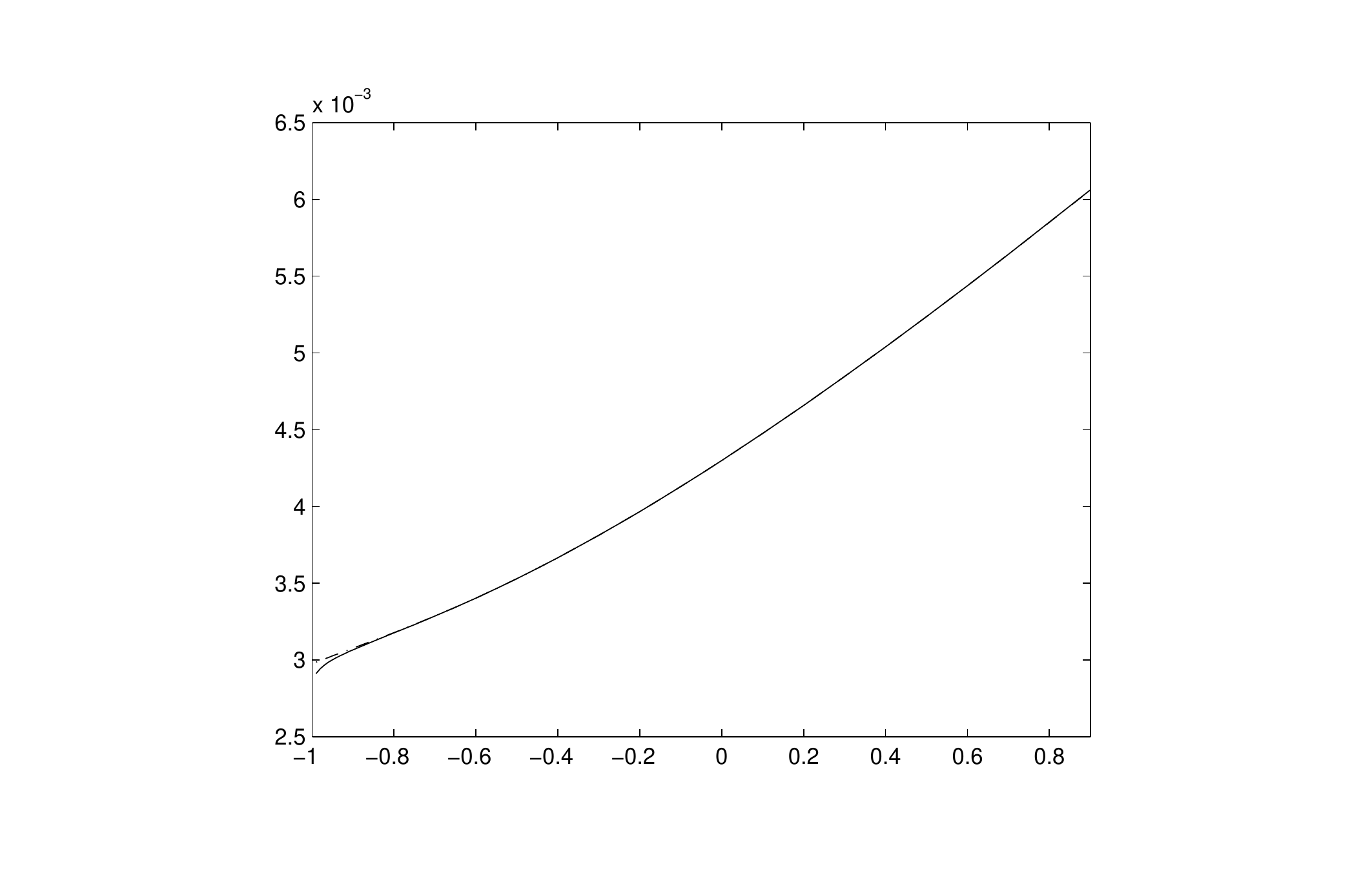}\tabularnewline
\includegraphics[trim=0cm 7cm 0cm 7cm, clip=true, width=0.7\textwidth]{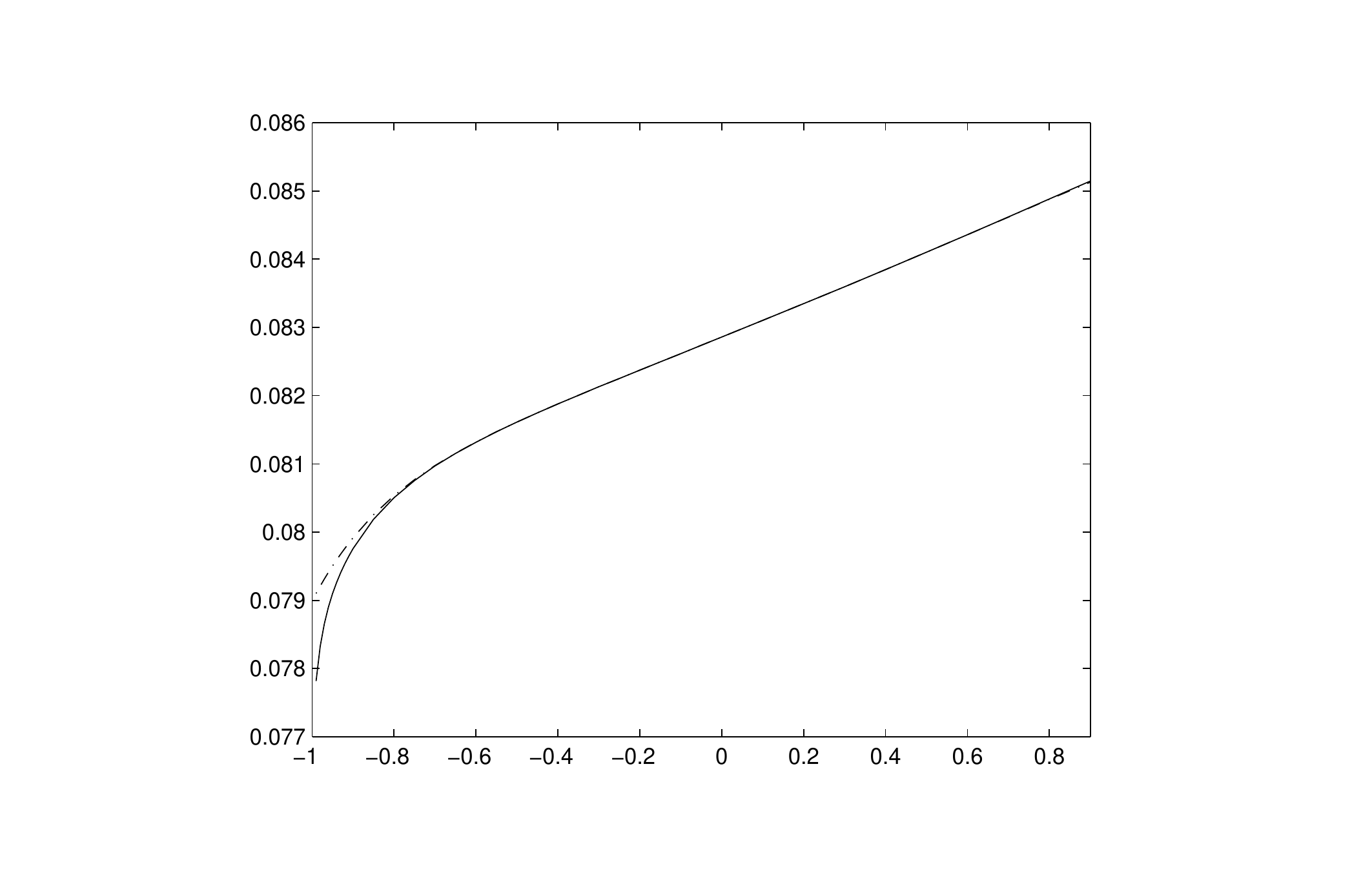} &
\hspace{-10em} \includegraphics[trim=0cm 7cm 0cm 7cm, clip=true, width=0.7\textwidth]{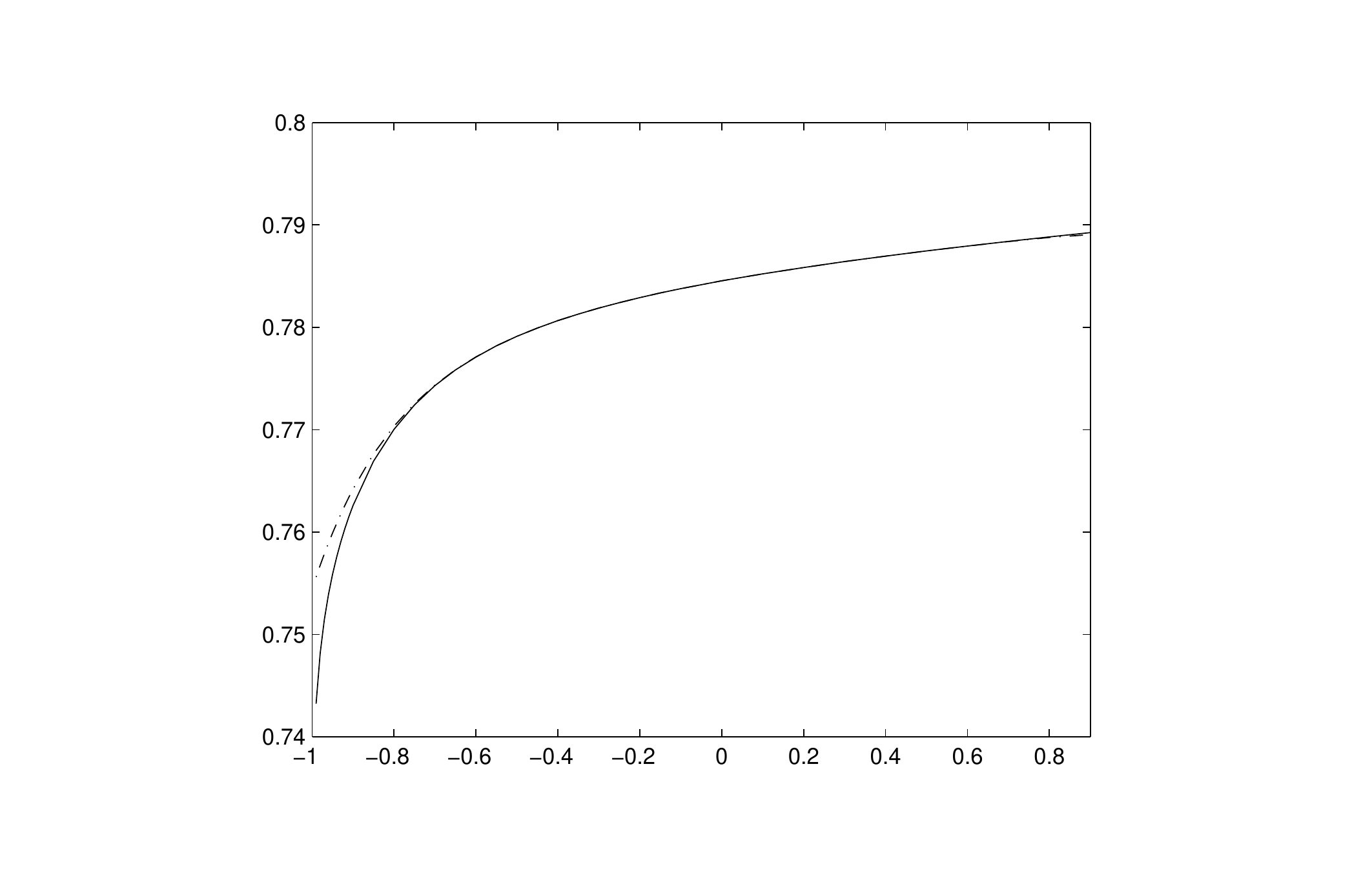}\tabularnewline
\end{tabular}
%\par\end{centering}

\centering{}\caption{\label{fig:multi_eps_on_I}Relative approximation error on $I$: $e\left(\mu\right)/\left\Vert f\right\Vert _{L^{2}\left(I\right)}$
from solution (solid) and series expansion (dash-dot) %\protect \\
for $\epsilon$=0 (top left), $\epsilon=0.1$ (top right), $\epsilon=0.5$
(bottom left), $\epsilon=2$ (bottom right).}
\end{figure}

\par\end{center}

\begin{center}
\begin{figure}[p]
%\begin{centering}
\hspace{-5em}
\begin{tabular}{cc}
\includegraphics[trim=0cm 7cm 0cm 7cm, clip=true, width=0.7\textwidth]{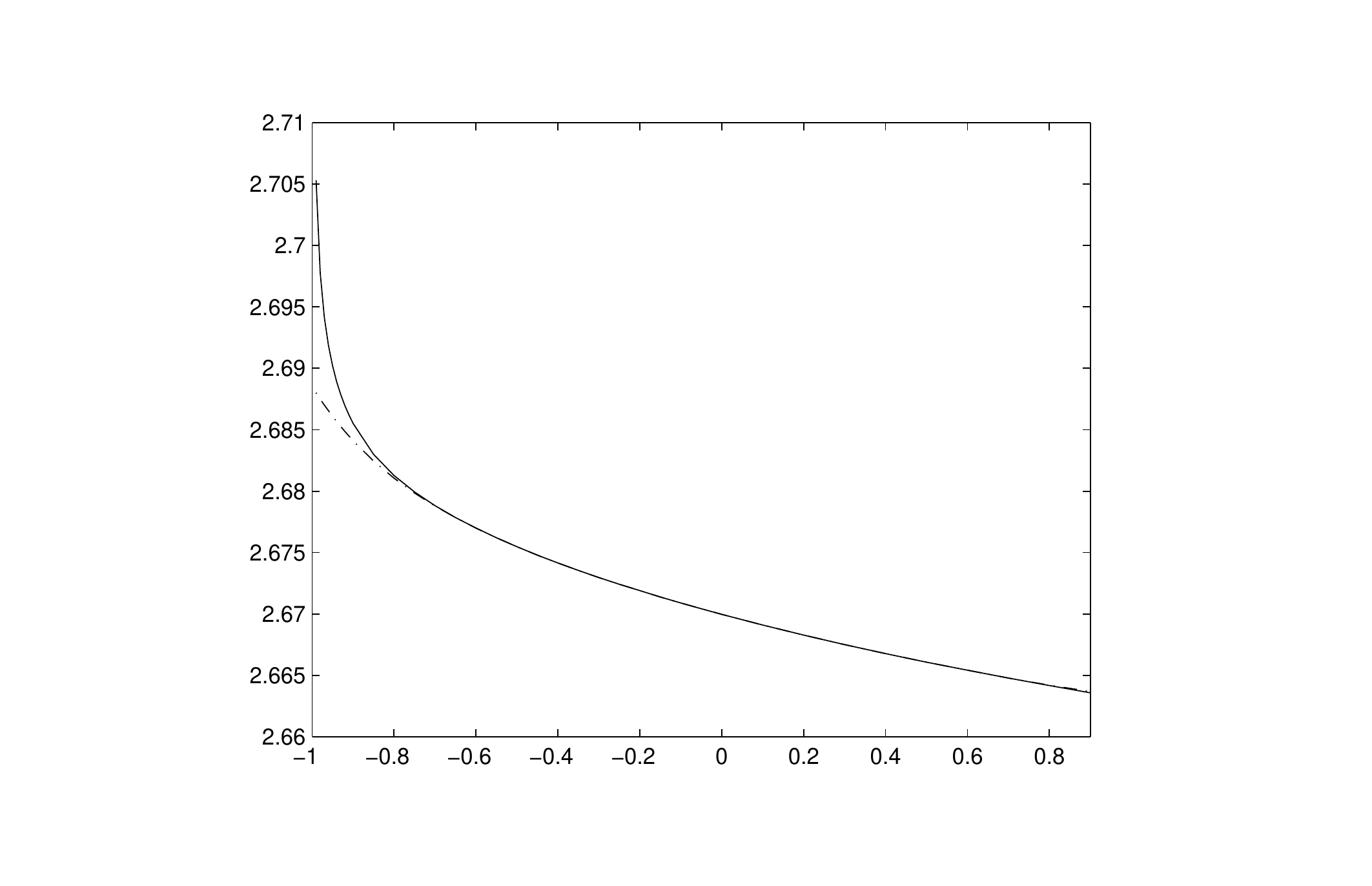} &
\hspace{-10em} \includegraphics[trim=0cm 7cm 0cm 7cm, clip=true, width=0.7\textwidth]{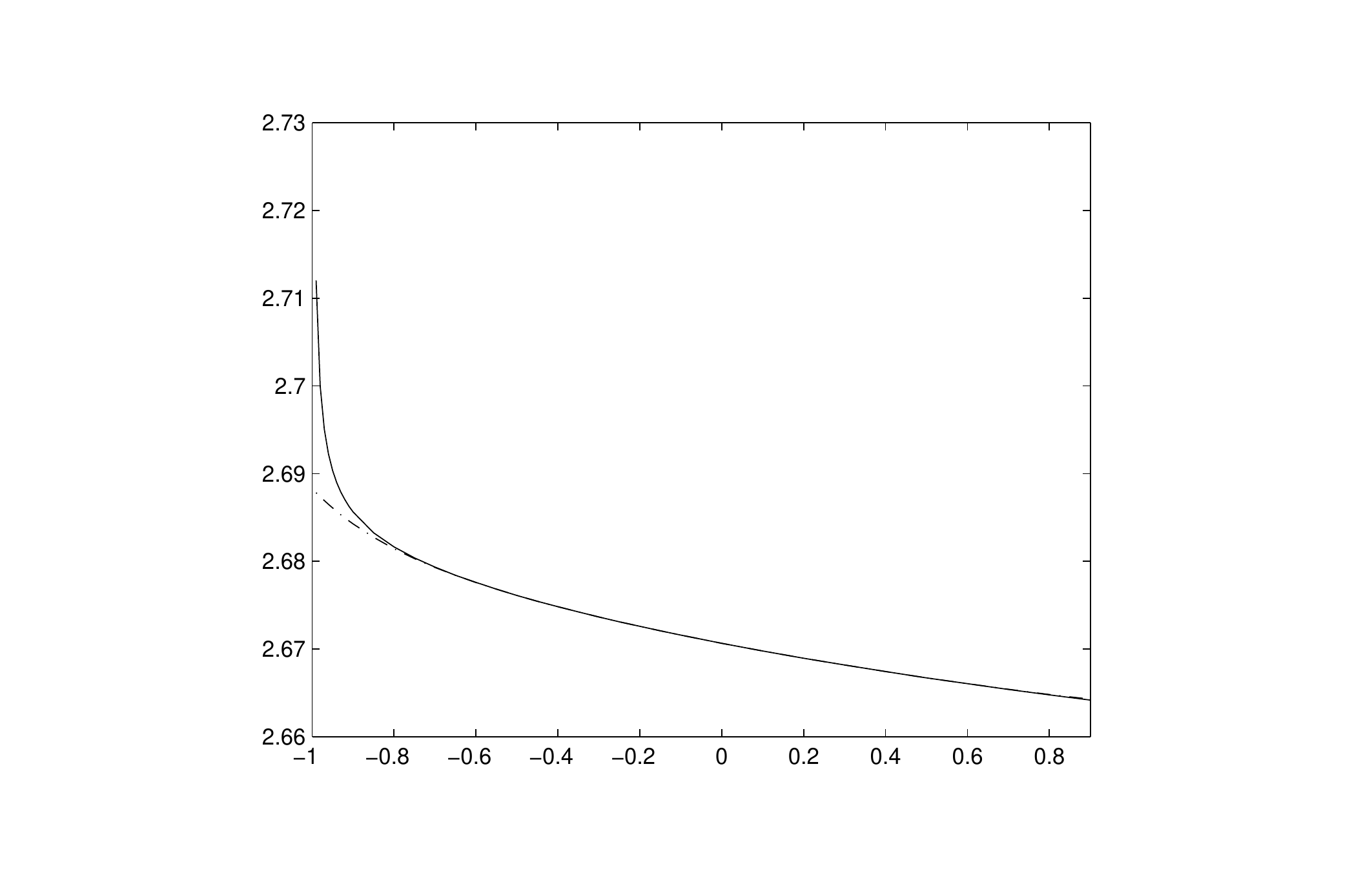}\tabularnewline
\includegraphics[trim=0cm 7cm 0cm 7cm, clip=true, width=0.7\textwidth]{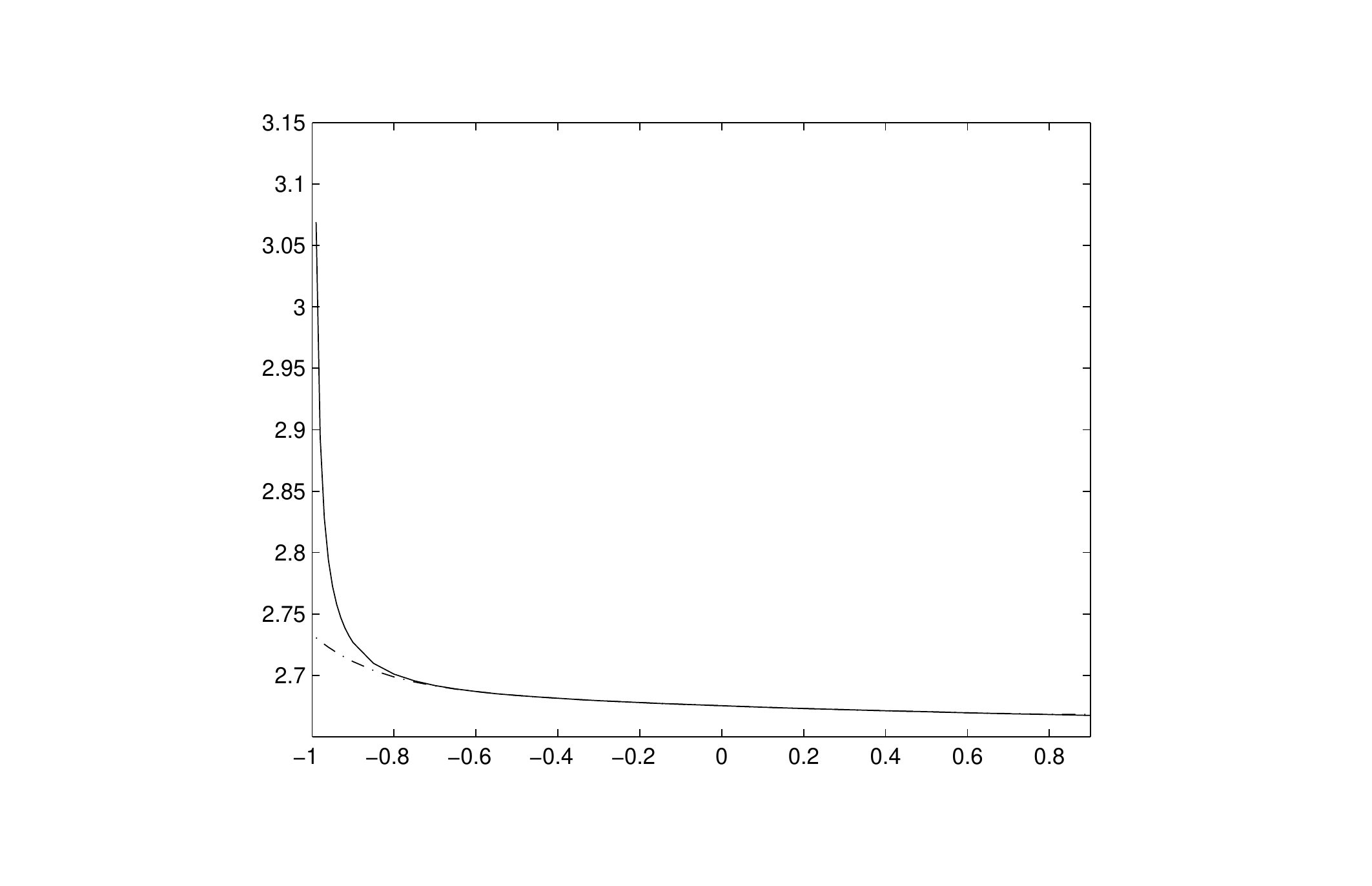} &
\hspace{-10em} \includegraphics[trim=0cm 7cm 0cm 7cm, clip=true, width=0.7\textwidth]{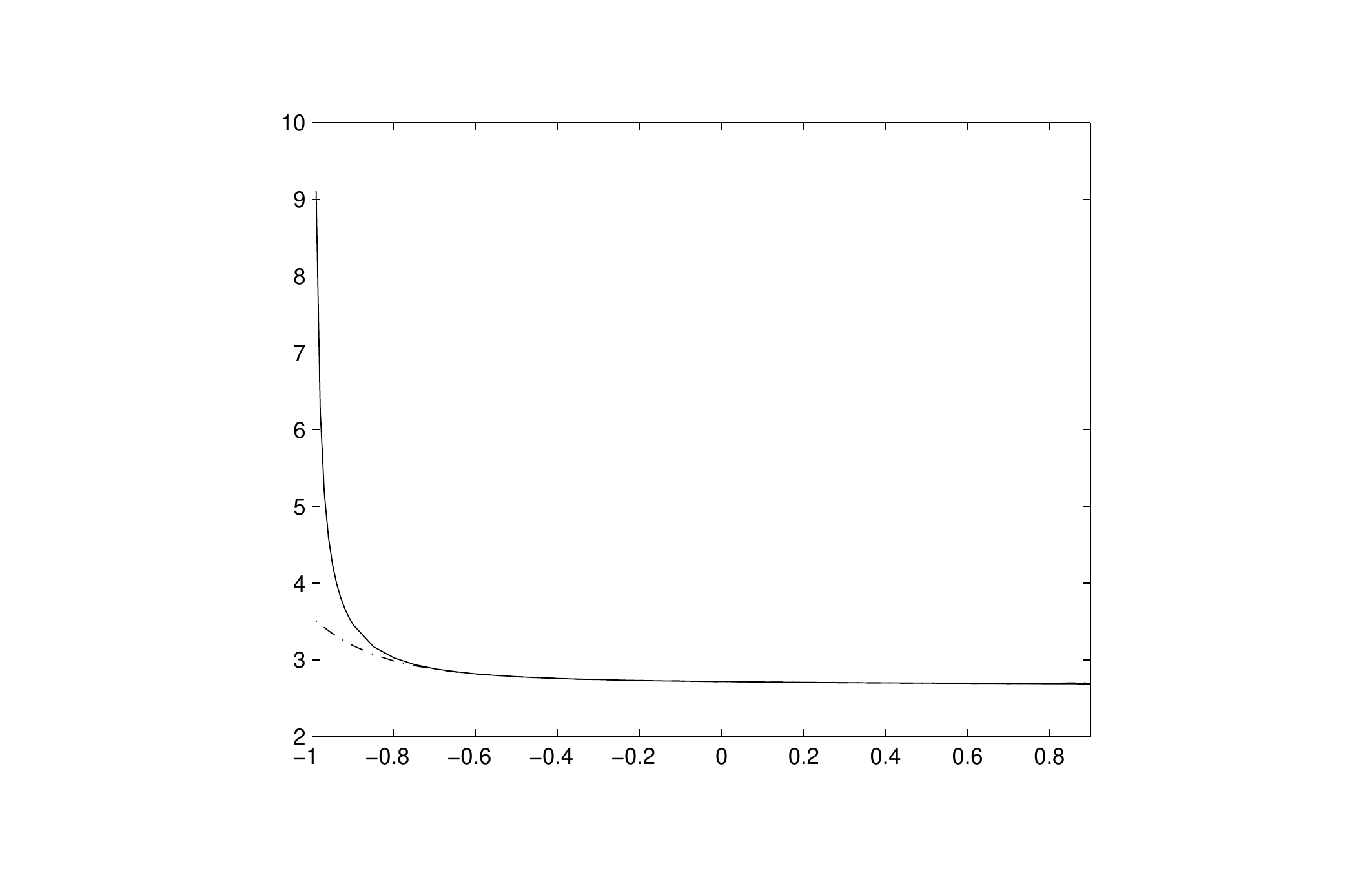}
\end{tabular}
%\par\end{centering}

\centering{}\caption{\label{fig:multi_eps_on_J}Relative discrepancy on $J$: $M_{0}^{2}\left(\mu\right)/\left\Vert h\right\Vert _{L^{2}\left(J\right)}$
from solution (solid) and series expansion (dash-dot) %\protect \\
for $\epsilon$=0 (top left), $\epsilon=0.1$ (top right), $\epsilon=0.5$
(bottom left), $\epsilon=2$ (bottom right).}
\end{figure}

\par\end{center}

\begin{center}
\begin{figure}[p]
%\begin{centering}

% \begin{tabular}{c}
% \includegraphics[width=0.8\textwidth]{on_I_multi_S}\tabularnewline
% \includegraphics[width=0.8\textwidth]{on_J_multi_S}\tabularnewline
% \end{tabular}
\hspace{-5em}
\begin{tabular}{cc}
\includegraphics[trim=0cm 7cm 0cm 7cm, clip=true, width=0.7\textwidth]{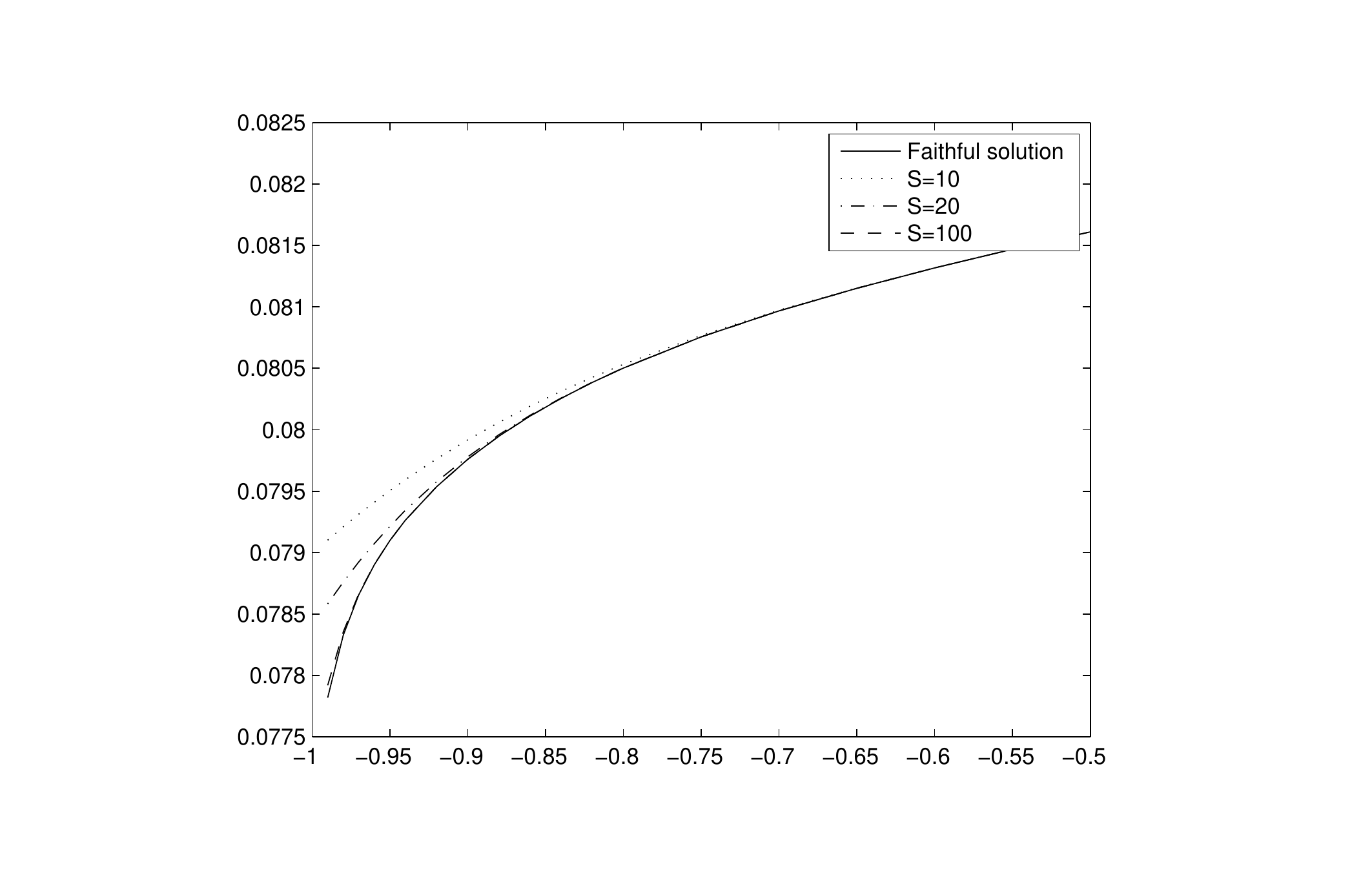} & 
\hspace{-10em}
\includegraphics[trim=0cm 7cm 0cm 7cm, clip=true, width=0.7\textwidth]{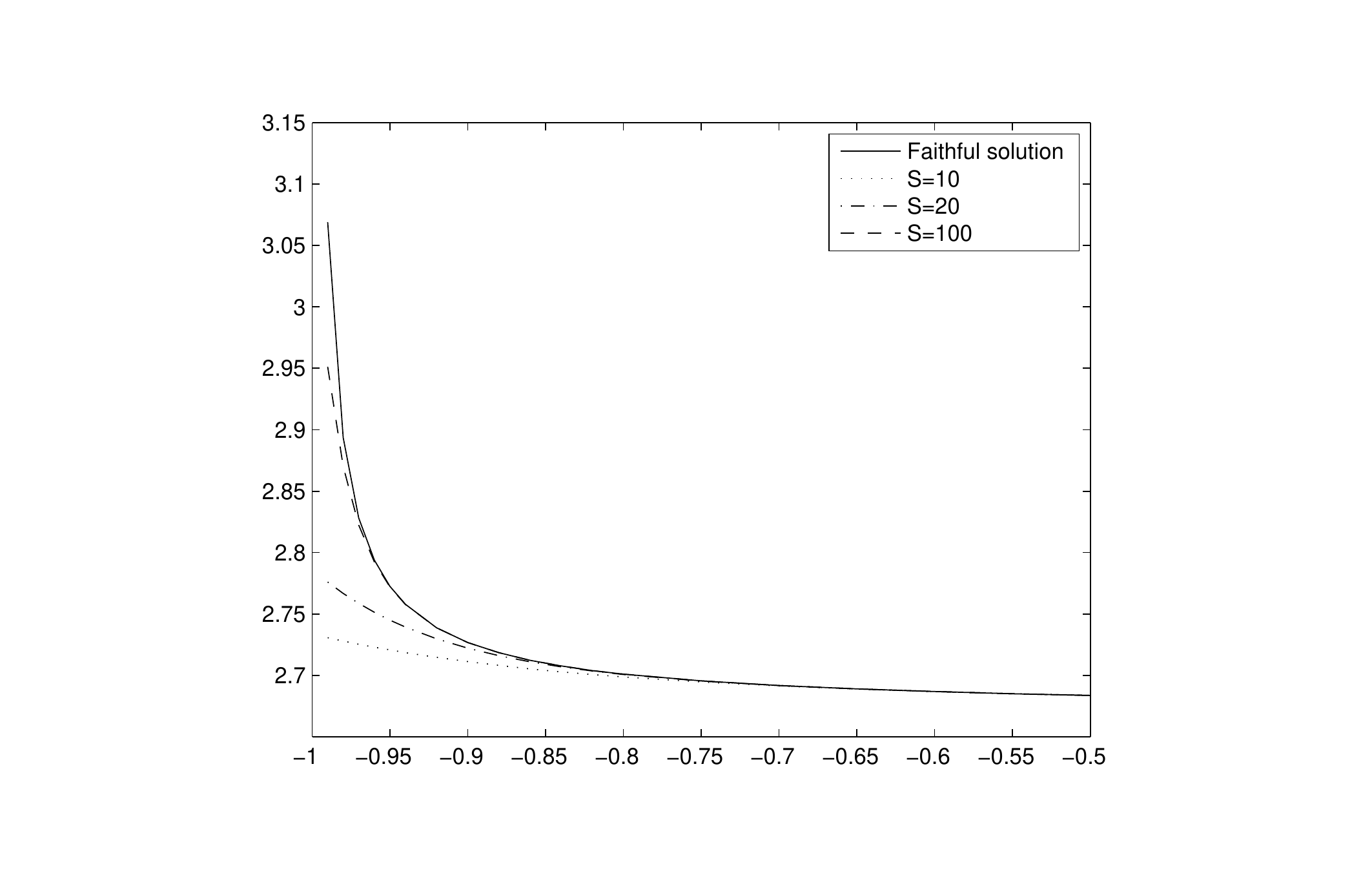}\tabularnewline
\end{tabular}
%\par\end{centering}

\centering{}\caption{\label{fig:multi_S}Relative approximation error on $I$ (left) and %top
relative discrepancy error on $J$ (right).} %bottom
\end{figure}

\par\end{center}

\begin{center}
\begin{figure}[p]
%\begin{centering}
\hspace{-5em}
\begin{tabular}{cc}
\includegraphics[trim=0cm 7cm 0cm 7cm, clip=true, width=0.7\textwidth]{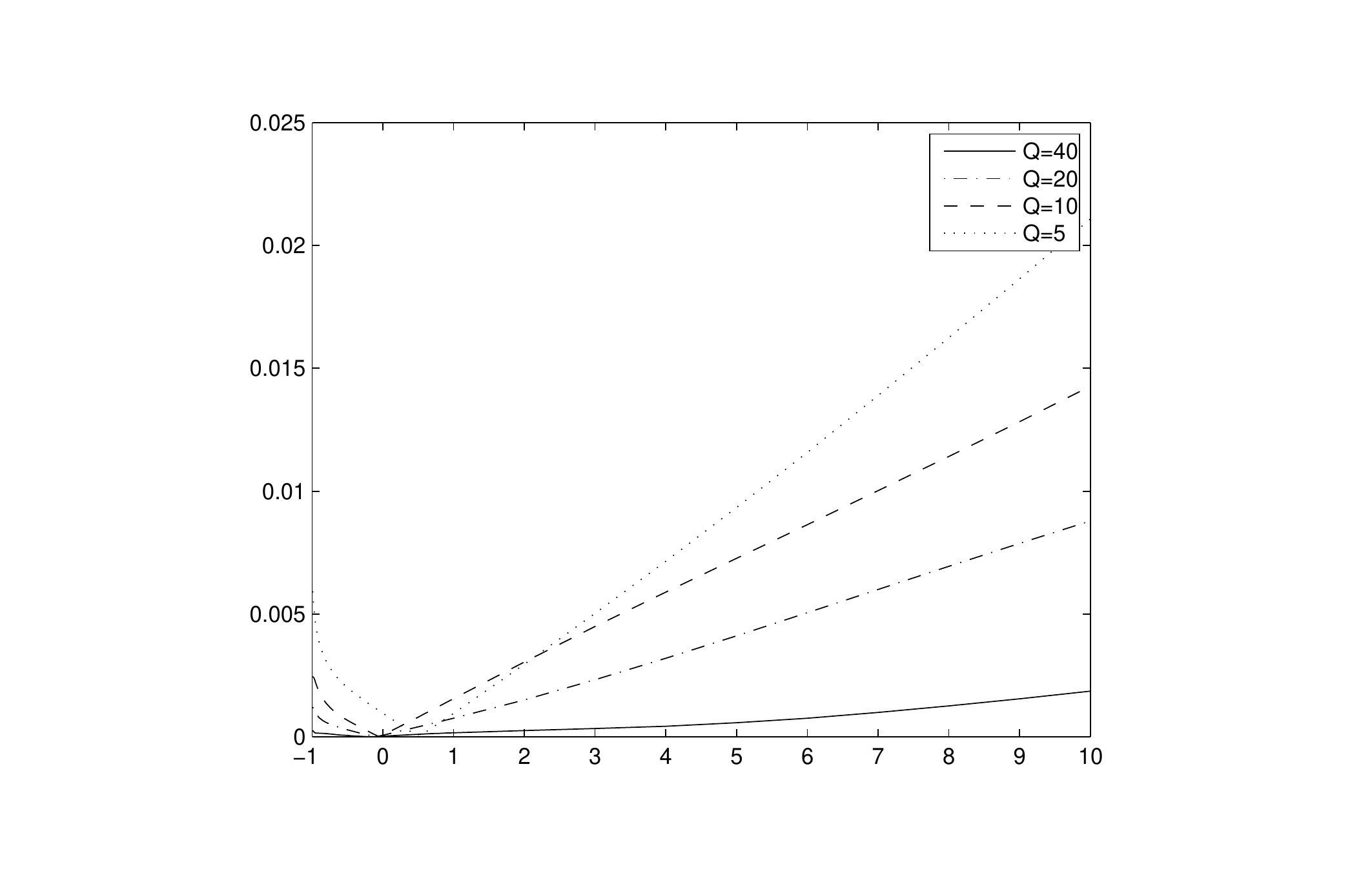} &
\hspace{-10em}
\includegraphics[trim=0cm 7cm 0cm 7cm, clip=true, width=0.7\textwidth]{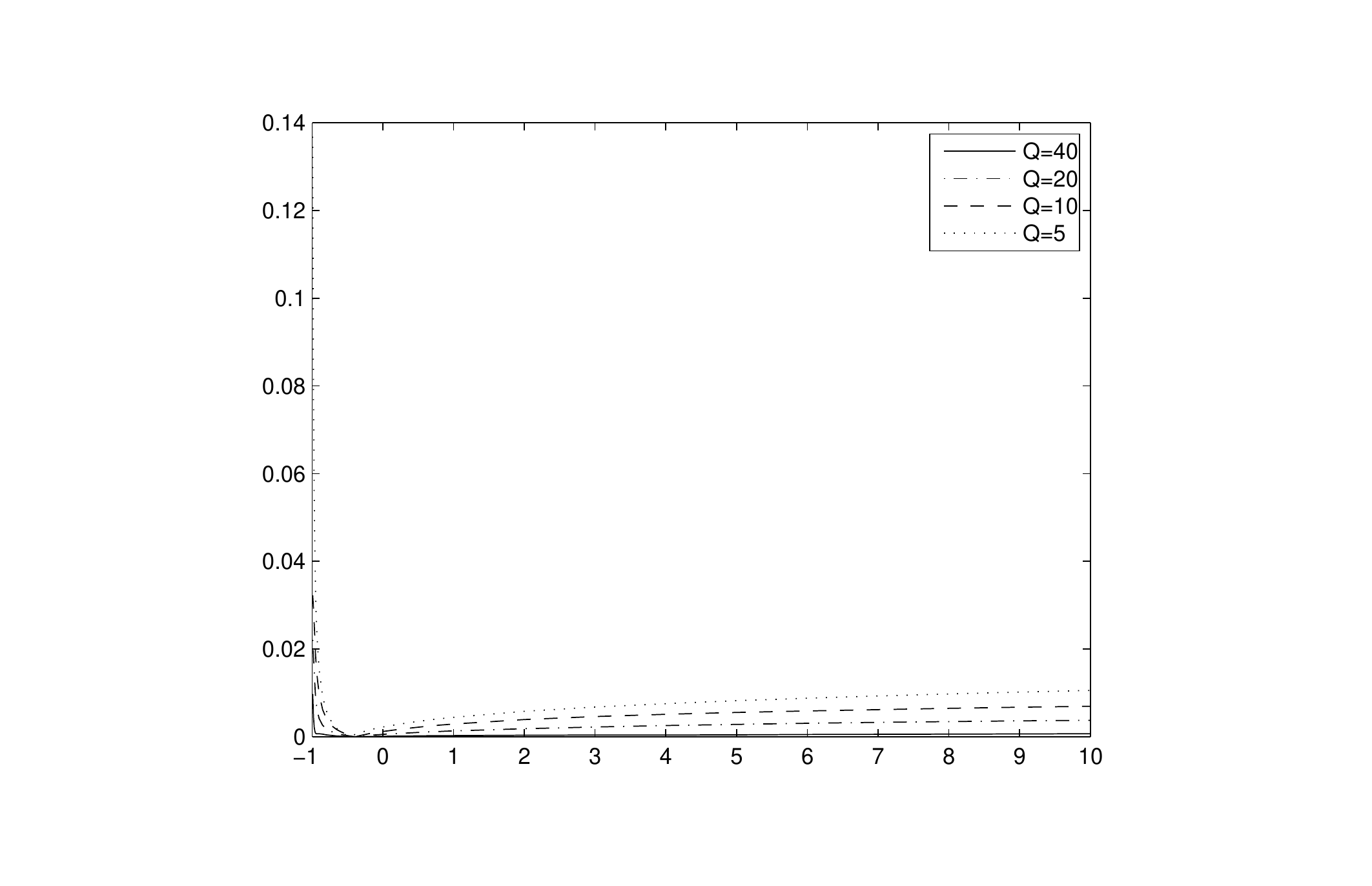}
\end{tabular}
%\par\end{centering}

\centering{}\caption{\label{fig:multi_Q}Errors on $I$ (left) and $J$ (right) compared
to results for $Q=50$.}
\end{figure}

\par\end{center}

\begin{center}
\begin{figure}[p]
%\begin{centering}
\hspace{-5em}
\begin{tabular}{cc}
\includegraphics[trim=0cm 7cm 0cm 7cm, clip=true, width=0.7\textwidth]{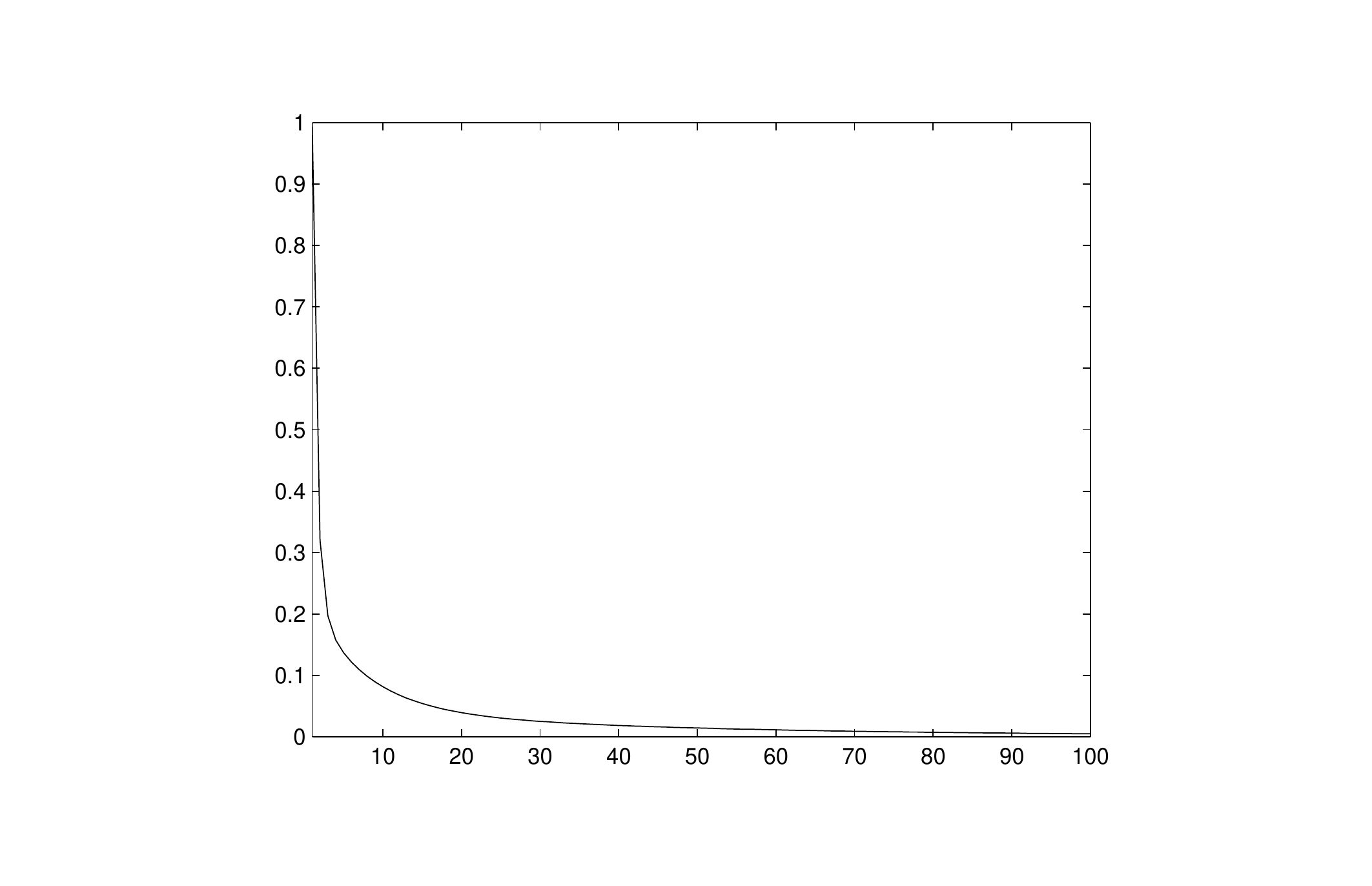} &
\hspace{-10em}
\includegraphics[trim=0cm 7cm 0cm 7cm, clip=true, width=0.7\textwidth]{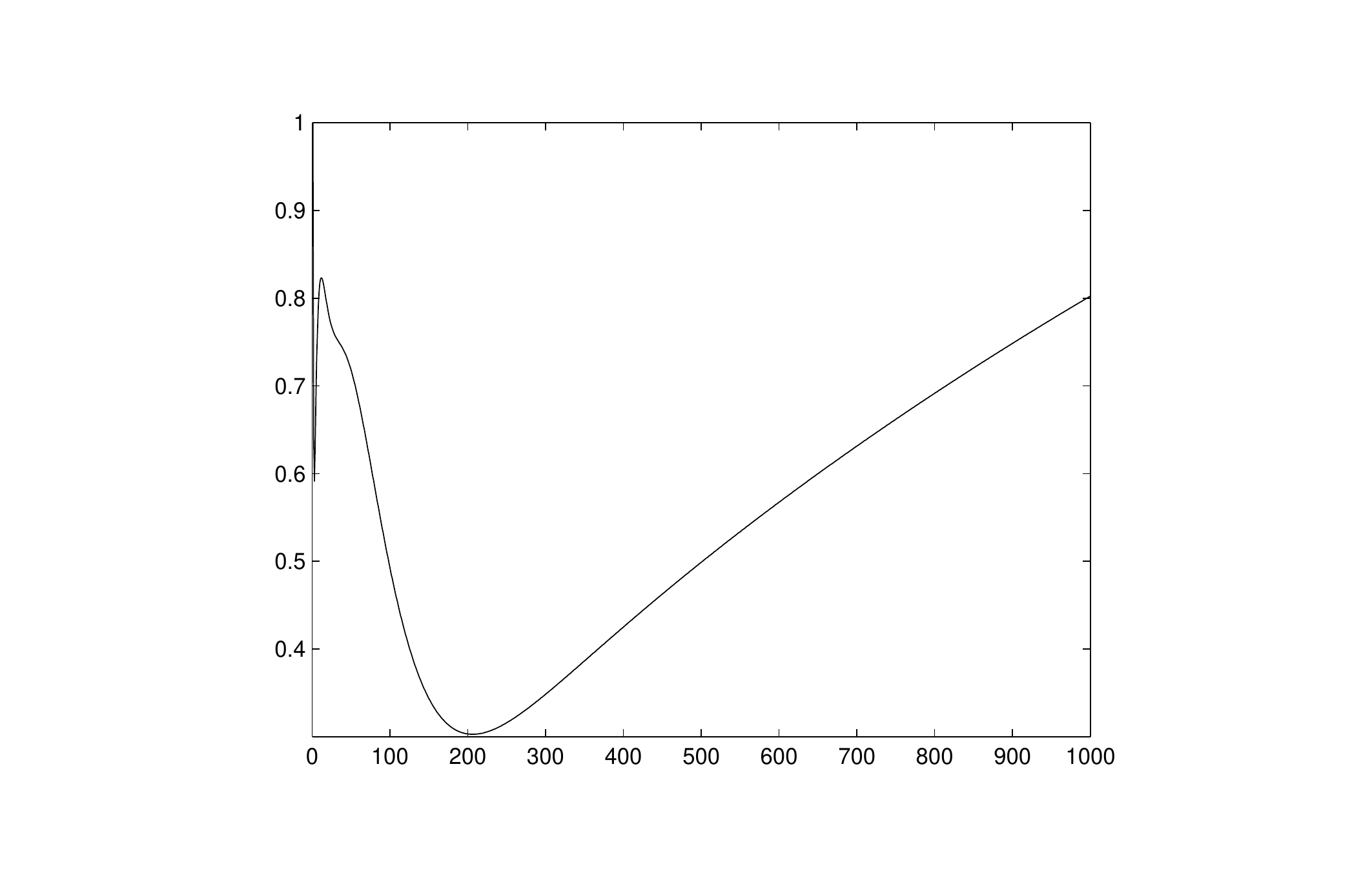}\tabularnewline
\end{tabular}
%\par\end{centering}

\centering{}\caption{\label{fig:F_k}Auxiliary quantities $F\left(k\right)$ and $kF\left(k\right)$
computed with $Q=50$.}
\end{figure}

\par\end{center}

\newpage{}

\FloatBarrier
\section{Conclusions}
\label{sec:conclu}

Motivated by some physical applications, we have introduced and solved a bounded extremal problem that extends the one of best norm-constrained
approximation of a given function on a subset of the circle by the
trace of a $H^{2}$ function \cite{Baratchart-Leblond}
to the case where additional pointwise constraints are imposed inside
the unit disk.

Under such a formulation, there were obtained new results which apply to a problem without pointwise constraints, as a particular case.
Namely, we suggested a method of computing the approximation rate and the
discrepancy growth in terms of a Lagrange parameter.
With an extra argument, the method was used to deduce asymptotic estimates
for quantities governing the approximation quality relying on a different
approach as compared to \cite{Baratchart-Grimm} and possible extensions of those results were discussed.
The new series expansion method was further numerically demonstrated to be very efficient especially beyond the asymptotic regime, thus making
redundant solving multiple instances of the bounded extremal problem
iteratively aiming to find the Lagrange parameter value corresponding
to a suitable trade-off between approximation rate and control of
the blow-up.

We have also observed a connection to a companion problem
which is intrinsic to the presence of internal pointwise data in the original
one. Solution of such a companion problem is computationally cheaper
which may be of big advantage when solving multiple instances of the original bounded extremal
problem. However, there is still room for further investigations in this direction. 

Another gap that was filled with the present work is stability estimates for bounded extremal problems with fixed constraints. Even without presence of pointwise data, the only available result, to our knowledge, is a proof of continuity of the solution with respect to approximated function without additional data ($h=0$; see \cite[Sect. 4.3.4]{Chaabane}).

Since the considered formulation is rather general and has potentially
many physical applications, there are number of issues one may
further want to look into. For example, it would be interesting to
see how the choice of positions of pointwise interpolation data affects the solution. How does increasing the number of points
boost the approximation rate and lower the discrepancy growth significantly?
With the same quantity $N$ of pointwise constraints, are the results better when points are located
closer to the boundary, when they are spread out evenly in the disk
or concentrated in an area or put along a curve? Physically, if positions
of sensors from which the boundary data are obtained are not precise,
does it worth to single out some far out points to be excluded from
interpolation of boundary data functions in order to be treated as
internal constraints? Though some insights into these questions can be obtained
numerically from already developed software, the precise analysis of some issues is expected to be quite involved. 

Another extension of the results may be considered in direction of
generalized analytic functions and annular domains \cite{Fischer-Leblond-Partington,Jaoua-Leblond-Mahjoub}.

%\newpage
\section*{\newpage{}APPENDIX}

\addcontentsline{toc}{section}{APPENDIX}
\begin{thm*}
(Hartman-Wintner)

Let $\xi\in L^{\infty}\left(\mathbb{T}\right)$: $\mathbb{T}\rightarrow\mathbb{R}$
be a symbol defining the Toeplitz operator $T_{\xi}:\, H^{2}\rightarrow H^{2}:$
$F\mapsto T_{\xi}\left(F\right)=P_{+}\left(\xi F\right)$. Then, the
operator spectrum is $\sigma\left(T_{\xi}\right)=\left[\text{ess inf }\xi,\,\text{ess sup }\xi\right]\subset\mathbb{R}$.\end{thm*}
\begin{proof}
We give a proof combining ideas from both \cite[Th. 7.20]{Douglas} and \cite[Th. 4.2.7]{Nikolski} in a way such that it is
short and self-consistent. 

First of all, since $\xi$ is a real-valued function, $T_{\xi}$ is
self-adjoint, and hence $\sigma\left(T_{\xi}\right)\subset\mathbb{R}$.\\
Now, to prove the result, we employ definition of $\sigma\left(T_{\xi}\right)$
as complement of resolvent set, namely, given $\mu\in\mathbb{R}$,
we aim to show that the existence and boundedness of $\left(T_{\xi}-\mu I\right)^{-1}$
on $H^{2}$ (i.e. when $\mu$ is in the resolvent
set) necessarily imply that either $\xi-\mu>0$ or $\xi-\mu<0$ a.e.,
in other words, $\left(\xi-\mu\right)$ must be strictly uniform in
sign a.e. on $\mathbb{D}$.\\
Assume $\mu$ is fixed so that the inverse of $\left(T_{\xi}-\mu I\right)$
exists and bounded on the whole $H^{2}$, in
particular, on constant functions. This means that there is $f\in H^{2}$
such that 
\[
T_{\xi-\mu}f=\left(T_{\xi}-\mu I\right)f=1.
\]
For any $n\in\mathbb{N}_{+}$, denoting $f_{k}$ the coefficients
of Fourier expansion of $f$ on $\mathbb{T}$, let us evaluate
\[
\left\langle T_{\xi-\mu}f,z^{n}f\right\rangle _{L^{2}\left(\mathbb{T}\right)}=\left\langle 1,z^{n}f\right\rangle _{L^{2}\left(\mathbb{T}\right)}=\left\langle z^{n},\bar{f}\right\rangle _{L^{2}\left(\mathbb{T}\right)}=\sum_{k=0}^{\infty}f_{k}\int_{0}^{2\pi}e^{i\left(n+k\right)\theta}d\theta=0.
\]
On the other hand, since $z^{n}f\in H^{2}$,
we have
\[
\left\langle T_{\xi-\mu}f,z^{n}f\right\rangle _{L^{2}\left(\mathbb{T}\right)}=\left\langle \left(\xi-\mu\right)f,z^{n}f\right\rangle _{L^{2}\left(\mathbb{T}\right)}=\int_{\mathbb{T}}\left(\xi-\mu\right)\left|f\right|^{2}\bar{z}^{n}d\sigma,
\]
and thus
\[
\int_{\mathbb{T}}\left(\xi-\mu\right)\left|f\right|^{2}z^{-n}d\sigma=0,\hspace{1em}n\in\mathbb{N}_{+},
\]
which implies that $\left(\xi-\mu\right)\left|f\right|^{2}$ cannot
be an analytic function on $\mathbb{D}$ unless it is constant.\\
However, since $\xi$ and $\mu$ are real-valued, taking conjugation
yields
\[
\int_{\mathbb{T}}\left(\xi-\mu\right)\left|f\right|^{2}z^{n}d\sigma=0,\hspace{1em}n\in\mathbb{N}_{+},
\]
which prohibits $\left(\xi-\mu\right)\left|f\right|^{2}$ being non-analytic
on $\mathbb{D}$ either. Therefore, $\left(\xi-\mu\right)\left|f\right|^{2}=\text{const}$,
and hence $\left(\xi-\mu\right)$ has constant sign a.e. on $\mathbb{D}$
that proves the result.
\end{proof}
$\hspace{1em}$


\newpage
\begin{thebibliography}{References}
\bibitem{Ablowitz} M. Ablowitz, S. Fokas, ``Complex Variables: Introduction
and Applications'', Cambridge University Press, 2003.

\bibitem{Abramowitz-Stegun} M. Abramowitz, I. Stegun, ``Handbook
of Mathematical Functions with Formulas, Graphs, and Mathematical
Tables'', Dover Publications, 1964. 

\bibitem{Alessandrini} G. Alessandrini, ``Examples of instability in inverse boundary-value problems'', Inverse Problems, 13, 887-897, 1997.


\bibitem{Aizenberg} L. Aizenberg, ``Carleman's formulas in complex
analysis'', Kluwer Academic Publishers, 1993.

\bibitem{Alpay} D. Alpay, L. Baratchart, J. Leblond, ``Some extremal
problems linked with identification from partial frequency data'',
Proc. 10 Conf. Analyse Optimisation Systemes, Sophia-Antipolis, Springer-Verlag,
LNCIS 185, 563-573, 1992.

\bibitem{Baratchart-Leblond} L. Baratchart, J. Leblond, ``Hardy
approximation to $L^{p}$ functions on subsets of the circle with
$1\leq p<\infty$'', Constructive Approximation, 14, 41-56, 1998. 

\bibitem{Baratchart-Grimm} L. Baratchart, J. Grimm, J. Leblond, J.
Partington, ``Asymptotic estimates for interpolation and constrained
approximation in $H^{2}$ by diagonalization of Toeplitz operators'',
Integral Equations and Operator Theory, 45, 269-299, 2003.

\bibitem{Baratchart-Leblond-Partington} L. Baratchart, J. Leblond,
J. Partington, ``Hardy approximation to $L^{\infty}$ functions on
subsets of the circle'', Constructive Approximation, 12, 423-436,
1996.

\bibitem{Chaabane} S. Chaabane, ``Etude de quelques probl\`emes inverses'', Th\`ese de Doctorat, Universit\'e de Tunis II - Ecole Natinonale d'Ing\'enieurs de Tunis, 1999.

\bibitem{Rigat-Russ} 
L.~Baratchart, J.~Leblond, S.~Rigat,  E.~Russ.
``Hardy spaces of the conjugate Beltrami equation'',
{J. Funct. Anal.}, 259, 2, 384--427, 2010.

\bibitem{Berrut} J.-P. Berrut, L.-N. Trefethen, ``Barycentric Lagrange
Interpolation'', SIAM Review, 46, 3, 501-517, 2004.

\bibitem{Codevico} G. Codevico, G. Heinig, M. Van Barel, ``A superfast
solver for real symmetric Toeplitz systems using real trigonometric
transformations'', Numerical Linear Algebra With Applications, 12
(8), 699-713, 2005.

\bibitem{Debnath} L. Debnath, P. Mikusinski, ``Introduction to Hilbert
spaces with applications'', Academic Press, 1990.

\bibitem{Douglas} L. G. Douglas, ``Banach algebra techniques in
operator theory'', Academic Press, 1972.

\bibitem{Duren} P. L. Duren, ``Theory of $H^{p}$ spaces'', Academic
Press, 1970.

\bibitem{Duren-Williams} P. L. Duren, D. L. Williams, ``Interpolation
problems in function spaces'', J. Funct. Anal., 9, 75-86, 1972.

\bibitem{Fischer-Leblond-Partington} Y. Fischer, J. Leblond, J. Partington,
E. Sincich, ``Bounded extremal problems in Hardy spaces for the conjugate
Beltrami equation in simply connected domains'', Applied and Computational
Harmonic Analysis, 31, 264-285, 2011.

\bibitem{Fischer} Y. Fischer, ``Approximation dans des classes de fonctions analytiques g\'en\'eralis\'ees et r\`esolution de probl\`emes inverses pour les tokamaks'', Th\`ese de Doctorat, Universit\'e de Nice-Sophia Antipolis, 2011.

\bibitem{Garnett} J. B. Garnett, ``Bounded Analytic Functions'',
Academic Press, 1981.

\bibitem{Goluzin} G. M. Goluzin, V. I. Krylov, ``Generalized
Carleman formula and its application to analytic continuation of functions'',
Matematicheskii Sbornik, 40, 144-149, 1933.

\bibitem{Hoffman} K. Hoffman, ``Banach Spaces of Analytic Functions'',
Prentice Hall, 1962.

\bibitem{Jaoua-Leblond-Mahjoub} M. Jaoua, J. Leblond, M. Mahjoub,
``Robust numerical algorithms based on analytic approximation for
the solution of inverse problems in annular domains'', IMA Journal
of Applied Mathematics, 74, 481-506, 2009. 

\bibitem{Krein-Nudel'man} M. G. Krein, P. Ya. Nudel'man, ``Approximation of $L_{2}\left(\omega_{1},\omega_{2}\right)$ functions by minimum energy transfer functions of linear systems'', Probl. Peredachi Inf., 11:2, 37-60, 1975. 

\bibitem{Lavrentiev} M. Lavrentiev, ``Some improperly
posed problems of mathematical physics'', Springer, 1967.

\bibitem{Lax} P. Lax, ``Functional Analysis'', Wiley-Interscience,
2002.

\bibitem{Martinez} R. A. Martinez-Avendano, P. Rosenthal, ``An Introduction
to Operators on the Hardy-Hilbert Space'', Springer, 2006.

\bibitem{Nehari} Z. Nehari, ``Conformal mapping'', Dover Publications,
2011.

\bibitem{Morse-Feshbach} P. M. Morse, H. Feshbach, ``Methods of
Theoretical Physics. Part I'', McGraw-Hill, 1953.

\bibitem{Nikolski} N. K. Nikolski, ``Operators, Functions and Systems:
An Easy Reading. Volume 1: Hardy, Hankel and Toeplitz'', American
Mathematical Society, 2001.

\bibitem{Olver} F. W. J. Olver, D. W. Lozier, R. F. Boisvert, C.
W. Clark, ``NIST Handbook of Mathematical Functions'', Cambridge
University Press, 2010.

\bibitem{Patil} D. J. Patil, ``Representation of $H^{p}$-functions'',
Bul. Am. Math. Soc., 78 (4), 617-620, 1972.

\bibitem{Rosenblum} M. Rosenblum, ``Self-adjoint Toeplitz operators
and associated orthonormal functions'', Proc. Amer. Math. Soc., 13,
590-595, 1962.

\bibitem{Rosenblum-Rovnyak} M. Rosenblum, J. Rovnyak, ``Hardy Classes
and Operator Theory'', Oxford, 1985.

\bibitem{Rudin} W. Rudin, ``Real and Complex Analysis'', McGraw-Hill,
1982.

\bibitem{Schvedenko} S. V. Schvedenko, ``Interpolation in some Hilbert
spaces of analytic functions'', Matematicheskie Zametki, 15 (1),
101-112, 1974.

\bibitem{Szego} G. Szego, ``Orthogonal Polynomials'', American
Mathematical Soc., 1992.

\bibitem{Varaia} P. Varaia, ``Notes on Optimization'', Van Nostrand
Reinhold, New York, 1972.\end{thebibliography}
\end{document}